\documentclass[12pt,oneside]{CUNY_PhD}

\usepackage{hyperref}
\hypersetup{colorlinks=true, linkcolor=blue, citecolor=blue}
% ...
% ...

\usepackage{amsmath,amsthm,amssymb,bbm,dsfont,color,verbatim,tikz}

\input xy 
\xyoption{all}

\title{Some Results on Large Cardinals and the Continuum Function}
\author{Brent Cody}

\newcommand{\CARD}{{\rm CARD}}
\newcommand{\REG}{{\rm REG}}

\newcommand{\CH}{{\rm CH}}
\newcommand{\GCH}{{\rm GCH}}

\newcommand{\ZFC}{{\rm ZFC}}

\newcommand{\ORD}{\mathop{{\rm ORD}}}

\renewcommand{\P}{{\mathbb P}}
\newcommand{\Q}{{\mathbb Q}}
\newcommand{\R}{{\mathbb R}}

\newcommand{\F}{{\mathbb F}}
\newcommand{\Sacks}{\mathop{\rm Sacks}}
\newcommand{\Add}{\mathop{\rm Add}}

\newcommand{\Ult}{\mathop{\rm Ult}}
\newcommand{\forces}{\Vdash}
\newcommand{\forced}{\Vdash}
\renewcommand{\1}{\mathbbm{1}}

\newcommand{\restrict}{\upharpoonright}
\newcommand{\concat}{\mathbin{{}^\smallfrown}}

\newcommand{\elemsub}{\prec}
\newcommand{\elesub}{\prec}

\newcommand{\supp}{\mathop{\rm supp}}
\newcommand{\Aut}{\mathop{\rm Aut}}
\newcommand{\dom}{\mathop{\rm dom}}
\newcommand{\ran}{\mathop{\rm ran}}

\newcommand{\ot}{\mathop{\rm ot}\nolimits}
\newcommand{\cf}{\mathop{\rm cf}}
\newcommand{\tc}{\mathop{\rm tc}}
\newcommand{\id}{\mathop{\rm id}}

\newcommand{\cp}{\mathop{\rm cp}}

\renewcommand{\and}{\mathop{\&}}

\newcommand{\length}{\mathop{\rm length}}
\newcommand{\Split}{\mathop{\rm Split}}

\newtheorem{theorem}{Theorem}[chapter]
\newtheorem{lemma}[theorem]{Lemma}
\newtheorem{corollary}[theorem]{Corollary}
\newtheorem{claim}[theorem]{Claim}
\newtheorem{sublemma}{Sublemma}[theorem]

\newtheorem*{theorem31}{Theorem 3.1}

\newtheorem*{keylemma}{Key Lemma}

\theoremstyle{definition}
\newtheorem{question}{Question}
\newtheorem{remark}{Remark}
\newtheorem{fact}[theorem]{Fact}
\newtheorem{definition}[theorem]{Definition}

\begin{document}

\frontmatter

\maketitle % same name as the 'book' documentclass command, but gives a
% CUNY titlepage

%optional:
%\makecopyrightpage

\makeapprovalpage{Joel David Hamkins}{Arthur W. Apter}{Gunter Fuchs}%{4th Committee member}

\makeabstractpage{Joel David Hamkins}{

Given a Woodin cardinal $\delta$, I show that if $F$ is any Easton function with $F"\delta\subseteq\delta$ and $\GCH$ holds, then there is a cofinality-preserving forcing extension in which $2^\gamma= F(\gamma)$ for each regular cardinal $\gamma<\delta$, and in which $\delta$ remains Woodin. 

I also present a new example in which forcing a certain behavior of the continuum function on the regular cardinals, while preserving a given large cardinal, requires large cardinal strength beyond that of the original large cardinal under consideration. Specifically, I prove that the existence of a $\lambda$-supercompact cardinal $\kappa$ such that $\GCH$ fails at $\lambda$ is equiconsistent with the existence of a cardinal $\kappa$ that is $\lambda$-supercompact and $\lambda^{++}$-tall.

I generalize a theorem on measurable cardinals due to Levinski, which says that given a measurable cardinal, there is a forcing extension preserving the measurability of $\kappa$ in which $\kappa$ is the least regular cardinal at which $\GCH$ holds. Indeed, I show that Levinski's result can be extended to many other large cardinal contexts. This work paves the way for many additional results, analogous to the results stated above for Woodin cardinals and partially supercompact cardinals.

}

% (In addition to these last four custom CUNY commands, all LaTex
% commands, including those from the LaTeX 'book' class, are available.
% Footnotes and the bibliography will be single spaced, as
% required by CUNY, while the main text will be double spaced.)

\chapter*{Acknowledgements}

It is with great pleasure that I thank Joel David Hamkins, my advisor. This dissertation would certainly not have been possible without his support and guidance. Joel helped to broadly chart the course of my research and also provided help with technical details. Without his encouragement and advice, the work on Woodin cardinals contained herein would have been far less general, and its presentation much less clear. He encouraged me to generalize my work on the Levinski property and Woodin cardinals, which led directly to Theorem \ref{theoremwoodin}. I would also like to thank Arthur Apter for many helpful conversations regarding the topics of this thesis and for suggesting that I consider Woodin cardinals in the first place.

I would like to thank my partner, Anna, for her unending kindness and emotional support, as well as for a relationship I thought impossible before meeting her. I am thankful for my family and friends, for all of their support over the years---especially my parents, Michael and Fonda Cody, who actively helped me to develop the curiosity and self confidence that led to this dissertation. I would also like to thank my sister, Julie, for her friendship and my grandmother, Marion Topping, for being an inspiration throughout my life.

This dissertation is dedicated to Christopher Lee Doyle and Loretta Dolan Cody.

\tableofcontents

\listoffigures

\mainmatter

%===========================================================
%
%
%
%
%===========================================================

\chapter{Introduction and Background}\label{chapterintroduction}

% The following two commands can be used to change the spacing for draft versions of the dissertation. The final version of the dissertation needs to be double spaced, which can be achieved by commenting these commands out.

%\almostdoublespaced
%\singlespaced

Easton proved that the continuum function $\kappa\mapsto 2^\kappa$ on regular cardinals can be forced to behave in any way that is consistent with K\"onig's Theorem ($\kappa<\cf(2^\kappa)$) and monotonicity ($\kappa<\lambda$ implies $2^\kappa\leq 2^\lambda$). In the presence of large cardinals, there are additional restrictions on the possible behaviors of the continuum function on regular cardinals. For example, Scott proved that if $\GCH$ fails at a measurable cardinal $\kappa$, then $\GCH$ fails with normal measure one below $\kappa$. An \emph{Easton function} is a class of the form $F:\REG\to\CARD$ such that (1) $\kappa<\cf(F(\kappa))$ for each $\kappa\in\REG$ and (2) $\kappa<\lambda$ implies $F(\kappa)\leq F(\lambda)$ for $\kappa,\lambda\in\REG$. It seems natural to ask: 
\begin{question}\label{question}
Given a large cardinal $\kappa$, what Easton functions can be forced to equal the continuum function on the regular cardinals, while preserving the large cardinal property of $\kappa$? 
\end{question}
This dissertation fits into the program of answering the above question for the plenitude of large cardinal axioms; see \cite{Menas:ConsistencyResultsConcerningSupercompactness}, \cite{Apter:AnEastonTheoremForLevelByLevel}, \cite{FriedmanHonzik:EastonsTheoremAndLargeCardinals}, and \cite{CodyGitman:EastonsTheoremForRamseyCardinals}.

I now give a brief summary of the literature in this area. In \cite{Silver:GCHForMeasurable}, Silver proved that given a measurable cardinal $\kappa$, one can force the continuum to agree with a very specific Easton function, namely $\gamma\mapsto\gamma^+$, and preserve the measurability of $\kappa$. Kunen and Paris showed in \cite{KunenParis:BooleanExtensionsAndMeasurableCardinals} that under $\GCH$, given a measurable cardinal $\kappa$, an Easton function $F$, and a set $E\subseteq\kappa$ assigned measure zero by some measure on $\kappa$, there is a cofinality-preserving forcing extension in which $\kappa$ remains measurable and for each regular $\gamma\in E$ one has $2^\gamma=F(\gamma)$. Silver then devised a method for forcing a violation of $\GCH$ at a measurable cardinal $\kappa$, which of course requires violating $\GCH$ on a measure one subset of $\kappa$, assuming $\kappa$ is $\kappa^{++}$-supercompact.  Menas combined the methods of Easton and Silver in \cite{Menas:ConsistencyResultsConcerningSupercompactness} to prove that if $F$ is a \emph{locally definable Easton function} (for a definition see \cite[Theorem 18]{Menas:ConsistencyResultsConcerningSupercompactness} or \cite[Definition 3.16]{FriedmanHonzik:EastonsTheoremAndLargeCardinals}), then there is a forcing extension $V[G]$ in which $2^\gamma=F(\gamma)$ for each regular cardinal $\gamma$ and each supercompact cardinal in $V$ remains supercompact in $V[G]$. Levinski proved \cite{Levinski:FiltersAndLargeCardinals} that if $\kappa$ is measurable, then there is a forcing extension in which $\kappa$ remains measurable, $2^\kappa=\kappa^+$, and yet $2^\delta=\delta^{++}$ for each regular cardinal $\delta<\kappa$. Levinski's result stands in contrast to the result of Scott mentioned above. Although not as general as some of the other results discussed herein, Levinski's result illustrates that there are special cases that are interesting and surprising. In \cite{Apter:AnEastonTheoremForLevelByLevel}, Apter proved an Easton theorem for the level-by-level equivalence of strong compactness and supercompactness.

In \cite{FriedmanHonzik:EastonsTheoremAndLargeCardinals}, Friedman and Honsik extend the theorems of Easton and Menas to the large cardinal concepts of strong cardinals and partially hypermeasurable cardinals. They prove that if $F$ is any locally definable Easton function, then there is a forcing extension in which the continuum function agrees with $F$ and in which all strong cardinals are preserved. They also determine precisely what additional assumptions need to be made on an Easton function $F$ and a measurable cardinal $\kappa$ in order to force the continuum function to agree with $F$ and preserve the measurability of $\kappa$. To prove these results, Friedman and Honsik use the tuning fork method, introduced in \cite{FriedmanThompson:PerfectTreesAndElementaryEmbeddings}, which has come to be an extremely versatile tool in the study of large cardinal embeddings and forcing. Indeed, the tuning fork method led to a solution of the number of normal measures problem, see \cite{FriedmanMagidor:TheNumberOfNormalMeasures}.

In Chapter \ref{chapterwoodin}, I will give a full answer to Question \ref{question} for the case of Woodin cardinals. Indeed, I will show that under $\GCH$, if $\delta$ is a Woodin cardinal and $F$ is an Easton function with closure point $\delta$, then there is a cofinality-preserving forcing extension in which $\delta$ remains Woodin and $2^\gamma=F(\gamma)$ for each regular cardinal $\gamma$. In Chapter \ref{chapterthefailure}, I produce a new example of the phenomenon in which, forcing a certain behavior of the continuum function, while preserving a large cardinal property, requires additional large cardinal strength beyond the original large cardinal under consideration. Specifically, I will determine the precise consistency strength of the existence of a $\lambda$-supercompact cardinal $\kappa$ such that $2^\lambda>\lambda^+$. In Chapter \ref{chapterlevinski}, I explore the possibility of generalizing the result of Levinski mentioned above to other large cardinal contexts. That is, I consider the task of forcing nonreflections of $\GCH$ while preserving large cardinals. The work in Chapter \ref{chapterlevinski} lays the groundwork for many more general results along the lines of the main result in Chapter \ref{chapterwoodin}.

%===========================================================
%
%
%
%
%===========================================================

\chapter{Preliminaries}\label{chapterpreliminaries}

\section{Lifting Embeddings}

In what follows, I will be concerned with arguing that various large cardinals are preserved through forcing. The large cardinal properties will be witnessed by elementary embeddings of the form $j:M\to N$ between models of set theory. To show that a given large cardinal property is preserved to a forcing extension, say $V[G]$, one can lift the embedding to $j:M[G]\to N[j(G)]$ and argue that the lifted embedding witnesses the large cardinal property in $V[G]$. In this chapter, I will present some standard lemmas that will be useful for lifting embeddings. For proofs of Lemmas \ref{lemmaground} - \ref{lemmalambdadist}, one may consult \cite{Hamkins:Book}, \cite{Cummings:Handbook}, or \cite{Cummings:AModelInWhichGCH}.

In what follows $N$ and $M$ are always assumed to be transitive models of $\ZFC$. The following two lemmas are useful for building generic objects.

\begin{lemma}\label{lemmaground}
Suppose that $M^\lambda\subseteq M$ in $V$ and there is in $V$ an $M$-generic filter $H\subseteq \Q$ for some forcing $\Q\in M$. Then $M[H]^\lambda\subseteq M[H]$ in $V$.
\end{lemma}

\begin{lemma}\label{lemmachain}
Suppose that $M\subseteq V$ is a model of $\ZFC$, $M^{<\lambda}\subseteq M$ in $V$ and $\P$ is $\lambda$-c.c. If $G\subseteq \P$ is $V$-generic, then $M[G]^{<\lambda}\subseteq M[G]$ in $V[G]$.
\end{lemma}

Suppose $j:M\to N$ is an embedding and $\P\in M$ a forcing notion. In order to lift $j$ to $M[G]$ where $G$ is $M$-generic for $\P$, one typically uses Lemmas \ref{lemmaground} and \ref{lemmachain} to build an $N$-generic filter $H$ for $j(\P)$ satisfying condition (1) in Lemma \ref{lemmaliftingcriterion} below.

\begin{lemma}\label{lemmaliftingcriterion}
Let $j:M\to N$ be an elementary embedding between transitive models of $\ZFC$. Let $\P\in M$ be a notion of forcing, let $G$ be $M$-generic for $\P$ and let $H$ be $N$-generic for $j(\P)$. Then the following are equivalent. 
\begin{enumerate}
\item $j"G\subseteq H$
\item There exists an elementary embedding $j^*:M[G]\to N[H]$, such that $j^*(G)=H$ and $j^*\restrict M =j$.
\end{enumerate}
\end{lemma}
\noindent The embedding $j^*$ in condition (2) above is called a $\emph{lift}$ of $j$.

Suppose $j:V\to M$ is an elementary embedding. A set $S\in V$ is said to \emph{generate $j$ over $V$} if $M$ is of the form 
\begin{align}
M&=\{j(h)(s)\mid h:[A]^{<\omega}\to V, s\in [S]^{<\omega}, h\in V\}.\label{seedrep}
\end{align}
where $A\in V$ and $S\subseteq j(A)$. In this context, the elements of $S$ are called seeds. For more on `seed theory' and its applications, see \cite{Hamkins:Book} and \cite{Hamkins:CanonicalSeedsAndPrickryTrees}. I will often make use of the following lemma which states that the above representation (\ref{seedrep}) of the target model of an elementary embedding remains valid after forcing.

\begin{lemma}\label{lemmaseedpreservation}
If $j:V\to M$ is an elementary embedding generated over $V$ by a set $S\in V$ then any lift of this embedding to a forcing extension $j^*:V[G]\to M[j^*(G)]$ is generated by $S$ over $V[G]$ even if $j^*$ is a class in some further forcing extension $N\supseteq V[G]$.
\end{lemma}

The following standard lemma, which appears in \cite[Chapter 1]{Hamkins:Book}, asserts that embeddings witnessed by extenders are preserved by highly distributive forcing.

\begin{lemma}\label{lemmalambdadist}

If $j:V\to M$ is generated by $S\subseteq j(I)$, and $V[G]$ is obtained by ${\leq}|I|$-distributive forcing, then $j$ lifts uniquely to an embedding $j:V[G]\to M[j(G)]$.

\end{lemma}

\begin{proof}
Suppose $\P$ is $\leq|I|$-distributive forcing and that $G$ is $V$-generic for $\P$. By intersecting at most $|I|$ open dense subsets of $\P$, one may show that $j"G$ generates an $M$-generic filter on $j(\P)$.
\end{proof}

The following standard Lemma due to Easton will be used in many proofs below.

\begin{lemma}\label{lemmaeaston}
Suppose $\P$ is $\kappa^+$-c.c. and $\Q$ is ${\leq}\kappa$-closed. Then $\Q$ remains ${\leq}\kappa$-distributive in $V^\P$.
\end{lemma}

For a proof of Lemma \ref{lemmaeaston} see \cite[Lemma 15.19]{Jech:Book}.

\section{Iterations of Almost Homogeneous Forcing}\label{sectionhomogeneousiteration}

%However, in many cases, assuming that $\dot{\Q}$ can be defined in $V^{\P}$ without referencing the generic for $\P$, the iteration $\P*\dot{\Q}$ will be almost homogeneous.

Recall that a poset $\P$ is \emph{almost homogeneous} if for each pair of conditions, $p,q\in\P$, there is an automorphism $f\in\Aut(\P)$ such that $f(p)$ and $q$ are compatible. Given an Easton support iteration $\P_\beta=\langle (\P_\alpha,\dot{\Q}_\alpha)\mid\alpha<\beta\rangle$, I will isolate a condition such that if each stage of forcing satisfies this condition, then the iteration $\P_\beta$ will be almost homogeneous. This will be used below in the proof of Theorem \ref{theoremwoodin}.

Let me discuss some preliminaries regarding automorphisms of forcing notions. Suppose $f\in\Aut(\P)$ is an automorphism of some forcing notion $\P$. One can recursively extend $f$ to $\P$-names by letting $\dot{x}^f=\{(\dot{y}^f,f(p))\mid (\dot{y},p)\in\dot{x}\}$. Since every automorphism of $\P$ fixes the top element $\1_\P$, it easily follows that check names are invariant under the application of automorphisms of $\P$. In other words, $\check{a}^f=\check{a}$ for each $a\in V$. Furthermore, I will use the fact that automorphisms of posets respect the forcing relation in the following sense. Suppose $f\in\Aut(\P)$ and $p\forces\varphi(\dot{x})$ where $p\in\P$ and $\varphi(\dot{x})$ is a formula in the forcing language for $\P$. Then $f(p)\forces \varphi(\dot{x}^f)$.

Suppose $\P$ is almost homogeneous and $G$ is $V$-generic for $\P$. Let $a_0,\ldots,a_n$ be elements of $V$. Then for each first order formula $\varphi$ with $n$ free variables, it follows by a density argument that $V[G]\models\varphi(a_0,\ldots,a_n)$ if and only if $\1\forces\varphi(\check{a}_0,\ldots,\check{a}_n)$. Hence, if $G$ and $H$ are $V$-generic for almost homogeneous forcing $\P$, then $V[G]$ and $V[H]$ are elementarily equivalent, and in fact they satisfy the same formulas with parameters from $V$.

Suppose $\P$ is almost homogeneous and $\forced_\P$ $``\dot{\Q}$ is almost homogeneous.'' It is not generally the case that the iteration $\P*\dot{\Q}$ is almost homogeneous. For example, suppose $\GCH$ holds and $\P$ is almost homogeneous. Let $\dot{\Q}$ be a $\P$-name and let $a_0,a_1\in\P$ such that $a_0\forces_\P\dot{\Q}=\Add(\omega_1,1)^{V^\P}$ and $a_1\forces_\P\dot{\Q}=\Add(\omega,\omega_2)^{V^\P}$ (such a name $\dot{\Q}$ can be obtained by the Mixing Lemma). In this case, $\forced_\P$ ``$\dot{\Q}$ is almost homogeneous,'' but there are conditions in $\P*\dot{\Q}$ that force incompatible statements, namely $\CH$ and $\lnot\CH$. Hence $\P*\dot{\Q}$ is not almost homogeneous. In this example, the value of $\dot{\Q}$ in $V^\P$ depends on the generic taken for $\P$.

I will now isolate a condition on iterations that will suffice to conclude that $\P*\dot{\Q}$ is almost homogeneous. If $\P$ is an almost homogeneous forcing notion, a $\P$-name $\dot{x}$ is called \emph{symmetric} if for every automorphism $f\in\Aut(\P)$ one has $\forced\dot{x}^f=\dot{x}$. Suppose $\P$ is almost homogeneous and $\forced_\P$ ``$\dot{\Q}$ is almost homogeneous.'' I will show below that if one also assumes that $\dot{\Q}$ is symmetric, then one can conclude that $\P*\dot{\Q}$ is almost homogeneous.

Let me discuss a property of $\P$-names that will be easy to verify in our application, and which will imply that a $\P$-name $\dot{X}$ is symmetric. Suppose that there is a first order formula $\varphi(x_0,\ldots,x_n)$ such that $\forced_\P$ ``$\forall x$ $[x\in\dot{X}$ if and only if $\varphi(x,\check{a}_1,\ldots,\check{a}_n)]$,'' where $a_1,\ldots,a_n\in V$. In this case I will say that \emph{$\varphi$ defines $\dot{X}$ in $V^\P$ from check names}. If there is such a formula $\varphi$, and $f\in\Aut(\P)$ is any automorphism of $\P$, then it follows that $\forced_\P$ ``$\forall x[x\in\dot{X}^f\textrm{ if and only if } \varphi(x,\check{a}_1,\ldots,\check{a}_n)]$.'' Hence, if $\dot{x}$ is a $\P$-name, then
$$ \forced_\P \dot{x}\in\dot{X}\longleftrightarrow \varphi(\dot{x},\check{a}_1,\ldots,\check{a}_n)\longleftrightarrow \dot{x}\in\dot{X}^f.$$
\begin{comment}
\begin{align*}
\forced_\P \dot{x}\in\dot{\Q} &\leftrightarrow \forced_\P \varphi(\dot{x},\check{a}_1,\ldots,\check{a}_n)\\
	&\iff \forced_\P\varphi(\dot{x}^f,\check{a}_1,\ldots,\check{a}_n)\\
	&\iff \forced_\P \dot{x}^f\in\dot{\Q}\\
	&\iff \forced_\P \dot{x}\in\dot{\Q}^f
\end{align*}
\end{comment}
This shows that if some first order formula $\varphi$ defines $\dot{X}$ in $V^{\P}$ from check names (indeed even from symmetric names), then for each $f\in\Aut(\P)$ one has $\forced_\P\dot{X}^f=\dot{X}$. Hence if $\varphi$ defines a given $\P$-name $\dot{X}$ from check names (or even symmetric names), then $\dot{X}$ is symmetric.

Now I will show that if $\P$ is almost homogeneous, $\dot{\Q}$ is forced to be almost homogeneous, and $\dot{\Q}$ is symmetric, then $\P*\dot{\Q}$ is almost homogeneous.

\begin{lemma}\label{lemmaahtwostep}
Suppose $\P$ is almost homogeneous and $\forced_\P$ ``$\dot{\Q}$ is almost homogeneous.'' Suppose further that $\dot{\Q}$ is symmetric; that is, for each automorphism $f\in\Aut(\P)$ one has $\forces_\P\dot{\Q}^f=\dot{\Q}$. Then $\P*\dot{\Q}$ is almost homogeneous.
\end{lemma}

\begin{proof}

Let $(p_0,\dot{p}_1)$ and $(q_0,\dot{q}_1)$ be conditions in $\P*\dot{\Q}$. Since $\P$ is almost homogeneous and $\dot{\Q}$ is symmetric, it follows that there is a $f\in\Aut(\P)$ such that $f(p_0)$ is compatible with $q_0$ and $\forced_\P\dot{\Q}^f=\dot{\Q}$. Let $r_0\in\P$ with $r_0\leq f(p_0)$ and $r_0\leq q_0$. Since $\forced_\P \dot{q}_1\in\dot{\Q}$, it follows that $\forced_\P \dot{q}_1^{f^{-1}}\in\dot{\Q}$ since $\forced_\P\dot{\Q}^{f}=\dot{\Q}$.
Furthermore, since $\forced_\P$ ``$\dot{\Q}$ is almost homogeneous,'' it follows that $\forced_\P$ ``there is a $h\in\dot{\Aut(\Q)}$ such that $h(\dot{p}_1)$ is compatible with $\dot{q}_1^{f^{-1}}$.'' By the fullness principle let $\dot{h}$ be a $\P$-name such that $\forced_\P$ ``$\dot{h}\in\dot{\Aut(\Q)}$ and $\dot{h}(\dot{p}_1)$ is compatible with $\dot{q}_1^{f^{-1}}$.'' Let $\dot{r}_1$ be a $\P$-name with $\forced_\P$ ``$\dot{r}_1\leq\dot{h}(\dot{p}_1)$ and $\dot{r}_1\leq \dot{q}_1^{f^{-1}}$.'' Now, for $(a_0,\dot{a}_1)\in\P*\dot{\Q}$, define $\pi(a_0,\dot{a}_1)=(f(a_0),\dot{h}(\dot{a}_1)^f)$ where $\dot{h}(\dot{a}_1)$ is shorthand notation for a $\P$-name, say $\tau$, with the property $\forced_\P \tau=\dot{h}(\dot{a}_1)$. I will now show that $\pi\in\Aut(\P*\dot{\Q})$ and that $\pi(p_0,\dot{p}_1)$ is compatible with $(q_0,\dot{q}_1)$ via $(r_0,\dot{r}_1^f)$.

First I will demonstrate the compatibility. We have $\pi(p_0,\dot{p}_1)=(f(p_0),$ $\dot{h}(\dot{p}_1)^f)$ where $r_0\leq f(p_0)$ and $r_0\leq q_0$. By applying $f$ to the statement $\forces_\P \textrm{``}\dot{r}_1\leq \dot{h}(\dot{p}_1)$ and $\dot{r}_1\leq\dot{q}_1^{f^{-1}}$,'' one obtains $\forces_\P$ ``$\dot{r}_1^f\leq\dot{h}(\dot{p}_1)^f$ and $\dot{r}_1^f\leq \dot{q}_1$.'' From this it follows that $(r_0,\dot{r}_1)$ extends both $\pi(p_0,\dot{p}_1)$ and $(q_0,\dot{q}_1)$.

Now suppose $(a_0,\dot{a}_1),(b_0,\dot{b}_1)\in\P*\dot{\Q}$ with $(a_0,\dot{a}_1)\leq(b_0,\dot{b}_1)$. Thus $a_0\leq b_0$ and this implies $f(a_0)\leq f(b_0)$. Furthermore, $a_0\forces_\P \dot{a}_1\leq \dot{b}_1$, and since $\forced_\P\dot{h}\in\dot{\Aut(\Q)}$ it follows that $a_0\forces_\P h(\dot{a}_1)\leq h(\dot{b}_1)$. Now applying $f$ to this statement yields $f(a_0)\forces_\P \dot{h}(\dot{a}_1)^f\leq\dot{h}(\dot{b}_1)^f$.

Now let me show that $\pi$ is a bijection. Suppose $(a_0,\dot{a}_1),(b_0,\dot{b}_1)\in\P*\dot{\Q}$ and that $(f(a_0),\dot{h}(\dot{a}_1)^f)=(f(b_0),\dot{h}(\dot{b}_1)^f)$. Since $f$ is an automorphism of $\P$, it easily follows that $a_0=b_0$ and that $\dot{h}(\dot{a}_1)=\dot{h}(\dot{b}_1)$, where the last equality is an equality of $\P$-names. Hence $\forced_\P\dot{h}(\dot{a}_1)=\dot{h}(\dot{b}_1)$ and since $\forced_\P\dot{h}\in\Aut(\dot{\Q})$ one has $\forced_\P \dot{a}_1=\dot{b}_1$. This implies that $(a_0,\dot{a}_1)\leq(b_0,\dot{b}_1)$ and $(b_0,\dot{b}_1)\leq(a_0,\dot{a}_1)$. By replacing the elements of $\P*\dot{\Q}$ with equivalence classes if necessary, we can assume without loss of generality that this implies $(a_0,\dot{a}_1)=(b_0,\dot{b}_1)$. It can easily be verified that $\pi$, as defined above, produces a well-defined map on equivalence classes. Furthermore, it follows that $\pi$ is surjective using the map $(a_0,\dot{a}_1)\mapsto (f^{-1}(a_0),\dot{h}^{-1}(\dot{a}_1)^{f^{-1}})$.
\end{proof}

\begin{lemma}\label{lemmahomogeneousiteration}
Suppose $\P_\beta=\langle (\P_\alpha,\dot{\Q}_\alpha)\mid\alpha<\beta\rangle$ is an Easton support iteration and that for each $\alpha<\beta$ one has $\forces_{\P_\alpha}$ ``$\dot{\Q}_\alpha$ is almost homogeneous.'' Suppose further that for each $\alpha<\beta$, one has that $\dot{\Q}_\alpha$ is a symmetric $\P_\alpha$-name; that is, for each automorphism $f\in\Aut(\P_\alpha)$ one has $\forced_{\P_\alpha}\dot{\Q}_\alpha^f=\dot{\Q}_\alpha$. Then the iteration $\P_\beta$ is almost homogeneous.
\end{lemma}

\begin{proof}
Suppose $p=\langle p_\alpha\mid\alpha<\beta\rangle$ and $q=\langle q_\alpha\mid\alpha<\beta\rangle$ are conditions in $\P$. I will build a condition $r=\langle r_\alpha\mid\alpha<\beta\rangle$ and a sequence of automorphisms $\langle \pi_\xi\mid\xi\leq\beta\rangle$ by induction such that for each $\xi\leq\beta$ one has
\begin{enumerate}
\item[(1)] $\pi_\xi\in\Aut(\P_\xi)$,
\item[(2)] $r\restrict\xi\leq\pi_\xi(p\restrict\xi)$ and $r\restrict\xi\leq q\restrict\xi$,
\end{enumerate}

Since $\P_0$ is trivial, let $\pi_0=h_0=\id_{\P_0}$. For the successor stages $\alpha<\beta$, assume that $\langle\pi_\xi\mid\xi<\alpha+1\rangle$, $\langle h_\xi\mid\xi<\alpha\rangle$, and $\langle r_\xi\mid\xi<\alpha\rangle$ have all been defined. By assumption, $\forced_{\P_\alpha}$ ``$\dot{\Q}_\alpha$ is almost homogeneous,'' and for each $f\in\Aut(\P_\alpha)$ one has $\forced_{\P_\alpha}\dot{\Q}_\alpha^f=\dot{\Q}_\alpha$. Thus $\forced_{\P_\alpha}(q_\alpha)^{\pi_\alpha^{-1}}\in\Q_\alpha$, and it follows from the fullness principle that there is a $\P_\alpha$ name $h_\alpha$ such that $\forced_{\P_\alpha}$ ``$h_\alpha\in\dot{\Aut(\Q_\alpha)}$ and $h_\alpha(p_\alpha)$ is compatible with $(q_\alpha)^{\pi_\alpha^{-1}}$.'' Thus there is a $\P_\alpha$ name $r_\alpha$ such that $\forces_{\P_\alpha}$ ``$r_\alpha\leq h_\alpha(p_\alpha)$ and $r_\alpha\leq (q_\alpha)^{\pi_\alpha^{-1}}$.'' For $a=\langle a_\xi\mid\xi<\alpha+1\rangle\in\P_{\alpha+1}$ define
$\pi_{\alpha+1}(a) = \langle h_\xi(a_\xi)^{\pi_\xi}\mid\xi<\alpha+1\rangle$. It follows as in the proof of Lemma \ref{lemmaahtwostep}, and by the induction hypothesis, that $\pi_{\alpha+1}\in\Aut(\P_{\alpha+1})$ and that the conditions $\pi_{\alpha+1}(p\restrict(\alpha+1))$ and $q\restrict(\alpha+1)$ are compatible via $r\restrict(\alpha+1)$. In the above successor stage procedure, with an eye toward preserving supports, I also dictate that if $\forced_{\P_\alpha}$ ``$(q_\alpha)^{\pi_\alpha^{-1}}$ is the trivial condition,'' then $\forced_{\P_\alpha}$ ``$h_\alpha$ is the identity map on $\dot{\Q}_\alpha$.''

For the limit stages $\eta<\beta$, assume that $\langle \pi_\xi\mid\xi<\eta\rangle$, $\langle h_\xi\mid\xi<\eta\rangle$, and $\langle r_\xi\mid\xi<\eta\rangle$ have all been defined. Given a condition $a=\langle a_\xi\mid\xi<\eta\rangle\in\P_\eta$, define $\pi_\eta(a)=\langle h_\xi(a_\xi)^{\pi_\xi}\mid\xi<\eta\rangle$. It follows that $\supp(\pi_\eta(a))\subseteq\supp(a)\cup\supp(q)$. 

This defines $\langle \pi_\xi\mid\xi<\beta\rangle$, $\langle h_\xi\mid\xi<\beta\rangle$, and $\langle r_\xi\mid\xi<\beta\rangle$. One can define $\pi_\beta$ as above depending on whether $\beta$ is a limit or a successor ordinal. It follows by induction that $\pi_\beta\in\Aut(\P_\beta)$ and that $\pi_\beta(p)$ and $q$ are compatible via $r$. One can check that $\pi_\beta$ provides a bijection as in the proof of Lemma \ref{lemmaahtwostep}.
\end{proof}

%Include Inductive Hypothesis???
\begin{comment}
\begin{enumerate}
\item[$(1)$] $\pi_\xi\in\Aut(\P_\xi)$,
\item[$(2)$] $\forced_{\P_\xi}$ $h_\xi\in\dot{\Aut(\Q_\xi)}$ and $h_\xi(p_\xi)$ is compatible with $(q_\xi)^{\pi_\xi^{-1}}$ via $r_\xi$.
\end{enumerate}
\end{comment}

%===========================================================
%===========================================================
%===========================================================
%===========================================================
%===========================================================
%===========================================================
%===========================================================
%
%
%  WOODIN CARDINALS AND EASTON'S THEOREM
%
%===========================================================
%===========================================================%===========================================================
%===========================================================

\chapter{Woodin Cardinals and Easton's Theorem}\label{chapterwoodin}

The concept of a Woodin cardinal (see Definition \ref{definitionwoodin} below) was originally formulated, by Woodin, for the purpose of establishing the large cardinal consistency strength of The Axiom of Determinacy. Although part of the folklore, there has been little published, to the author's knowledge, concerning the preservation of Woodin cardinals through forcing. For example, it is widely known that if $\delta$ is a Woodin cardinal, then the following forcing notions preserve this: (1) any forcing of size less than $\delta$ (see \cite{HamkinsWoodin:SmallForcing} for this result and more), (2) the canonical forcing to achieve $\GCH$, and (3) any ${<}\delta$-closed forcing (see Lemma \ref{lemmawoodinclosed} below). 

In this Chapter I will apply the methods of \cite{FriedmanHonzik:EastonsTheoremAndLargeCardinals} and \cite{FriedmanThompson:PerfectTreesAndElementaryEmbeddings} to prove the following theorem.

\begin{theorem}\label{theoremwoodin}
Suppose $\GCH$ holds, $F:\REG\to\CARD$ is an Easton function, and $\delta$ is a Woodin cardinal with $F"\delta\subseteq\delta$. Then there is a cofinality-preserving forcing extension in which $\delta$ remains Woodin and $2^\gamma=F(\gamma)$ for each regular cardinal $\gamma$.
\end{theorem}
Notice that in Theorem \ref{theoremwoodin}, there is no requirement stating that $F$ must be locally definable as in the results of \cite{Menas:ConsistencyResultsConcerningSupercompactness} and \cite{FriedmanHonzik:EastonsTheoremAndLargeCardinals}. It is the property $j(A)\cap\gamma = A\cap\gamma$ in the characterization of Woodin cardinals (see Lemma \ref{lemmawoodin}) that allows the removal of this additional requirement on $F$.

Since a straight forward argument shows that ${<}\delta$-closed forcing preserves the Woodinness of $\delta$ (see Lemma \ref{lemmawoodinclosed} below), the bulk of the work in proving Theorem \ref{theoremwoodin} will be to show that the continuum function can be forced to agree with $F$ below $\delta$ while preserving the Woodinness of $\delta$.

Let me remark here that as a corollary to the proof of Theorem \ref{theoremwoodin}, one has the following.

\begin{corollary}
Suppose $C$ is a class of Woodin cardinals and $F$ is an Easton function such that for each $\delta\in C$ one has $F"\delta\subseteq \delta$. Then there is a cofinality-preserving forcing extension in which $\delta$ remains Woodin and $2^\gamma=F(\gamma)$ for each regular cardinal $\gamma$.
\end{corollary}

\section{Preliminaries for the Proof of Theorem \ref{theoremwoodin}}\label{sectionpreliminaries}

I now give some definitions and lemmas that will be used in the proof of Theorem \ref{theoremwoodin}. The following definition is due to Woodin.

\begin{definition}\label{definitionwoodin}
A cardinal $\delta$ is called a \emph{Woodin cardinal} if for every function $f:\delta\to\delta$ there is a $\kappa<\delta$ with $f"\kappa\subseteq\kappa$ and there is a $j:V\to M$ with critical point $\kappa$ such that $V_{j(f)(\kappa)}\subseteq M$.
\end{definition}

As it turns out, Woodin cardinals have another characterization which is more commonly used in practice. We present several versions of this characterization in the next lemma. First let me give a few definitions. Suppose $A\subseteq V_\delta$ and $\kappa<\delta$. One says that $\kappa$ is \emph{$\gamma$-strong for $A$} if there is a $j:V\to M$ with critical point $\kappa$ such that $V_\gamma\subseteq M$, $j(\kappa)>\gamma$, and $j(A)\cap V_\gamma= A\cap V_\gamma$. By definition $\kappa$ is \emph{${<}\delta$-strong for $A$} if $\kappa$ is $\gamma$-strong for $A$ for each $\gamma<\delta$.

\begin{lemma}\label{lemmawoodin}
The following are equivalent.
\begin{enumerate}
\item[$(1)$]\label{l1} $\delta$ is a Woodin cardinal. 
\item[$(2)$] For every $A\subseteq V_\delta$ the following set is stationary. $$\{\kappa<\delta\mid\textrm{$\kappa$ is ${<}\delta$-strong for $A$}\}$$
\item[$(3)$] For every $A\subseteq V_\delta$ there is a $\kappa<\delta$ that is ${<}\delta$-strong for $A$.
\item[$(4)$] For every $A\subseteq \delta$ there is a $\kappa<\delta$ such that for any $\gamma<\delta$ there is a $j:V\to M$ with critical point $\kappa$ such that $\gamma<j(\kappa)$ and $j(A)\cap\gamma =A\cap\gamma$.
\item[$(5)$] For any pair of sets $A_0,A_1\subseteq\delta$ there is a $\kappa<\delta$ such that for any $\gamma<\delta$ there is a $j:V\to M$ with critical point $\kappa$ such that $\gamma<j(\kappa)$, $j(A_0)\cap\gamma= A_0\cap\gamma$, and $j(A_1)\cap\gamma = A_1\cap\gamma$. 

%\item If $A_0,A_1\subseteq \delta$ then there is a $\kappa<\delta$ such that for any $\gamma<\delta$ there is a $j:V\to M$ with critical point $\kappa$ such that $j(A_0)\cap\gamma=A_0\cap\gamma$ and $j(A_1)\cap\gamma=A_1\cap\gamma$.
%\item If $A\subseteq\delta$ then there is a $\kappa<\delta$ such that for any $\gamma<\delta$ there is a $j:V\to M$ with critical point $\kappa$ such that $j(A)\cap\gamma= A\cap\gamma$ and $V_\gamma\subseteq M$.

\end{enumerate}
\end{lemma}

\begin{proof}

%For a proof of $(1) \iff (2)$ one may consult \cite[Lemma 26.14]{Kanamori:Book} or \cite[Lemma 34.2]{Jech:Book}. $(2)\implies (3)$ is trivial. The proof of, (3) $\implies$ (4) $\implies$ (2) is standard and implicitly assumed in \cite{HamkinsWoodin:SmallForcing}. For completeness, I will now provide a proof of (3) $\implies$ (4) $\implies$ (2). 

%I will show (1) $\implies$ (2) $\implies$ (3) $\implies$ (1) and (3) $\implies$ (4) $\implies$ (5) $\implies$ (3).

%Since (2) $\implies$ (3) and (3) $\implies$ (4) are trivial it will suffice to show (1) $\iff$ (2) and (4) $\implies$ (5) $\implies$ (3) $\implies$ (3)...

(1) $\implies$ (2). Suppose $A\subseteq V_\delta$ and let $C\subseteq\delta$ be closed and unbounded. I must show that there is a $\kappa\in C$ that is ${<}\delta$-strong for $A$. Define $f:\delta\to\delta$ as follows. If $\alpha<\delta$ is not ${<}\delta$-strong for $A$, let $f(\alpha)$ be a limit ordinal in $C\setminus\alpha$ such that there is a $\gamma<f(\alpha)$ such that $\alpha$ is not $\gamma$-strong for $A$. Otherwise, let $f(\alpha)=0$. Now let $\kappa<\delta$ be such that $f"\kappa\subseteq\kappa$ and there is a $j:V\to M$ with critical point $\kappa$ such that $V_{j(f)(\kappa)}\subseteq M$. Since $f"\kappa\subseteq\kappa$, it follows that $C\cap\kappa$ is unbounded in $\kappa$ and hence $\kappa\in j(C)$. By elementarity, it will suffice to show that $M\models$ $\kappa$ is ${<}j(\delta)$-strong for $j(A)$. Suppose that $M\models$ $\kappa$ is not ${<}j(\delta)$-strong for $j(A)$. Notice that $j(\delta)=\delta$ since $j(\delta)=\sup j"\delta\leq\delta$. Then, using the definition of $f$, one may let $\gamma<j(f)(\kappa)$ be the least ordinal such that $M\models$ $\kappa$ is not $\gamma$-strong for $j(A)$. Let $X=\{j(h)(s)\mid h:V_\kappa\to V, s\in V_\gamma, h\in V\}$. It follows by the Tarski-Vaught criterion that $X\elemsub M$. Let $\pi:X\to M_0$ be the Mostowski collapse. Since $\ran(j)\subseteq X$, we can define an elementary map $j_0=\pi\circ j:V\to M_0$. Since $\pi(\gamma)=\gamma$, it follows that $\gamma<j(\kappa)$. Furthermore, it follows that the critical point of $j_0$ is $\kappa$, $V_\gamma\subseteq M_0$, and $M_0=\{j_0(h)(s)\mid h:V_\kappa\to V, s\in V_\gamma, h\in V\}$. It follows that $M_0=\Ult(V,E)$ is the ultrapower by a $(\kappa,|V_\gamma|)$-extender $E$, and that $E\in M$. Moreover, $V_\gamma^M=V_\gamma\subseteq\Ult(M,E)^M$. Let $j_E^M:M\to N:=\Ult(M,E)^M$. This implies that $\kappa$ is $\gamma$-strong in $M$. To derive a contradiction it will suffice to show that $j^M_E(j(A))\cap V_\gamma^M=j(A)\cap V_\gamma^M$. It follows that
\begin{align*}
j(A)\cap V_{j(\kappa)}^M &= j(A\cap V_\kappa)\\
	&= j_E(A\cap V_\kappa)\\
	&= j_E^M(A\cap V_\kappa)\\
	&= j_E^M(j(A)\cap V_\kappa)\\
	&= j_E^M(j(A))\cap j_E^M(V_\kappa).
\end{align*}

(2) $\implies$ (3) $\implies$ (4) is trivial.

% (3) $\implies$ (1). Let $f:\delta\to\delta$. Then $f\subseteq V_\delta$ and by (3) there is a $\kappa<\delta$ that is ${<}\delta$-strong for $f$. Let me show that $f"\kappa\subseteq\kappa$. Choose $\alpha<\kappa$ and let $\gamma$ be the next limit ordinal greater than $f(\alpha)$. Let $j:V\to M$ have critical point $\kappa$ such that $\gamma<j(\kappa)$, $V_{\gamma}\subseteq M$, and $j(f)\cap V_\gamma= f\cap V_\gamma$. Then $j(f)(\alpha)=f(\alpha)<j(\kappa)$. Since $\alpha$ is less than the critical point of $j$, it follows that $j(f(\alpha))<j(\kappa)$. By elementarity this implies $f(\alpha)<\kappa$. Now I will show that there is an embedding with critical point $\kappa$ as in the definition of Woodin cardinal, for the function $f$. Choose $\lambda>f(\kappa)$ sufficiently large and let $j:V\to M$ have critical point $\kappa$ such that $\lambda<j(\kappa)$, $V_\lambda\subseteq M$, and $j(f)\cap V_\lambda=f\cap V_\lambda$. Since $\lambda$ was choosen large enough, it follows that $j(f)(\kappa)=f(\kappa)$ and hence $j(f)(\kappa)<\lambda$. Thus $V_{j(f)(\kappa)}\subseteq M$.

(4) $\implies$ (5).
Suppose $A_0$ and $A_1$ are subsets of $\delta$ and let 
$$A:=\{\langle 0, \alpha\rangle \mid \alpha\in A_0\}\cup\{\langle 1,\alpha\rangle \mid\alpha\in A_1\}$$
where $\langle i,\alpha\rangle$ denotes the ordinal given by the G\"{o}del pairing function $\langle \cdot,\cdot\rangle$. Now, let $\kappa<\delta$ be as in $(4)$ for the $A$ specified above. I will show that $\kappa$ also satisfies (5) for this $A_0$ and $A_1$. Fix a cardinal $\gamma<\delta$. Then by (4), there is a $j:V\to M$ with critical point $\kappa$ such that $j(A)\cap\gamma=A\cap\gamma$ and $j(\kappa)>\gamma$. Since $\gamma$ is closed under G\"{o}del pairing and since G\"{o}del pairing is absolute to $M$, it follows that $j(A_0)\cap\gamma=A_0\cap\gamma$ and $j(A_1)\cap\gamma=A_1\cap\gamma$. 
Thus (4) $\implies$ (5).

(5) $\implies$ (1). Suppose $f:\delta\to\delta$. Let $R\subseteq\delta\times\delta$ be a relation such that $\langle \delta,R\rangle\cong \langle V_\delta,\in\rangle$ with the property that for each $\beta$-fixed point $\eta<\delta$ one has $\langle \eta, R\restrict\eta\rangle\cong\langle V_\eta,\in\rangle$. It follows by the Mostowski Collapse Lemma that the isomorphism, say $\pi:\langle\delta,R\rangle\to\langle V_\delta,\in\rangle$, is unique and indeed, $\pi\restrict\eta:\langle \eta,R\restrict\eta\rangle\to\langle V_\eta,\in\rangle$ is an isomorphism for each $\beth$-fixed point $\eta<\delta$. Let $A_0:=\{\langle\alpha,\beta\rangle\mid(\alpha,\beta)\in R\}$ be the subset of $\delta$ that codes $R$ via G\"{o}del pairing. Let $A_1:=\pi^{-1}"f$ be the subset of $\delta$ that codes $f$ via the isomorphism $\pi^{-1}$. For this choice of $A_0$ and $A_1$ let $\kappa$ be as in (5) above. Let $\gamma$ be the least $\beth$-fixed point greater than $\max(\kappa,f(\kappa))$. Let $j:V\to M$ have critical point $\kappa$ such that $\gamma<j(\kappa)$, $j(A_0)\cap\gamma=A_0\cap\gamma$, and $j(A_1)\cap\gamma = A_1\cap\gamma$. Since $A_0\cap\gamma$ codes $R\restrict\gamma$ via G\"{o}del pairing, which is absolute to $M$, it follows by elementarity that $j(R)\restrict\gamma=R\restrict\gamma$. Furthermore, $\langle V_\gamma,\in\rangle\cong\langle\gamma,R\restrict\gamma=\langle\gamma,j(R)\restrict\gamma\rangle$ and thus the Mostowski collapse of $\langle \gamma,j(R)\restrict\gamma\rangle$ taken in $M$ is $\langle V_\gamma,\in\rangle$. Hence $V_\gamma\subseteq M$. 

It will suffice to show that $f"\kappa\subseteq\kappa$ and that $j(f)(\kappa)<\gamma$. Let me first illustrate that $j(\pi)\restrict\gamma = \pi\restrict\gamma$. It follows that 
$$\pi\restrict\gamma:\langle\gamma, R\restrict\gamma\rangle\stackrel{\cong}{\longrightarrow}\langle V_\gamma,\in\rangle$$
and since $V_\gamma\subseteq M$ one also has
$$j(\pi)\restrict\gamma:\langle \gamma, j(R)\restrict\gamma\rangle\stackrel{\cong}{\longrightarrow}\langle V_\gamma,\in\rangle.$$
Since $j(R)\restrict\gamma=R\restrict\gamma$ it follows from the uniqueness of the Mostowski collapse that $j(\pi)\restrict\gamma=\pi\restrict\gamma$. Now I will show that $f"\kappa\subseteq\kappa$. Suppose $\alpha<\kappa$. Since $A_1$, the code for $f$, agrees with $j(A_1)$ up to $\gamma$, and since $j(\pi)\restrict\gamma=\pi\restrict\gamma$, it follows that $j(f)(\alpha)=f(\alpha)<\gamma<j(\kappa)$. Since $\alpha$ is less than the critical point of $j$, it follows that $j(f(\alpha))<j(\kappa)$. By elementarity this implies $f(\alpha)<\kappa$. It easily follows that $j(f)(\kappa)=f(\kappa)<\gamma$.

% First I will show that $f"\kappa\subseteq\kappa$. Choose $\alpha<\kappa$ and let $\gamma$ be the next $\beth$-fixed point greater than $\max(\alpha,f(\alpha))$. Let $j:V\to M$ have critical point $\kappa$ such that $\gamma<j(\kappa)$, $j(A_0)\cap\gamma=A_0\cap\gamma$, and $j(A_1)\cap\gamma = A_1\cap\gamma$. Since $A_1$, the code for $f$, agrees with $j(A_1)$ up to $\gamma$, it follows that $j(f)(\alpha)=f(\alpha)<\gamma<j(\kappa)$. Since $\alpha$ is less than the critical point of $j$, it follows that $j(f(\alpha))<j(\kappa)$. By elementarity this implies $f(\alpha)<\kappa$. Now I will show that there is an embedding with critical point 

\end{proof}

%=============================================
%=============================================
%=============================================

\begin{comment}

As a corollary to the proof of Lemma \ref{lemmawoodin} above, we obtain another characterization of a cardinal being Woodin.

\begin{corollary}
The following is equivalent to $\delta$ being a Woodin cardinal.
\begin{itemize}
\item[$(5)$] For any pair of sequences $\langle A_\xi \mid \xi<\delta\rangle$ and $\langle B_\zeta\mid \zeta<\delta\rangle$ where for each $(\xi,\zeta)\in\delta\times\delta$ one has $A_\xi\subseteq V_\delta$ and $B_\zeta\subseteq\delta$, there is a $\kappa<\delta$ such that for each $\gamma<\delta$ there is a $j:V\to M$ witnessing the $\gamma$-strongness of $\kappa$ such that for each $(\xi,\zeta)\in\delta\times\delta$ it follows that $j(A_\xi)\cap V_\gamma = A_\xi\cap V_\gamma$ and $j(B_\zeta)\cap\gamma = B_\zeta\cap \gamma$.
\end{itemize}
\end{corollary}

\end{comment}

%=============================================
%=============================================
%=============================================

If $\delta$ is Woodin, then this is witnessed by embeddings as in Lemma \ref{lemmawoodin}(3). By considering a factor diagram, these embeddings can always be assumed to be extender embeddings, meaning that the target of such an embedding, $j:V\to M$, is of the form 
$$M=\{j(h)(a)\mid \textrm{$h:V_\kappa\to V$, $a\in V_\gamma$, and $h\in V$}\}.$$

The following lemma will be required in our proof of Theorem \ref{theoremwoodin}. 

\begin{lemma}\label{lemmawoodinmenas}
Suppose $\kappa$ is ${<}\delta$-strong for $A\subseteq V_\delta$ where $\delta$ is a Woodin cardinal. There is a function $\ell:\kappa\to \kappa$ such that for any $\theta<\delta$ there is a $j:V\to M$ witnessing that $\kappa$ is $\theta$-strong for $A$ such that $j(\ell)(\kappa)=\theta$.
\end{lemma}

%\begin{lemma}\label{lemmawoodinmenas}
%Suppose $\delta$ is a Woodin cardinal. There is a function $\ell:\delta\to V_\delta$ such that for every $A\subseteq V_\delta$, for every $x\in V_\delta$, and for every $\kappa<\delta$ that is ${<}\delta$-strong for $A$, there is a $\lambda<\delta$ and an embedding $j:V\to M$ witnessing that $\kappa$ is $\lambda$-strong for $A$ such that $j(\ell)(\kappa)=x$.
%\end{lemma}

% Let me note here that the proof of Lemma \ref{lemmawoodinmenas} follows Laver's proof of a similar result for supercompact cardinals in \cite{Laver:MakingSupercompactnessIndestructible}. Additionally, in \cite{GitikShelah:OnCertainIndestructibility} Gitik and Shelah state an analagous result for strong cardinals and Hamkins carries out a systematic study of such results for many different large cardinals in \cite{Hamkins:AClassOfStrongDiamondPrinciples}.

\begin{proof}

Define a function $\ell$ with domain $\kappa$ as follows. If $\gamma<\kappa$ is not ${<}\delta$-strong for $A$ then define $\ell(\gamma)$ to be the least ordinal such that $\gamma$ is not $\ell(\gamma)$-strong for $A$. Otherwise define $\ell(\gamma)=0$. 

Let me show that $\ell(\gamma)<\kappa$ for each $\gamma<\kappa$. Suppose $\gamma$ is not ${<}\delta$-strong for $A$ and that $\ell(\gamma)\geq\kappa$. I will show that since $\kappa$ is ${<}\delta$-strong for $A$ it follows that $\gamma$ is also ${<}\delta$-strong for $A$, a contradiction. Choose $\theta<\delta$ and let $j:V\to M$ witness that $\kappa$ is $\theta$-strong for $A$. Since $\ell(\gamma)\geq\kappa$ it follows that $\gamma$ is ${<}\kappa$-strong for $A$. By elementarity $\gamma=j(\gamma)$ is ${<}j(\kappa)$-strong for $j(A)$ in $M$. Thus $\gamma$ is $\theta$-strong for $j(A)$ in $M$. Let $i:M\to N$ witness this. Now let $j^*:=i\circ j:V\to N$. It follows that $\gamma$ is the critical point of $j^*$, that $j^*(\gamma)=i(j(\gamma))=i(\gamma)>\theta$, and $j^*(A)\cap\theta = i(j(A))\cap\theta = j(A)\cap\theta = A\cap\theta$. Hence $\gamma$ is $\theta$-strong for $A$. This implies that $\gamma$ is ${<}\delta$-strong for $A$, a contradiction. This shows that $\ell$ is a function from $\kappa$ to $\kappa$.

Now fix $\theta<\delta$ and let $j:V\to M$ be an embedding witnessing that $\kappa$ is $\theta$-strong for $A$ such that $\kappa$ is not $\theta$-strong for $A$ in $M$. Such an embedding can be obtained by taking $j(\kappa)$ to be minimal. It follows that $\kappa$ is $\beta$-strong for $A$ in $M$ for every $\beta<\theta$. Thus, $j(f)(\kappa)=\theta$.

\end{proof}

The next widely known lemma\footnote{I would like to thank Arthur Apter for an enlightening discussion concerning Lemma \ref{lemmawoodinclosed} and its proof.} is important for our proof of Theorem \ref{theoremwoodin}, because it easily implies that if $\delta$ is a Woodin cardinal, then one can force the continuum function to agree with any Easton function on the interval $[\delta,\infty)$.

\begin{lemma}\label{lemmawoodinclosed}
If $\delta$ is a Woodin cardinal and $\P$ is ${<}\delta$-closed then $\delta$ remains Woodin after forcing with $\P$.
\end{lemma}

\begin{proof} For this proof, I will use the definition of Woodin cardinal as opposed to one of the characterizations given in Lemma \ref{lemmawoodin}.
Let $G$ be generic for $\P$ and suppose $p\in G$ and $p\forces\dot{f}:\delta\to\delta$. Let $D$ be the set of conditions $q\leq p$ such that $q$ forces there is a $\kappa<\delta$ such that $\dot{f}"\kappa\subseteq\kappa$ and there is a $j:V[\dot{G}]\to M[j(\dot{G})]$ with critical point $\kappa$ and $(V_{j(\dot{f})(\kappa)})^{V[\dot{G}]}\subseteq M[j(\dot{G})]$. Note that the existence of the previous embedding is equivalent to the existence of an extender that has a first order definition. I will show that $D$ is dense below $p$. Choose $r\leq p$ and use the ${<}\delta$-closure of $\P$ to find a descending sequence $\langle p_\alpha\mid\alpha<\delta\rangle$ of conditions below $r$ such that $p_\alpha$ decides $\dot{f}\restrict(\alpha+1)$ for each $\alpha<\delta$. Let $F:\delta\to\delta$ be the function in $V$ determined by the sequence $\langle p_\alpha\mid\alpha<\delta\rangle$. By applying the Woodinness of $\delta$ in $V$ to $F$ find a $\kappa<\delta$ such that $F"\kappa\subseteq \kappa$ and there is a $j:V\to M$ with critical point $\kappa$ and $V_{j(F)(\kappa)}\subseteq M$. In addition, by taking a factor embedding if necessary, one may assume without loss of generality that $M=\{j(h)(a)\mid \textrm{$h:V_\kappa\to V$, $a\in V_{j(F)(\kappa)}$, and $h\in V$}\}$. Now choose $\alpha<\delta$ large enough so that $p_\alpha$ forces $\dot{f}$ to agree with $F$ up to and including at $\kappa$. Let $H$ be $V$-generic for $\P$ with $p_\alpha\in H$. Then $\dot{f}^H"\kappa\subseteq\kappa$. Since $\P$ is ${\leq}\kappa$-distributive, it follows by Lemma \ref{lemmalambdadist} that $j$ lifts to $j:V[H]\to M[j(H)]$. By elementarity and the fact that $p_\alpha\in H$, it follows that $j(\dot{f}^H)(\kappa)=j(F)(\kappa)$. Since $\P$ is ${<}\delta$-closed, it follows that $(V_{j(F)(\kappa)})^{V[H]}=V_{j(F)(\kappa)}$. Thus, $(V_{j(\dot{f}^H)(\kappa)})^{V[H]}=(V_{j(F)(\kappa)})^{V[H]}=V_{j(F)(\kappa)}\subseteq M\subseteq M[j(H)]$. This shows that $p_\alpha\in D$ and thus that $D$ is dense below $p$.

Now choose a condition $q\in G\cap D$ so that by the definition of $D$ it follows that in $V[G]$ there is a $\kappa<\delta$ such that $f"\kappa\subseteq\kappa$ and there is a $j:V[G]\to M[j(G)]$ with critical point $\kappa$ and $V[G]_{j(f)(\kappa)}\subseteq M[j(G)]$.
\end{proof}

Kanamori was the first to generalize Sacks forcing to uncountable cardinals in \cite{Kanamori:PerfectSetForcing}. In what follows I will use a version of generalized Sacks forcing introduced by Friedman and Thompson in \cite{FriedmanThompson:PerfectTreesAndElementaryEmbeddings}, which works particularly well for preserving large cardinals. For the reader's convenience I will recall the definition and properties of generalized Sacks forcing given in \cite{FriedmanThompson:PerfectTreesAndElementaryEmbeddings} and \cite{FriedmanHonzik:EastonsTheoremAndLargeCardinals}.

Suppose $\kappa$ is an inaccessible cardinal. Then $p\subseteq 2^{<\kappa} $ is a \emph{perfect $\kappa$-tree} if the following conditions hold.
\begin{enumerate}
\item[$(1)$] If $s\in p$ and $t \in 2^{<\kappa}$ is an initial segment of $s$, then $t\in p$.
\item[$(2)$] If $\langle s_\alpha\mid\alpha<\eta\rangle$ is a sequence of elements of $p$ with $\eta<\kappa$ where $s_\alpha\subseteq s_\beta$ for $\alpha<\beta$, then $\bigcup_{\alpha<\eta} s_\alpha\in p$.
\item[$(3)$] For each $s\in p$ there is a $t\in p$ with $s\subseteq t$ and $t\concat 0,t\concat 1\in p$.
\item[$(4)$] Let $\Split(p)=\{s\in p\mid s\concat 0,s\concat 1\in p\}$. Then for some unique closed unbounded set $C(p)\subseteq\kappa$, $\Split(p)=\{s\in p\mid \length(s)\in C(p)\}$.
\end{enumerate}

\emph{Sacks forcing on $\kappa$} is denoted by $\Sacks(\kappa)$ and conditions in $\Sacks(\kappa)$ are perfect $\kappa$-trees. For $p,q\in\Sacks(\kappa)$, one says that $p$ is stronger than $q$ and writes $p\leq q$ if and only if $p\subseteq q$. For a condition $p\in\Sacks(\kappa)$ let $\langle \alpha_i\mid i<\kappa\rangle$ be the increasing enumeration of $C(p)$. Let $\Split_i(p):=\{s\in p \mid \length(s)=\alpha_i\}$ denote the $i^{th}$ splitting level of $p$. For $p,q\in\Sacks(\kappa)$, define $p\leq_\beta q$ if and only if $p\leq q$ and $\Split_i(p)=\Split_i(q)$ for $i<\beta$. It is easy to verify that $\Sacks(\kappa)$ is ${<}\kappa$-closed and satisfies the $\kappa^{++}$-chain condition under $\GCH$. By standard arguments, this implies that $\Sacks(\kappa)$ preserves cardinals less than or equal to $\kappa$ and greater than or equal to $\kappa^{++}$ under $\GCH$. Furthermore, as shown in \cite{FriedmanThompson:PerfectTreesAndElementaryEmbeddings}, $\Sacks(\kappa)$ satisfies the following fusion property. If $\langle p_\alpha\mid\alpha<\kappa\rangle$ is a decreasing sequence of conditions in $\Sacks(\kappa)$ and for each $\alpha<\kappa$, $p_{\alpha+1}\leq_{\alpha}p_{\alpha}$, then the sequence has a lower bound in $\Sacks(\kappa)$. The sequence $\langle p_\alpha\mid\alpha<\kappa\rangle$ is called a fusion sequence. This fusion property implies that $\Sacks(\kappa)$ preserves $\kappa^+$ by the following straightforward argument. Suppose $p\forces \dot{f}:\check{\kappa}\to\check{\kappa}^+$. One can build a fusion sequence $\langle p_\alpha\mid\alpha<\kappa\rangle$ such that for each $\alpha<\kappa$, the condition $p_\alpha\in\Sacks(\kappa)$ forces $\dot{f}(\check{\alpha})$ to equal the check name of an element of some set $A_\alpha=\{{\beta}_\xi\mid\xi<2^\alpha\}$ where each $\beta_\xi$ is less than $\kappa^+$. By the fusion property, this sequence has a lower bound, call it $r$, and it follows that $r\forces \ran(\dot{f})\subseteq \bigcup_{\alpha<\kappa}A_\alpha$. Since $\bigcup_{\alpha<\kappa}A_\alpha$ has size at most $\kappa$, it follows that $r$ forces $\ran(\dot{f})$ to be bounded below $\kappa^+$. The forcing $\Sacks(\kappa)$ adds a single subset of $\kappa$ given by a cofinal branch through $2^{<\kappa}$ and preserves cardinals under $\GCH$.

Define $\Sacks(\kappa,\lambda)$ to be the product forcing obtained by taking the product of $\lambda$-many copies of $\Sacks(\kappa)$ with supports of size less than or equal to $\kappa$. Thus, a condition $\vec{p}\in\Sacks(\kappa,\lambda)$ can be thought of as a function $\vec{p}:\lambda\to\Sacks(\kappa)$ such that the set $\{\alpha<\lambda\mid \vec{p}(\alpha)\not= 2^{<\kappa}\}$ has size at most $\kappa$. The ordering on $\Sacks(\kappa,\lambda)$ is given by the usual product ordering. It is easy to verify that $\Sacks(\kappa,\lambda)$ is ${<}\kappa$-closed and satisfies the $\kappa^{++}$-chain condition under $\GCH$. Thus, assuming $\GCH$, the poset $\Sacks(\kappa,\lambda)$ preserves cardinals less than or equal to $\kappa$ and greater than or equal to $\kappa^{++}$. To show that $\Sacks(\kappa,\lambda)$ preserves $\kappa^{+}$ one may use the following generalized fusion property (see \cite{FriedmanThompson:PerfectTreesAndElementaryEmbeddings}). For $X\subseteq\lambda$ and $\vec{p},\vec{q}\in\Sacks(\kappa,\lambda)$ write $\vec{p}\leq_{\beta,X}\vec{q}$ if and only if $\vec{p}\leq \vec{q}$ and for each $\alpha\in X$, $\vec{p}(\alpha)\leq_\beta\vec{q}(\alpha)$. The generalized fusion property for $\Sacks(\kappa,\lambda)$ asserts that if $\langle \vec{p}_\alpha\mid\alpha<\kappa\rangle$ is a descending sequence of conditions in $\Sacks(\kappa,\lambda)$ and there is an increasing sequence $\langle X_\alpha\mid\alpha<\kappa\rangle$ of subsets of $\lambda$, each of size less than $\kappa$, such that $\bigcup_{\alpha<\kappa}X_\alpha=\bigcup_{\alpha<\kappa}\supp(\vec{p}_\alpha)$, and for each $\beta<\kappa$, $\vec{p}_{\beta+1}\leq_{\beta,X_\beta}\vec{p}_{\beta}$, then there is a lower bound of the sequence $\langle \vec{p}_\alpha\mid\alpha<\kappa\rangle$ in $\Sacks(\kappa,\lambda)$. The above generalized fusion property implies that $\kappa^+$ is preserved by the following argument. Suppose $\vec{p}\forces \dot{f}:\check{\kappa}\to\check{\kappa}^+$. One can build a fusion sequence $\langle \vec{p}_\alpha\mid\alpha<\kappa\rangle$ such that for each $\alpha<\kappa$, the condition $\vec{p}_\alpha$ forces $\dot{f}(\alpha)$ to belong to a subset of $\kappa^+$ of size $(2^\alpha)^\gamma$ for some $\gamma<\kappa$. A lower bound $\vec{r}$ of this fusion sequence forces a bound on $f$ below $\kappa^+$.

Since $\Sacks(\kappa,\lambda)$ is not $\kappa^+$-c.c. more than Lemma \ref{lemmachain} will be required to see that $\Sacks(\kappa,\lambda)$ preserves closure under $\kappa$ sequences on inner models. For this reason will need the following.
\begin{lemma}\label{lemmaclosuresacks}
Suppose $M\subseteq V$ is an inner model with $M^\kappa\subseteq M$ in $V$. If $G$ is $V$-generic for $\Sacks(\kappa,\lambda)$, then $M[G]^\kappa\subseteq M[G]$ in $V[G]$.
\end{lemma}
\begin{proof}
Let me recall the proof given in \cite[Lemma 3]{FriedmanThompson:PerfectTreesAndElementaryEmbeddings}. Let $G$ be generic for $\Sacks(\kappa,\lambda)$. Suppose $X$ is a $\kappa$-sequence of ordinals in $V[G]$ and that this is forced by $p\in G$. Using generalized fusion, one can show that every $q\leq p$ can be extended to a condition $r$ such that $r$ forces that $X$ can be determined from $r$ and $G$. This implies that there is such an $r\in G$. Since $r$ and $G$ are both in $M[G]$, it follows that $X\in M[G]$. 
\end{proof}

%For a proof of Lemma \ref{lemmaclosuresacks}, which replaces the chain condition argument with a generalized fusion argument, see \cite[Lemma 3]{FriedmanThompson:PerfectTreesAndElementaryEmbeddings}.

The following lemma, which is analagous to Lemma \ref{lemmaeaston} above, will be important for the proof of our main theorem. 
\begin{lemma}\label{lemmaeastonforsacks}
Suppose $\P$ is any ${\leq}\kappa$-closed forcing and $\alpha$ is an ordinal. Then after forcing with $\Sacks(\kappa,\alpha)$, $\P$ remains ${\leq}\kappa$-distributive.
\end{lemma}

\begin{proof}
Suppose $p\in\Sacks(\kappa,\lambda)\times\P$ forces that $\dot{f}$ is a function with $\dom(\dot{f})=\kappa$. One can show, using generalized fusion in the first coordinate and closure in the second coordinate, that every condition $q$ below $p$ can be extended to a condition $r$ which forces over $\Sacks(\kappa,\lambda)\times\P$ that the values of $\dot{f}$ can be determined from $r$ and $G$, the generic for $\Sacks(\kappa,\lambda)$. 
\end{proof}

%to be definable in $V[\dot{G}]$. Given a condition $(\vec{r},q)\leq(\vec{t},p)$ one can build a descending sequence $\langle (\vec{r}_\alpha,q_\alpha)\mid\alpha<\kappa\rangle$, where $\langle \vec{r}_\alpha\mid\alpha<\kappa\rangle$ is a generalized fusion sequence, such that if $(\vec{r}',q')$ is a lower bound of $\langle (\vec{r}_\alpha,q)\mid\alpha<\kappa\rangle$ then $\dot{f}^{\\times H}$ 

%Suppose $G$ is $V$-generic for $\Sacks(\kappa,\lambda)$ and $H$ is $V[G]$-generic for $\P$. Suppose $f\in V[G][H]$ is a function with domain $\kappa$. Let $\dot{f}$ be a $\Sacks(\kappa,\lambda)\times\P$-name and let $(\vec{t},p)$ be a condition in  $G\times H$ forcing that $\dot{f}$ is a function with $\dom(\dot{f})=\kappa$. One can build a descending sequence $\langle (\vec{t}_\alpha,p_\alpha)\mid\alpha<\kappa\rangle$, where $\langle \vec{t}_\alpha\mid\alpha<\kappa\rangle$ is a generalized fusion sequence, such that if $(\vec{r},q)$ is a lower bound of $\langle (\vec{t}_\alpha,p)\mid\alpha<\kappa\rangle$ then $\dot{f}^{G\times H}$ can be defined from $(\vec{r}

For a more detailed proof of Lemma \ref{lemmaeastonforsacks} see \cite[Lemma 3.7]{FriedmanHonzik:EastonsTheoremAndLargeCardinals}.

%===============================================
%
%	Proof of main theorem
%
%===============================================

\section{Proof of Theorem \ref{theoremwoodin}}

Recall the statement of Theorem \ref{theoremwoodin}.

\begin{theorem31}
Suppose $\GCH$ holds, $F:\REG\to\CARD$ is an Easton function, and $\delta$ is a Woodin cardinal with $F"\delta\subseteq\delta$. Then there is a cofinality-preserving forcing extension in which $\delta$ remains Woodin and $2^\gamma=F(\gamma)$ for each regular cardinal $\gamma$.
\end{theorem31}

\noindent \textit{Proof of Theorem \ref{theoremwoodin}.}

Suppose $\delta$ is a Woodin cardinal and $F:\REG\to\CARD$ is an Easton function with $F"\delta\subseteq\delta$. For an ordinal $\alpha$ let $\bar{\alpha}$ denote the least closure point of $F$ greater than $\alpha$. The forcing is, the same iteration introduced in \cite{FriedmanHonzik:EastonsTheoremAndLargeCardinals}, that is, an Easton support iteration $\P=\langle (\P_{\eta},\dot{\Q}_{\eta}) : \eta\in\ORD \rangle$ of Easton support products defined as follows. 
\begin{enumerate}
\item[$(1)$] If $\eta$ is an inaccessible closure point of $F$ in $V^{\P_{\eta}}$, then $\dot{\Q}_{\eta}$ is a $\P_{\eta}$-name for the Easton support product $$\Sacks(\eta,F(\eta))\times \prod_{\gamma\in[\eta,\bar{\eta})\cap\REG} \Add(\gamma,F(\gamma))$$
as defined in $V^{\P_{\eta}}$ and $\P_{\eta+1}=\P_{\eta}*\dot{\Q}_{\eta}$
\item[$(2)$] If $\eta$ is a singular closure point of $F$ in $V^{\P_{\eta}}$, then $\dot{\Q}_{\eta}$ is a $\P_{\eta}$-name for $\prod_{\gamma\in[\eta,\bar{\eta})\cap\REG} \Add(\gamma,F(\gamma))$  as defined in $V^{\P_{\eta}}$ and $\P_{\eta+1}=\P_{\eta}*\dot{\Q}_{\eta}$.
\item[$(3)$] Otherwise, if $\eta$ is not a closure point of $F$, then $\dot{\Q}_\eta$ is a $\P_\eta$-name for trivial forcing and $\P_{\eta+1}=\P_\eta*\dot{\Q}_\eta$.
%\item[$(4)$] If $\eta_\alpha$ is a singular cardinal then conditions in $\P_{\eta_\alpha}$ have full support and if $\eta_\alpha$ is a regular cardinal then conditions in $\P_{\eta_\alpha}$ have bounded support.
\end{enumerate}

Let $G$ be $V$-generic for $\P$. As in \cite{FriedmanHonzik:EastonsTheoremAndLargeCardinals}, it follows that cardinals are preserved (see \cite[Lemma 3.6]{FriedmanHonzik:EastonsTheoremAndLargeCardinals}) and that for each regular cardinal $\gamma$ one has $2^\gamma=F(\gamma)$ (see \cite[Theorem 3.8]{FriedmanHonzik:EastonsTheoremAndLargeCardinals}). 

Let me now discuss some notation that will be useful for factoring $\P$. If $\eta$ is a closure point of $F$, then one can factor $\P\cong\P_\eta * \dot{\P}_{[\eta,\infty)}$ where $\P_\eta$ denotes the iteration up to stage $\eta$ and $\dot{\P}_{[\eta,\infty)}$ is a $\P_\eta$-name for the remaining stages. Thus $G$ naturally factors as $G\cong G_\eta*G_{[\eta,\infty)}$. The stage $\eta$ forcing in the iteration $\P$ is $\Q_{\eta}$ and I will write $\Q_{\eta}=\Q_{[\eta,\bar{\eta})}$ to emphasize the interval on which the stage $\eta$ forcing has an effect. Let $H_{[\eta,\bar{\eta})}$ denote the $V[G_{\eta}]$-generic for $\Q_{[\eta,\bar{\eta})}$ obtained from $G$. Let $\R_\gamma$ denote a particular factor of the product forcing $\Q_{[\eta,\bar{\eta})}$ so that $\Q_{[\eta,\bar{\eta})}=\prod_{\gamma\in[\eta,\bar{\eta})\cap\REG}\R_\gamma$. In this situation let $H_\gamma$ denote that $V[G_{\eta}]$-generic for $\R_\gamma$ obtained from $G$. In general, if $I\subseteq [\eta,\bar{\eta})$ then let $\Q_I=\prod_{\gamma\in I\cap\REG}\R_\gamma$.

Since $\P_{[\delta,\infty)}$ is ${<}\delta$-closed in $V^{\P_\delta}$, it follows by Lemma \ref{lemmawoodinclosed} that if $\delta$ is Woodin in $V^{\P_\delta}$ then $\delta$ remains Woodin in $V^{\P_\delta*\dot{\P}_{[\delta,\infty)}}$. Thus it will suffice to show that $\delta$ remains Woodin in $V[G_\delta]$. Let me note here that by the previous statements, one could have defined the iteration above so that $\dot{\P}_{[\delta,\infty)}$ is simply a $\P_\delta$-name for an Easton support product of Cohen forcing.

Let me note here that in what follows I will use the fact that since conditions in $\P_\delta$ have bounded support, one can view them as sequences of length less than $\delta$. Indeed, by cutting off trivial coordinates, one can view a condition $p\in\P_\delta$ as being a condition in some initial segment of the poset.

I will show that property (3) in Lemma \ref{lemmawoodin} holds in $V[G_\delta]$. Suppose $A\subseteq \delta$ with $A\in V[G_\delta]$ and let $\dot{A}$ be a $\P_\delta$-name for $A$. For each $\alpha<\delta$, let $A_\alpha$ be a maximal antichain of conditions in $\P_\delta$ that decide $\check{\alpha}\in\dot{A}$. Define a function $\sigma:\delta\to\delta$ such that $\sigma(\gamma)$ equals the least ordinal $\beta$ such that for each $\alpha<\gamma$, the antichain $A_\alpha$, is contained in $\P_\beta$.

Now I will apply the Woodinness of $\delta$ in $V$. By an argument similar to that for Lemma \ref{lemmawoodin}(5), i.e. by coding the name $\dot{A}\subseteq V_\delta$, the Easton funciton $F\cap\delta\times\delta$, and the function $\sigma\subseteq\delta\times\delta$, into a single subset of $\delta$, that there is a $\kappa<\delta$ that is ${<}\delta$-strong for the name $\dot{A}$, the Easton function $F\restrict \delta$, and the function $\sigma$. As an abbreviation, I will say that such a $\kappa$ is ${<}\delta$-strong for $\langle\dot{A},F,\sigma\rangle$. Since $C_F:=\{\alpha<\delta\mid F"\alpha\subseteq\alpha\}$ is a closed unbounded subset of $\delta$ and since the set $S:=\{\kappa<\delta\mid\textrm{$\kappa$ is ${<}\delta$-strong for $\langle \dot{A},F,\sigma\rangle$}\}$ is stationary, one may choose such a $\kappa\in C\cap S$. This is, of course, necessary since there is no hope of $\kappa$ remaining measurable in $V[G_\delta]$ if $\kappa$ is not a closure point of $F$. 

Fix $\kappa<\delta$ such that $\kappa$ is a closure point of $F$ and $\kappa$ is ${<}\delta$-strong for $\langle \dot{A},F,u\rangle$. Fix a function $\ell:\kappa\to\kappa$ as in Lemma \ref{lemmawoodinmenas}. I will show that property (3) in Lemma \ref{lemmawoodin} holds for this $\kappa$ and the initially chosen $A\subseteq \delta$ in $V[G_\delta]$. 

Since the inaccessible closure points of $F$ are unbounded in $\delta$, one can choose $\mu$ to be an inaccessible closure point of $F$ with $F(\kappa)<\mu<\delta$. It will suffice to show that in $V[G_\delta]$ there is an embedding $j:V[G_\delta]\to M[j(G_\delta)]$ with critical point $\kappa$ and $j(A)\cap\mu = A\cap\mu$. Now I will define a singular $\theta>\mu$ and lift an embedding that is $\theta$-strong for $\langle\dot{A},F,\sigma\rangle$. I will also show that the lifted embedding satisfies $j(A)\cap \mu=A\cap\mu$ and indeed witnesses that $\kappa$ is $\mu$-strong for $A$ in $V[G_\delta]$. Using a singular degree of strength is advantageous since this will mean there will be no forcing over $\theta$ on the $M$ side, and hence the relevant tail forcing will be sufficiently closed. Let $\mu'$ be the least inaccessible closure point of $\sigma$ greater than $\mu$. Define a sequence $\langle\gamma_\alpha\mid\alpha<\kappa^+\rangle$ by recursion as follows. Let $\gamma_0$ be the least inaccessible closure point of $F$ greater than $\mu'$. Assuming $\gamma_\alpha$ is defined where $\alpha<\kappa^+$, let $\gamma_{\alpha+1}$ be the least inaccessible closure point of $F$ greater than $\gamma_\alpha$. At limit stages $\zeta<\kappa^+$, assuming $\langle \gamma_\alpha\mid\alpha<\zeta\rangle$ is defined, let $\gamma_\zeta$ be the least inaccessible closure point of $F$ greater than $\sup\{\gamma_\alpha\mid\alpha<\zeta\}$. Now define $\theta:=\sup\{\gamma_\alpha\mid\alpha<\kappa^+\}$. We have
$$\kappa<F(\kappa)<\mu<\mu'<\gamma_0<\cdots<\gamma_\alpha<\cdots<\theta.$$
For emphasis, let me state the following explicitly.
\begin{itemize}
\item $\langle\gamma_\alpha\mid\alpha<\kappa^+\rangle$ is a discontinuous sequence of inaccessible closure points of $F$.
\item $\theta=\sup\{\gamma_\alpha\mid\alpha<\kappa^+\}$
\item $\sigma"\mu'\subseteq\mu'$
\end{itemize}

By assumption on $\kappa$, there is a $j:V\to M$ with critical point $\kappa$ such that the following hold.
\begin{enumerate}
\item[(1)] $V_\theta\subseteq M$ $\and$ $\theta<j(\kappa)$
\item[(2)] $j(\dot{A})\cap\theta=\dot{A}\cap\theta$ $\and$ $j(F)\restrict\theta=F\restrict\theta$ $\and$ $j(\sigma)\restrict\theta=\sigma\restrict\theta$
\item[(3)] $M=\{j(h)(s)\mid h:V_\kappa\to V, s\in V_\theta, h\in V\}$
\item[(4)] $j(\ell)(\kappa)=\theta$ (using Lemma \ref{lemmawoodinmenas})
\end{enumerate}
Since $j(F)\restrict\theta=F\restrict\theta$, the sequence $\langle \gamma_\alpha\mid\alpha<\kappa^+\rangle$ can be constructed in $M$ from $j(F)$ just as it was constructed in $V$ from $F$. This implies that
\begin{enumerate}
\item[(5)] $\cf(\theta)^M=\kappa^+$.
\end{enumerate}

\subsection{Lifting $j$ Through $G_\kappa$.}\label{subsectionwoodingkappa}

In order to lift $j$ to $V[G_\kappa]$, I will find an $M$-generic filter $j(G_\kappa)$ for $j(\P_\kappa)$ that satisfies $j"G_\kappa\subseteq j(G_\kappa)$. To do so, the length $j(\kappa)$ iteration $j(\P_\kappa)$ will be factored in $M$. See Figure \ref{figureordinals} for the placement of various ordinals involved.

% For convenience, let me recall the placement of the various ordinals involved in our argument in the following figure.

%\begin{center}
\begin{figure}
\centering
%\begin{centering}
%\beginpgfgraphicnamed{placementofordinals}
\begin{tikzpicture}[scale=0.7,>=latex]
\draw[thick] (0,0) -- (0,5.3);
\draw[thick] (3,0) -- (3,5.3); 
\draw[->,thick] (0,1) -- (3,4) [];

\draw[thick] (0,0) node [anchor=north] {$V$};
\draw[thick] (-0.15,0) -- (0.15,0);
\draw[thick] (-0.15,1) node [anchor=east] {$\kappa$} -- (0.15,1);
\draw[thick] (-0.15,1.9) -- (0.15,1.9);
\draw[thick] (-0.15,2.4) -- (0.15,2.4);
\draw[thick] (-0.15,3.2) -- (0.15,3.2);
\draw[thick] (-0.15,4) -- (0.15,4);
\draw[thick] (-0.15,5) node [anchor=east] {$\delta$} -- (0.15,5);

\draw[thick] (3,0) node [anchor=north] {$M$};
\draw[thick] (2.85,0) -- (3.15,0);
\draw[thick] (2.85,1) -- (3.15,1);
\draw[thick] (2.85,1.9) -- (3.15,1.9) node [anchor=west] {$\mu$};
\draw[thick] (2.85,2.4) -- (3.15, 2.4) node [anchor=west] {$\gamma_0$};
\draw[thick] (2.85,3.2) -- (3.15,3.2) node [anchor=west] {$\theta$};
\draw[thick] (2.85,4) -- (3.15,4) node [anchor=west] {$j(\kappa)$};
\draw[thick] (2.85,5) -- (3.15,5);
\end{tikzpicture}\caption{The placement of relevant ordinals.}\label{figureordinals}
%\end{centering}
%\endpgfgraphicnamed
\end{figure}
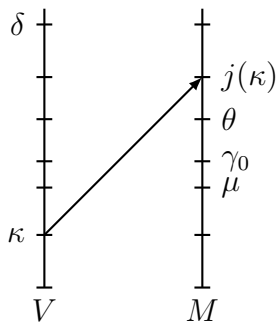
%\end{center}

Since $V_\theta\subseteq M$ it follows that $j(\P_\kappa)\cong \P_{\gamma_0}*\dot{\widetilde{\P}}_{[\gamma_0,\theta)}*\dot{\widetilde{\P}}_{[\theta,j(\kappa))}$ where $\dot{\widetilde{\P}}_{[\gamma_0,\theta)}$ is a $\P_{\gamma_0}$-term for the iteration over the interval $[\gamma_0,\theta)$ as defined in $M^{\P_{\gamma_0}}$ and similarly $\dot{\widetilde{\P}}_{[\theta,j(\kappa))}$ is a $\P_{\gamma_0}*\dot{\widetilde{\P}}_{[\gamma_0,\theta))}$-term for the tail of the iteration $j(\P_\kappa)$ as defined in $M^{\P_{\gamma_0}*\dot{\widetilde{\P}}_{[\gamma_0,\theta)}}$. Since $V_\theta\subseteq M$, the iteration $j(\P_\kappa)$ agrees with $\P_\delta$ up to stage $\gamma_0$. Thus it follows that $G_{\gamma_0}$ is $M$-generic for $\P_{\gamma_0}$. Since $\theta$ is singular in $V$, conditions in $\P_{[\gamma_0,\theta)}$ are allowed to have unbounded support. Since $M$ and $V$ do not agree on the collection of unbounded subsets of $\theta$, it follows by a density argument that $G_{[\gamma_0,\theta)}$ is not contained in $\widetilde{\P}_{[\gamma_0,\theta)}$. Nonetheless, Lemmas \ref{lemmapinfty} and \ref{lemmaaut} below will establish that there is an $M[G_{\gamma_0}]$-generic filter, call it $\widetilde{G}_{[\gamma_0,\theta)}$, in $V[G_{\gamma_0}][G_{[\gamma_0,\theta)}]$ for $\widetilde{\P}_{[\gamma_0,\theta)}$. In Lemma \ref{lemmapinfty}, I will show that there is a condition $p_\infty\in \P_{[\gamma_0,\theta)}$ which forces all dense subsets of $\widetilde{\P}_{[\gamma_0,\theta)}$ in $M[G_{\gamma_0}]$ to be met by $G_{[\gamma_0,\theta)}$. It might not be the case that $p_\infty\in G_{[\gamma_0,\theta)}$, but in Lemma \ref{lemmaaut} I will show that $p_\infty$ is in an automorphic image of $G_{[\gamma_0,\theta)}$, which I shall argue is good enough.

Let me note here that the proof of Lemma \ref{lemmapinfty} resembles the construction of $p_\infty$ in \cite[Sublemma 3.12]{FriedmanHonzik:EastonsTheoremAndLargeCardinals}. However, there is an important difference in that the forcing here, namely $\P_{[\gamma_0,\theta)}$, is an iteration, whereas in \cite{FriedmanHonzik:EastonsTheoremAndLargeCardinals}, the analagous forcing  is a product.

\begin{lemma}\label{lemmapinfty}
There is a condition $p_\infty\in \P_{[\gamma_0,\theta)}$ such that if $G^*_{[\gamma_0,\theta)}$ is $V[G_{\gamma_0}]$-generic for $\P_{[\gamma_0,\theta)}$ with $p_\infty\in G^*_{[\gamma_0,\theta)}$, then $G^*_{[\gamma_0,\theta)}\cap \widetilde{\P}_{[\gamma_0,\theta)}$ is $M[G_{\gamma_0}]$-generic for $\widetilde{\P}_{[\gamma_0,\theta)}$.
\end{lemma}

\begin{proof}

By our choice of $\theta$, the sequence $\langle \gamma_\alpha\mid \alpha<\kappa^+\rangle$ is an increasing cofinal sequence of inaccessible closure points of $F$ in $\theta$. Recall the placement of the following ordinals.
$$\mu<\mu'<\gamma_0<\gamma_1<\cdots<\gamma_\alpha<\cdots<\theta$$
It follows that, in $M[G_{\gamma_0}]$, for each $\alpha<\kappa^+$,
$$\widetilde{\P}_{[\gamma_0,\theta)}\cong\P_{[\gamma_0,\gamma_\alpha)}*\dot{\widetilde{\P}}_{[\gamma_\alpha,\theta)}$$
where $\P_{[\gamma_0,\gamma_\alpha)}$ is $\gamma_\alpha^+$-c.c. in $V[G_{\gamma_0}]$ and $\widetilde{\P}_{[\gamma_\alpha,\theta)}$ is forced to be ${<}\gamma_\alpha$-closed.

A few sublemmas will be required.

\begin{sublemma}\label{sublemmapd}
Suppose $p_*=(r_*,\dot{q}_*)\in\R*\dot{\Q}$ and $D\subseteq \R*\dot{\Q}$ is open dense. Then there is an $\R$-name $\dot{q}_D$ such that the following hold.
\begin{enumerate}
\item[$(1)$] $(r_*,\dot{q}_D)\leq(r_*,\dot{q}_*)$
\item[$(2)$] $\bar{D}=\{r\leq r_*\mid(r,\dot{q}_D)\in D\}$ is open dense in $\R$ below $r_*$.
\item[$(3)$] $r_*\forces_\R \exists r\in\dot{G}\ (r,\dot{q}_D)\in D$
\end{enumerate}
\end{sublemma}

\begin{proof} I will work below $(r_*,\dot{q}_*)$. Choose $(r_0,\dot{q}_0)\leq (r_*,\dot{q}_*)$ with $(r_0,\dot{q}_0)\in D$. Let $r_0'\leq r$  with $r_0'\perp r_0$. Now let $(r_1,\dot{q}_1)\leq (r_0',\dot{q}_*)$ with $(r_1,\dot{q}_1)\in D$. Proceed by induction.

If $\alpha$ is a successor ordinal, say $\alpha=\beta+1$, choose $r_\beta'\leq r_*$ with $r_\beta'\perp\{r_\xi\mid\xi\leq\beta\}$. Let $(r_{\beta+1},\dot{q}_{\beta+1})\in D$ with $(r_{\beta+1},\dot{q}_{\beta+1})\leq (r_\beta',\dot{q}_*)$.

If $\alpha$ is a limit ordinal, suppose $\{r_\xi\mid\xi<\alpha\}$ is the antichain of $\R$ constructed so far. Let $r_\alpha''\in\R$ be such that $r''_\alpha\perp\{r_\xi\mid\xi<\alpha\}$. Let $(r_\alpha,\dot{q}_\alpha)\in D$ with $(r_\alpha,\dot{q}_\alpha)\leq (r_\alpha'',\dot{q}_*)$.

The process terminates at some stage $\gamma$ once $A:=\{r_\xi\mid\xi<\gamma\}$ forms a maximal antichain of $\R$ below $r_*$. Let $\dot{q}_D$ be the $\R$-name obtained by mixing the names $\dot{q}_\xi$, defined above, over $A$. In other words, $\dot{q}_D$ has the property that for each $\xi<\gamma$ the condition $r_\xi$ forces $\dot{q}_D=\dot{q}_\xi$.

Let me show that (1) holds. Any generic for $\R$ containing $r_*$ will contain $r_\xi$ for some $\xi<\gamma$. Since $r_\xi\forces\dot{q}_D=\dot{q}_\xi$ and $(r_\xi,\dot{q}_\xi)\leq (r_*,\dot{q}_*)$, it follows that $r_\xi\forces \dot{q}_D=\dot{q}_\xi \leq\dot{q}_*$. Hence $r_*\forces \dot{q}_D\leq\dot{q}_*$.

%Let $A\subseteq\R$ be a maximal antichain of conditions below $r^*$ such that $r\in A$ implies that there is a $\dot{q}_r$ with $(r,\dot{q}_r)\in D$ where $(r,\dot{q}_r)\leq (r_*,\dot{q}_*)$ (such an antichain can be built by induction). By the mixing lemma, let $\dot{q}_D$ be an $\R$-name such that for each $r\in A$, $r\forces \dot{q}_D=\dot{q}_r$. 

I will now show that (2) holds. Since $D$ is open it easily follows that $\bar{D}$ is open. Suppose $p\leq r_*$ with $p\in \R$. Since $A$ is a maximal antichain of $\R$ below $r_*$ the condition $p$ is compatible with some $r_\xi\in A$. Thus, let $s\in \R$ with $s\leq r_\xi$ and $s\leq p$. Since $(r_\xi,\dot{q}_\xi)\in D$ and $D$ is open dense, to show that $s\in \bar{D}$ it will suffice to show that $(s,\dot{q}_D)\leq (r_\xi,\dot{q}_\xi)$. This easily follows since $s\leq r_\xi$ and $r_\xi\forces \dot{q}_D=\dot{q}_\xi$ imply that $s\forces \dot{q}_D\leq\dot{q}_\xi$.
\end{proof}

\begin{sublemma}\label{sublemmajoel}
Suppose $q\in\widetilde{\P}_{[\gamma_0,\theta)}$. For all functions $h\in V$ with $\dom(h)=V_\kappa$ and all $\beta<\theta$ there is a $p\leq q$ with $p\in\widetilde{\P}_{[\gamma_0,\theta)}$ such that if $p\in G^*_{[\gamma_0,\theta)}$ is $V[G_{\gamma_0}]$-generic for $\P_{[\gamma_0,\theta)}$, then $G^*_{[\gamma_0,\theta)}$ meets every dense subset of $\widetilde{\P}_{[\gamma_0,\theta)}$ of the form $j(h)(a)^{G_{\gamma_0}}$ where $a\in V_\beta$.
\end{sublemma}

\begin{proof}
Fix $q\in\widetilde{\P}_{[\gamma_0,\theta)}$, a function $h$, and $\beta$ as in the statement of the sublemma. I will obtain the condition $p\leq q$ as a lower bound of a descending sequence of conditions in $\widetilde{\P}_{[\gamma_0,\theta)}$. Since $\langle\gamma_\alpha\mid\alpha<\kappa^+\rangle$ is cofinal in $\theta$, one may choose $\gamma_\alpha>|V_\beta|$. It follows that there is an enumeration $\vec{D}=\langle D^h_\xi\mid\xi<\zeta\rangle$, in $M[G_{\gamma_0}]$, of all dense subsets of $\widetilde{\P}_{[\gamma_0,\theta)}$ of the form $j(h)(a)^{G_{\gamma_0}}$ with $a\in V_\beta$. Clearly one has $\zeta\leq|V_\beta|<\gamma_\alpha$. Factor $\widetilde{\P}_{[\gamma_0,\theta)}$ as $\widetilde{\P}_{[\gamma_0,\theta)}\cong\P_{[\gamma_0,\gamma_{\alpha})} * \widetilde{\P}_{[\gamma_{\alpha},\theta)}$. In order to simplify notation, let me define $\R:=\P_{[\gamma_0,\gamma_{\alpha})}$ and $\dot{\Q}:=\widetilde{\P}_{[\gamma_{\alpha},\theta)}$, so that $\widetilde{\P}_{[\gamma_0,\theta)}\cong \R*\dot{\Q}$. Note that $\forced_\R$ ``$\dot{\Q}$ is ${<}\gamma_{\alpha}$-closed.'' Since $q\in\widetilde{\P}_{[\gamma_0,\theta)}\cong \R*\dot{\Q}$ one may write $q=(r_*,\dot{q}_*)$ where $r_*=q\restrict[\gamma_0,\gamma_{\alpha})\in\R$ and $\dot{q}_*$ denotes the $\R$-name, $q\restrict[\gamma_{\alpha},\theta)$.

By the repeated application of Sublemma \ref{sublemmapd}, and using the fact that $\forced_\R$ ``$\dot{\Q}$ is ${<}\gamma_{\alpha}$-closed,'' one may build a descending sequence of conditions $\langle (r_*,\dot{q}_\xi)\mid\xi\leq\zeta\rangle$ in $\R*\dot{\Q}$ such that for each $\xi\leq\zeta$, the set 
$$\bar{D}^h_\xi:=\{r\leq r_*\mid (r,\dot{q}_\xi)\in D_\xi\}$$ 
is dense below $r^*$ in $\R=\P_{[\gamma_0,\gamma_{\alpha})}$. Let $p:=(r_*,\dot{q}_\zeta)$.

Suppose $p\in G^*_{[\gamma_0,\theta)}$ is $V[G_{\gamma_0}]$-generic for $\P_{[\gamma_0,\theta)}$. Fix an $a\in V_\beta$ such that $j(h)(a)^{G_{\gamma_0}}$ is a dense subset of $\widetilde{\P}_{[\gamma_0,\theta)}$. Since $j(h)(a)^{G_{\gamma_0}}$ must appear on the enumeration of dense sets we fixed above, there is a $\xi<\zeta$ such that $D^h_\xi=j(h)(a)^{G_{\gamma_0}}$. Since $\bar{D}_\xi$ is dense below $r^*$ in $\R=\P_{[\gamma_0,\gamma_{\alpha})}$ there is a condition $r\in G^*_{[\gamma_0,\gamma_{\alpha})}\cap\bar{D}_\xi$. By definition of $\bar{D}_\xi$, it follows that $(r,\dot{q}_\xi)\in D_\xi^{h}$. By padding $r$ with $\1$'s, one sees that there is an $\R$-name $\dot{b}$ such that $(r,\dot{b})\in G^*_{[\gamma_0,\theta)}$. Since $p=(r_*,\dot{q}_{\zeta})$ and $(r,\dot{b})$ are both in $G^*_{[\gamma_0,\theta)}$ they have a common extension $(r',\dot{q}')\in G^*_{[\gamma_0,\theta)}$. Since $(r',\dot{q}')\leq (r,\dot{q}_{\zeta})$, and since $r_*\forces \dot{q}_\zeta\leq\dot{q}_\xi$, it follows that $(r',\dot{q}')\leq (r,\dot{q}_\xi)$. Since $G^*_{[\gamma_0,\theta)}$ is a filter, one concludes that $(r,\dot{q}_\xi)\in G^*_{[\gamma_0,\theta)}\cap D^{h}_\xi$.
\end{proof}

Continuing with the proof of Lemma \ref{lemmapinfty}, I will now use Sublemma \ref{sublemmajoel} to construct the condition $p_\infty\in\P_{[\gamma_0,\theta)}$. Let $\langle f_\xi\mid\xi<\kappa^+\rangle\in V$ be a sequence of functions with domain $V_\kappa$ such that every dense subset of $\widetilde{\P}_{[\gamma_0,\theta)}$ in $M[G_{\gamma_0}]$ has a name of the form $j(f_\xi)(a)$ for some $\xi<\kappa^+$ and some $a\in V_\theta$. Let ${w}:\kappa^+\to\kappa^+\times\kappa^+$ be a bijection. It follows that ${w}\in M[G_{\gamma_0}]$ since ${w}\in V_\theta$. For each $\alpha<\kappa^+$ let ${w}(\alpha)=({w}(\alpha)_0,{w}(\alpha)_1)$. The function ${w}$ provides a well-ordering of pairs of the form $(f_\xi,\gamma_\alpha)$. Notice that the well-ordering is not in $M[G_{\gamma_0}]$ since the sequence $\langle f_\xi\mid\xi<\kappa^+\rangle$ is not in $M[G_{\gamma_0}]$. I will use this well-ordering of all pairs of the form $(f_\xi,\gamma_\alpha)$ of order type $\kappa^+$ to build a descending sequence of conditions $\langle p_\beta\mid\beta<\kappa^+\rangle$ in $V[G_{\gamma_0}]$ with $p_\beta\in \widetilde{\P}_{[\gamma_0,\theta)}$ such that if $p_\beta\in G^*_{[\gamma_0,\theta)}$ is $V[G_{\gamma_0}]$-generic for $\P_{[\gamma_0,\theta)}$, then $G^*_{[\gamma_0,\theta)}$ meets $D^{f_\xi}_a=j(f_\xi)(a)_{G_{\gamma_0}}$ for each $a\in V_{\gamma_\alpha}$ where ${w}(\beta)=(\xi,\alpha)$. Since the above mentioned well-ordering will not be in $M[G_{\gamma_0}]$, I will need the next lemma to build the descending sequence.

\begin{lemma}
The model $M[G_{\gamma_0}]$ is closed under $\kappa$-sequences in $V[G_{\gamma_0}]$.
\end{lemma}

\begin{proof}
Since $\P_\kappa$ is $\kappa$-c.c. in $V$, it follows that $M[G_\kappa]^\kappa\subseteq M[G_\kappa]$ in $V[G_\kappa]$. By Lemma \ref{lemmaclosuresacks} it follows that $M[G_\kappa][H_\kappa]^\kappa\subseteq M[G_\kappa][H_\kappa]$ in $V[G_\kappa][H_\kappa]$. Since the remaining forcing $\Q_{[\kappa^+,\bar{\kappa})}*\P_{[\bar{\kappa},\gamma_0)}$ is ${\leq}\kappa$-distributive in $V[G_\kappa][H_\kappa]$ (by Lemma \ref{lemmaeastonforsacks}) it follows that $M[G_{\gamma_0}]^\kappa\subseteq M[G_{\gamma_0}]$ in $V[G_{\gamma_0}]$.
\end{proof}

I will now use the bijection $w:\kappa^+\to\kappa^+\times\kappa^+$ defined above to build the descending sequence. Let $p_0$ be the condition obtained by applying Sublemma \ref{sublemmajoel} below the trivial condition to the function $h=f_{\xi}$ where $\xi={w}(0)_0$ and to the ordinal $\beta=\gamma_\alpha$ where $\alpha={w}(0)_1$. For successor stages, assume that $\langle p_\eta\mid\eta\leq\zeta\rangle$ has been constructed, where $\zeta<\kappa^+$. Let $p_{\zeta+1}\in\widetilde{\P}_{[\gamma_0,\theta)}$ be obtained by applying Sublemma \ref{sublemmajoel} below $p_\zeta$ to the function $h=f_\xi$ where $\xi={w}(\zeta+1)_0$ and to the ordinal $\beta=\gamma_\alpha$ where $\alpha={w}(\zeta+1)_1$. At limit stages $\zeta<\kappa^+$, assume $\langle p_\eta\mid\eta<\zeta\rangle$ has been constructed. The fact that $M[G_{\gamma_0}]^\kappa\subseteq M[G_{\gamma_0}]$ implies that the sequence $\langle p_\eta\mid\eta<\zeta\rangle$ is in $M[G_{\gamma_0}]$ since it has been constructed from an initial segment of $\langle f_\xi\mid\xi<\kappa^+\rangle$ and from $\langle \gamma_\alpha\mid\alpha<\kappa^+\rangle\in M[G_{\gamma_0}]$. Since $\widetilde{\P}_{[\gamma_0,\theta)}$ is ${<}\gamma_0$-closed in $M[G_{\gamma_0}]$, one may let $p_\zeta'\in \widetilde{\P}_{[\gamma_0,\theta)}$ be a lower bound of $\langle p_\beta\mid\beta<\zeta\rangle$. Now let $p_\zeta$ be obtained by applying Sublemma \ref{sublemmajoel} below $p_\zeta'$ to the function $f_\xi$ where $\xi={w}(\zeta)_0$ and the ordinal $\beta=\gamma_\alpha$ where $\alpha={w}(\zeta)_1$.

This defines the sequence $\langle p_\eta\mid\eta<\kappa^+\rangle$ in $V[G_{\gamma_0}]$ where $p_\eta\in\widetilde{\P}_{[\gamma_0,\theta)}\subseteq\P_{[\gamma_0,\theta)}$ for each $\eta<\kappa^+$. Let $p_\infty\in\P_{[\gamma_0,\theta)}$ be a lower bound of $\langle p_\eta\mid\eta<\kappa^+\rangle$.

Suppose $p_\infty\in G^*_{[\gamma_0,\theta)}$ is $V[G_{\gamma_0}]$-generic for $\P_{[\gamma_0,\theta)}$. Suppose $D\in M[G_{\gamma_0}]$ is a dense subset of $\widetilde{\P}_{[\gamma_0,\theta)}$. Then $D=D^{f_\xi}_a=j(f_\xi)(a)^{G_{\gamma_0}}$ for some $\xi<\kappa^+$ and where $a\in V_{\gamma_\alpha}$ for some $\alpha<\kappa^+$. Let $\zeta<\kappa^+$ with ${w}(\zeta)=({w}(\zeta)_0,{w}(\zeta)_1)=(\xi,\alpha)$. Since $p_\infty\leq p_\zeta$, it follows that $p_\zeta\in G^*_{[\gamma_0,\theta)}$ and hence, $G^*_{[\gamma_0,\theta)}$ meets $D^{f_\xi}_a$, by Sublemma \ref{sublemmajoel}.

This concludes the proof of Lemma \ref{lemmapinfty}. \hfill $\Box$

I will now show that there is an automorphic image of $G_{[\gamma_0,\theta)}$ containing $p_\infty$.

\begin{lemma}\label{lemmaaut}
Suppose $c\in \P_{[\gamma_0,\theta)}$. There is an automorphism $\pi:\P_{[\gamma_0,\theta)}\to\P_{[\gamma_0,\theta)}$ in $V[G_{\gamma_0}]$ such that $c\in\pi"G_{[\gamma_0,\theta)}$. 
\end{lemma}

\begin{proof}[Proof of Lemma \ref{lemmaaut}]

Working in $V[G_{\gamma_0}]$, I claim each stage in the iteration $\P_{[\gamma_0,\theta)}$ is forced to be homogeneous over the previous stages. For the Cohen forcing conditions, this claim is obvious. One can see that the Sacks forcing $\Sacks(\eta,\lambda)$ is almost homogeneous by using automorphisms that permute coordinates. Furthermore, at each stage $\alpha\in[\gamma_0,\theta)$, there is a formula $\varphi_\alpha$ defining $\dot{\Q}_\alpha$ in $V^{\P_\alpha}$ from check names (see Section \ref{sectionhomogeneousiteration}). Thus it follows from Lemma \ref{lemmahomogeneousiteration}, that $\P_{[\gamma_0,\theta)}$ is almost homogeneous in $V[G_{\gamma_0}]$. By the almost homogeneity of $\P_{[\gamma_0,\theta)}$ in $V[G_{\gamma_0}]$, every condition $p\in\P_{[\gamma_0,\theta)}$ can be extended to a condition $q\leq p$ such that there is an $f\in\Aut(\P_{[\gamma_0,\theta)})$ with $f(q)\leq c$. Therefore, by the genericity of $G_{[\gamma_0,\theta)}$, there is such a $q\in G_{[\gamma_0,\theta)}$ with such an $f\in\Aut(\P_{[\gamma_0,\theta)})$. Let $\pi:=f$. Since $\pi"G_{[\gamma_0,\theta)}$ is a filter and $\pi(q)\leq c$, it follows that $c\in \pi"G_{[\gamma_0,\theta)}$.
\end{proof}

As discussed above, one may use Lemmas \ref{lemmapinfty} and \ref{lemmaaut} to obtain $\widetilde{G}_{[\gamma_0,\theta)}\in V[G_{\gamma_0}][G_{[\gamma_0,\theta)}]$, an $M[G_{[\gamma_0,\theta)}]$-generic for $\widetilde{\P}_{[\gamma_0,\theta)}$. 

To finish lifting $j$ through $j(\P_\kappa)\cong \P_{\gamma_0}*\dot{\widetilde{\P}}_{[\gamma_0,\theta)}*\dot{\widetilde{\P}}_{[\theta,j(\kappa))}$, I will build an $M[G_{\gamma_0}][\widetilde{G}_{[\gamma_0,\theta)}]$-generic for $\widetilde{\P}_{[\theta,j(\kappa))}$ in $V[G_{\gamma_0}][G_{[\gamma_0,\theta)}]$. The following lemma will be required.

\begin{lemma}
$M[G_{\gamma_0}][\widetilde{G}_{[\gamma_0,\theta)}]$ is closed under $\kappa$-sequences in $V[G_{\gamma_0}][G_{[\gamma_0,\theta)}]$.
\end{lemma}
\begin{proof}
Since $\P_\kappa$ is $\kappa$-c.c., it follows by Lemma \ref{lemmachain} that $M[G_\kappa]$ is closed under $\kappa$-sequences in $V[G_\kappa]$. It is shown in \cite[Lemma 3.14]{FriedmanHonzik:EastonsTheoremAndLargeCardinals} and \cite[Lemma 3]{FriedmanThompson:PerfectTreesAndElementaryEmbeddings}, using a fusion argument, that $M[G_\kappa][H_\kappa]$ is closed under $\kappa$-sequences in $V[G_\kappa][H_\kappa]$. It will suffice to show that $M[G_{\gamma_0}][\widetilde{G}_{[\gamma_0,\theta)}]$ has every $\kappa$-sequence of ordinals in $V[G_{\gamma_0}][G_{[\gamma_0,\theta)}]$. Suppose $\vec{x}$ is a $\kappa$-sequence of ordinals in $V[G_{\gamma_0}][G_{[\gamma_0,\theta)}]$. Then since $\Q_{[\kappa^+,\bar{\kappa})}*\P_{[\bar{\kappa},\theta)}$ is ${\leq}\kappa$-distributive in $V[G_\kappa][H_\kappa]$, it follows that $\vec{x}\in V[G_\kappa][H_\kappa]$. Thus $\vec{x}\in M[G_\kappa][H_\kappa]\subseteq M[G_{\gamma_0}][\widetilde{G}_{[\gamma_0,\theta)}]$.
\end{proof}

Suppose $D$ is a dense subset of $\widetilde{\P}_{[\theta,j(\kappa))}$ in $M[G_{\gamma_0}][\widetilde{G}_{[\gamma_0,\theta)}]$. Let $\dot{D}\in M$ be a nice $\P_\theta$-name for $D$. Let $h$ be a function in $V$ with $\dom(h)=V_\kappa$ and $s\in V_\theta$ with $\dot{D}=j(h)(s)$. Without loss of generality, assume that $\ran(h)$ is contained in the set of nice names for dense subsets of a particular tail of $\P$. Since $\theta$ is singular, $\widetilde{\P}_{[\theta,j(\kappa))}$ is $\leq\theta$-closed in $M[G_{\gamma_0}][\widetilde{G}_{[\gamma_0,\theta)}]$. The collection $\mathcal{D}:=\{ j(h)(s)_{G_{\gamma_0}*\widetilde{G}_{[\gamma_0,\theta)}}\mid s\in V_\theta\}$ is in $M[G_{\gamma_0}][\widetilde{G}_{[\gamma_0,\theta)}]$. Since $\theta$ is a $\beth$-fixed point, there are at most $\theta$ dense subsets of $\widetilde{\P}_{[\theta,j(\kappa))}$ in $\mathcal{D}$. Thus, there is a single condition in $\widetilde{\P}_{[\theta,j(\kappa))}$ that meets every dense set in $\mathcal{D}$. Since there are at most $\kappa^+$ functions from $V_\kappa$ to nice names for dense subsets of a tail of $\P_\kappa$, and since every dense subset of $\widetilde{\P}_{[\theta,j(\kappa))}$ has a name in $M$ which is represented by such a function, the above procedure can be iterated to obtain a descending $\kappa^{+}$-sequence of conditions in $\widetilde{\P}_{[\theta,j(\kappa))}$ meeting every dense subset of $\widetilde{\P}_{[\theta,j(\kappa))}$ in $M[G_{\gamma_0}][\widetilde{G}_{[\gamma_0,\theta)}]$. Let $\widetilde{G}_{tail}$ be the $M[G_{\gamma_0}][\widetilde{G}_{[\gamma_0,\theta)}]$-generic filter for $\P_{tail}$ generated by this sequence.

Now let $j(G_\kappa):=G_{\gamma_0}*\widetilde{G}_{[\gamma_0,\theta)}*\widetilde{G}_{tail}$ and note that $j"G_\kappa\subseteq j(G_\kappa)$ since conditions in $G_\kappa$ have support bounded below the critical point of $j$. Hence by Lemma \ref{lemmaliftingcriterion}, the embedding lifts to 
$$j:V[G_\kappa]\to M[j(G_\kappa)]$$
in $V[G_{\gamma_0}][G_{[\gamma_0,\theta)}]$.

\subsection{Lifting $j$ Through $\Sacks(\kappa,F(\kappa))$.}

It remains to show that the embedding lifts further through the forcing $\P_{[\kappa,\lambda)}$. I will now argue that $j$ lifts through $\R_\kappa=\Sacks(\kappa,F(\kappa))^{V[G_\kappa]}$, the first factor of the stage $\kappa$ forcing. I will use the tuning fork method of \cite{FriedmanThompson:PerfectTreesAndElementaryEmbeddings} to construct an $M[j(G_\kappa)]$-generic for $j(\R_\kappa)=\Sacks(j(\kappa),j(F(\kappa)))^{M[j(G_\kappa)]}$ in $V[G_\kappa][H_\kappa]$ that satisfies the lifting criterion in Lemma \ref{lemmaliftingcriterion}. Say that $t\subseteq 2^{<j(\kappa)}$ is a \emph{tuning fork that splits at $\kappa$} if and only if $t=t^0\cup t^1$ where $t^0$ and $t^1$ are two distinct cofinal branches of $2^{<j(\kappa)}$ such that $t^0\cap\kappa=t^1\cap\kappa$, $t^0(\kappa)=0$, and $t^1(\kappa)=1$. For $\alpha<j(F(\kappa))$ let 
$$t_\alpha:=\bigcap\{j(p)(\alpha)\mid p\in H_\kappa\}.$$
The next lemma is key.
\begin{lemma}\label{lemmatuningfork}
If $\alpha\in j"F(\alpha)$ then $t_\alpha$ is a tuning fork that splits at $\kappa$. Otherwise, if $\alpha<j(F(\kappa))$ is not in the range of $j$, then $t_\alpha$ is a cofinal branch through $2^{<j(\kappa)}$. 
\end{lemma}

\begin{proof}
The following proof follows \cite{FriedmanThompson:PerfectTreesAndElementaryEmbeddings} closely, except that here Lemma \ref{lemmawoodinmenas} is required. Working in $V[G_\kappa]$, let 
$$X:=\bigcap \{j(C)\mid \textrm{$C\subseteq\kappa$ is club and $C\in V$}\}.$$
First let me show that $X=\{\kappa\}$. If $\alpha<\kappa$ then clearly $\alpha\notin X$ since there is a closed unbounded subset $C$ of $\kappa$ whose least element is greater than $\alpha$, and thus $\alpha\notin j(C)$. Since the limit cardinals below $\kappa$ form a closed unbounded subset of $\kappa$ it follows that any element of $X$ must be a limit cardinal in $M[j(G_\kappa)]$ which is greater than or equal to $\kappa$. Suppose $\lambda<j(\kappa)$ is a limit cardinal and $\lambda>\theta$. Then $\lambda=j(h)(a)$ for some function $h:V_\kappa\to \kappa$ in $V[G_\kappa]$ and some $a\in V_\theta$. Let $C_h:=\{\gamma<\kappa\mid\textrm{$\gamma$ is a limit cardinal and $h"V_\gamma\subseteq\gamma$}\}$. Then $C_h$ is a closed unbounded subset of $\kappa$ and $\lambda\notin j(C_h)$ since $\lambda>\theta$ and $j(h)"V_\lambda\not\subseteq\lambda$. Now suppose $\kappa<\lambda\leq\theta$. Above, the function $\ell$ is chosen using Lemma \ref{lemmawoodinmenas} so that $\ell:\kappa\to\kappa$ and $j(\ell)(\kappa)=\theta$. Then $C_\ell:=\{\gamma<\kappa\mid\textrm{$\ell"\gamma\subseteq\gamma$}\}$ is a closed unbounded subset of $\kappa$ in $V[G_\kappa]$ and $\lambda\notin j(C_\ell)$ since $\theta\in j(\ell)"\lambda$ and this implies $j(\ell)"\lambda\not\subseteq\lambda$. This shows that $X\subseteq\{\kappa\}$. Clearly $\kappa\in X$ since for each closed unbounded $C\subseteq\kappa$ in $V[G_\kappa]$, $j(C)\cap\kappa=C$.

The rest of the proof is exactly as in \cite{FriedmanThompson:PerfectTreesAndElementaryEmbeddings} and \cite{FriedmanHonzik:EastonsTheoremAndLargeCardinals}.

Let $C$ be any closed unbounded subset of $\kappa$ in $V[G_\kappa]$. Choose $\alpha<j(F(\kappa))$ and write $\alpha=j(f)(a)$ where $f:V_\kappa\to 	F(\kappa)$ and $a\in V_\theta$. It is easy to show that the following set is dense in $\Sacks(\kappa,F(\kappa))$.
$$D_C=\{p\in\Sacks(\kappa,F(\kappa))\mid\xi\in\ran(f)\implies C(p(\xi))\subseteq C\}$$
Thus there is a $p\in H_\kappa\cap D_C$ with $C(j(p)(\alpha))\subseteq j(C)$. Since $C$ was an arbitrary closed unbounded subset of $\kappa$, this, together with the fact that $X=\{\kappa\}$, implies that $t_\alpha$ can only possibly split at $\kappa$. If $\alpha\in\ran(j)$ then since $\kappa$ is a limit point of $j(C)$ for every closed unbounded $C\subseteq\kappa$ in $V[G_\kappa]$, it follows that $t_\alpha$ splits at $\kappa$ and is a tuning fork. 

If $\alpha\notin\ran(j)$ then $\ran(f)$ must have size $\kappa$ since otherwise $\alpha\in j(\ran(f))= j"\ran(f)$. Let $\langle \bar{\alpha}_i\mid i<\kappa\rangle$ enumerate $\ran(f)$. Then $j(\langle \bar{\alpha}_i\mid i<\kappa\rangle)=\langle \alpha_i\mid i<j(\kappa)\rangle$ in an enumeration of $\ran(j(f))$. It is easy to see that the set of conditions $p\in\Sacks(\kappa,F(\kappa))$ such that for each $i<\kappa$, the least splitting level of $p(\bar{\alpha}_i)$ is above level $i$ is dense. Thus there is a $p\in H_\kappa$ such that for each $i<j(\kappa)$ the least splitting level of $j(p)(\alpha_i)$ is beyond level $i$. Since $\alpha\notin\ran(j)$ it follows that $\alpha=\alpha_i$ for some $i\in[\kappa,j(\kappa))$. It follows that the first splitting level of $j(p)(\alpha)$ is above $\kappa$. Thus, $t_\alpha$ is a cofinal branch.
\end{proof}

Each $t_\alpha$ generates an $M[j(G_\kappa)]$-generic filter for $j(\R_\kappa)$ as follows. For $\alpha\in j"F(\kappa)$, let $t_\alpha^0$ and $t_\alpha^1$ be the left-most and right-most branches of $t_\alpha$ respectively; that is, for $k\in\{0,1\}$ let
$$t^k_\alpha:=\{s\in t_\alpha\mid \kappa\in\dom(s)\implies s(\kappa)=k\}.$$
For $\alpha<j(F(\kappa))$ not in the range of $j$, let $t^0_\alpha:=t_\alpha$ be the cofinal branch in Lemma \ref{lemmatuningfork}.
Let 
$$g:=\{\widetilde{p}\in j(\R_\kappa)\mid \forall \alpha<j(F(\kappa))\ t^0_\alpha\subseteq \widetilde{p}(\alpha)\}.$$
It is easy to check that $j"H_\kappa\subseteq g$, so to show that $j$ lifts through $\R_\kappa$ it remains to show that $g$ is $M[j(G_\kappa)]$-generic for $j(\R_\kappa)$. For this the following two definitions will be used, both of which are given in \cite{FriedmanThompson:PerfectTreesAndElementaryEmbeddings}. Suppose $ p\in \Sacks(\kappa,F(\kappa))^{V[G_\kappa]}$, $S\subseteq F(\kappa)$ with $|S|^{V[G_\kappa]}<\kappa$. Friedman and Thompson say that an \emph{$(S,\alpha)$-thinning of $ p$} is an extension of $ p$ obtained by thinning each $ p(\xi)$ for $\xi\in S$ to the subtree 
$$ p(\xi)\restrict s_\xi:=\{s\in p(\xi)\mid s_\xi\subseteq s\textrm{ or } s\subseteq s_\xi\}$$
where $s_\xi$ is some particular node of $p(\xi)$ on the $\alpha$-th splitting level of $p(\xi)$. A condition $p\in \Sacks(\kappa, F(\kappa))^{V[G_\kappa]}$ is said to \emph{reduce} a dense subset $D$ of $\Sacks(\kappa, F(\kappa))^{V[G_\kappa]}$ if and only if for some $S\subseteq  F(\kappa)$ of size less than $\kappa$ in $V[G_\kappa]$, any $(S,\alpha)$-thinning of $p$ meets $D$.

Let me now argue that $g$ is $M[j(G_\kappa)]$-generic for the poset $j(\R_\kappa)=\Sacks(j(\kappa),j( F(\kappa)))^{M[j(G_\kappa)]}$. Suppose $D$ is a dense subset of $j(\R_\kappa)$ in the model $M[j(G_\kappa)]$. Then by Lemma \ref{lemmaseedpreservation} one can write $D=j(h)(a)$ where $h\in V[G_\kappa]$ is a function from $V_\kappa$ to the collection of dense subsets of $\Sacks(\kappa,$ $F(\kappa))^{V[G_\kappa]}$ and $a\in V_\theta$. Let $\langle D_\beta\mid\beta<\kappa\rangle\in V[G_\kappa]$ enumerate the range of $h$. One may show, as in \cite{FriedmanThompson:PerfectTreesAndElementaryEmbeddings} that any condition $p\in \Sacks(\kappa, F(\kappa))^{V[G_\kappa]}$ can be extended to $q\leq p$ which reduces each $D_\beta$ for $\beta<\kappa$. This implies that the following is a dense subset of $\Sacks(\kappa, F(\kappa))^{V[G_\kappa]}$.
$$D':=\{p\in\R_\kappa\mid \textrm{$p$ reduces each $D_\beta$ for $\beta<\kappa$}\}$$
Thus one may choose a condition $p\in H\cap D'$. By elementarity $j(p)$ reduces each dense subset of $j(\R_\kappa)$ in the range of $j(h)$; in particular, $j(p)$ reduces $D=j(h)(a)$. Thus it follows that there is an $S\subseteq j( F(\kappa))$ of size less than $j(\kappa)$ and an $\alpha<j(\kappa)$ such that any $(S,\alpha)$-thinning of $j(p)$ meets $D$. For each $\xi\in S$ let $\widetilde{q}(\xi)$ be the thinning of $j(p)(\xi)$ obtained by choosing an initial segment of $t_\xi^0$ on the $\alpha$-th splitting level of $j(p)(\xi)$. For $\xi\in j( F(\kappa))\setminus S$ let $\widetilde{q}(\xi):=j(p)(\xi)$. The fact that $\widetilde{q}$ is a condition in $j(\R_\kappa)$ will follow from the next lemma.

\begin{lemma}\label{lemmathinned}
For any $\beta<j(\kappa)$ and any subset $S$ of $j( F(\kappa))$ of size at most $j(\kappa)$ in $M[j(G_\kappa)]$, the sequence $\langle t_\xi^0\restrict\beta\mid\xi\in S\rangle$ belongs to $M[j(G_\kappa)]$. 
\end{lemma}

\begin{proof}[Proof of Lemma \ref{lemmathinned}]

Write $\beta=j(f_0)(a)$ where $f_0:V_\kappa\to \kappa$ and $a\in V_\theta$. Let $C=\{\lambda<\kappa\mid\textrm{$f_0"V_\lambda\subseteq\lambda$ and $\lambda$ is a limit cardinal}\}$. By Lemma \ref{lemmaseedpreservation} it follows that $S=j(f)(b)$ where $f:V_\kappa\to[F(\kappa)]^{\leq\kappa}$ and $b\in V_\theta$. Since $S\subseteq j(\bigcup\ran(f))$ it can be assumed without loss of generality that $S=j(\bar{S})$ for some $\bar{S}\in[F(\kappa)]^{\leq\kappa}$. Let $\langle \bar{\alpha}_i\mid i<\kappa\rangle$ be an enumeration of $\bar{S}$. Then $j(\langle \bar{\alpha}_i\mid i<\kappa\rangle)=\langle\alpha_i\mid i<j(\kappa)\rangle$ is an enumeration of $S$. One can easily see that 
$$D=\{\bar{p}\in\Sacks(\kappa, F(\kappa))\mid\textrm{for each $i<\kappa$, $C(\bar{p}(\bar{\alpha}_i))\subseteq C\setminus(i+1)$}\}$$
is a dense subset of $\Sacks(\kappa, F(\kappa))$. Let $\bar{p}\in H_\kappa\cap D$. Then for each $i<j(\kappa)$, $C(j(\bar{p})(\alpha_i))\subseteq C\setminus (i+1)$. Thus, for each $\alpha_i$, the tree $j(\bar{p})(\alpha_i)$ has no splits between $\kappa$ and $\alpha$. If $\kappa\leq i< j(\kappa)$ then $j(\bar{p})(\alpha_i)$ does not split between $0$ and $\alpha$. If $\kappa\leq i <j(\kappa)$ then $t^0_{\alpha_i}\restrict \alpha$ is the unique element of $j(\bar{p})(\alpha_i)$ of length $\alpha$. If $i<\kappa$, then $t^0_{\alpha_i}\restrict\alpha$ is the unique element of $j(\bar{p})(\alpha_i)$ that extends $t^0_{\alpha_i}\restrict\kappa$ and takes on value $0$ at $\kappa$. 
\end{proof}

By Lemma \ref{lemmathinned}, $\widetilde{p}$ is in $M[j(G_\kappa)]$ and is thus a condition in $j(\R_\kappa)$. Furthermore, $\widetilde{p}$ meets $D$ and since $t^0_\xi\subseteq\widetilde{p}(\xi)$ for each $\xi<F(\kappa)$, it follows that $\widetilde{p}$ is in $g$. This establishes that $g$ is $M[j(G_\kappa)]$-generic for $j(\R_\kappa)$. Thus the embedding lifts to $j:V[G_\kappa][H_\kappa]\to M[j(G_\kappa)][j(H_\kappa)]$.

\subsection{Lifting $j$ Through $\Q_{[\kappa^+,\bar{\kappa})}*\P_{[\bar{\kappa},\delta)}$.}

By Lemma \ref{lemmaeastonforsacks}, the poset $\Q_{[\kappa^+,\bar{\kappa})}*\P_{[\bar{\kappa},\delta)}$ is $\leq\kappa$-distributive in $V[G_\kappa][H_\kappa]$. Thus, from Lemma \ref{lemmalambdadist} one sees that $j"H_{[\kappa^+,\bar{\kappa})}*G_{[\bar{\kappa},\delta)}$ generates an $M[j(G_\kappa)]$ $[j(H_\kappa)]$-generic filter for $j(\Q_{[\kappa^+,\bar{\kappa})}*\P_{[\bar{\kappa},\delta)})$, call it $j(H_{[\kappa^+,\bar{\kappa})}*G_{[\bar{\kappa},\delta)})$. Thus $j$ lifts to $j:V[G_\delta]\to M[j(G_\delta)]$ where $j(G_\delta):=j(G_\kappa)*(j(H_\kappa)\times j(H_{[\kappa^{+},\bar{\kappa})}))*j(G_{[\bar{\kappa},\delta)})$.

\subsection{Verifying strongness for $A$}

Let me argue that the lifted embedding $j:V[G_\delta]\to M[j(G_\delta)]$ satisfies $j(A)\cap\mu = A\cap \mu$. This will follow from the next fact. 
\begin{fact}\label{factequation}\ 
\begin{itemize}
\item[$(1)$] $j(\dot{A})\cap V_\theta = \dot{A}\cap V_\theta$
\item[$(2)$] $j(G_\delta)=G_{\gamma_0}*\widetilde{G}_{[\gamma_0,\theta)}*\widetilde{G}_{[\theta,j(\kappa))}$ agrees with $G_\delta$ up to $\mu'$ since $\mu'<\gamma_0$.
\item [$(3)$] $j(u)\restrict \mu' = u\restrict\mu'$
\end{itemize}
\end{fact}
Using the above fact, one has the following.
\begin{align*}
A\cap\mu & = \dot{A}^{G_\delta}\cap\mu \\
 & = (\dot{A}\cap V_{\mu'})^{G_{\mu'}}\cap\mu \tag{\textrm{using the definition of $u$}}\\
 & = (j(\dot{A})\cap V_{\mu'})^{G_{\mu'}}\cap\mu \tag{\textrm{by Fact \ref{factequation}(1)}}\\
 & = j(\dot{A})^{j(G_\delta)}\cap\mu \tag{\textrm{by Fact \ref{factequation}(2) and (3)}}\\
 & = j(A)\cap \mu \\
\end{align*}
This completes the proof of Theorem \ref{theoremwoodin}.
\end{proof}

%===========================================================
%===========================================================
%===========================================================
%===========================================================
%===========================================================
%===========================================================
%===========================================================
%
%
%  SUPERCOMPACTNESS AND FAILURES OF GCH
%
%===========================================================
%===========================================================%===========================================================
%===========================================================

\chapter[Supercompactness and Failures of $\GCH$]{The Failure of $\GCH$ at a Degree of Supercompactness}\label{chapterthefailure}

\section{Introduction}\label{introduction}

Silver proved that if $\kappa$ is $\kappa^{++}$-supercompact and $\GCH$ holds then there is a cofinality-preserving forcing extension in which $\kappa$ is measurable and $\GCH$ fails at $\kappa$. It was known at the time, by the work of Kunen \cite{Kunen:GCHAtMeasurables}, that a model with a measurable cardinal at which $\GCH$ fails could not be obtained from a mere measurable cardinal. In \cite{Gitik:TheNegationOfTheSingularCardinalHypothesis}, Gitik established the consistency of the existence of a measurable cardinal at which $\GCH$ fails from the existence of a $\kappa$ with $o(\kappa)=\kappa^{++}$. Gitik then proved that this hypothesis was optimal, see \cite{Gitik:TheStrengthOfTheFailureOfSCH}.

Woodin produced a new proof of the consistency of the existence of a measurable cardinal at which $\GCH$ fails, from a hypothesis that is equiconsistent with $o(\kappa)=\kappa^{++}$. That is,
Woodin showed that the existence of a measurable cardinal at which $\GCH$ fails is equiconsistent with the existence of a cardinal $\kappa$ that is $\kappa^{++}$-tall (see \cite{Hamkins:TallCardinals}, \cite{Gitik:TheNegationOfTheSingularCardinalHypothesis}, or \cite{Jech:Book}), where a cardinal $\kappa$ is \emph{$\theta$-tall} if there is a nontrivial elementary embedding $j:V\to M$ with critical point $\kappa$ such that $j(\kappa)>\theta$ and $M^\kappa\subseteq M$ in $V$. In this chapter, I extend Woodin's result into the realm of partially supercompact cardinals. Since $\kappa$ is measurable if and only if $\kappa$ is $\kappa$-supercompact, one immediately sees several natural ways of doing this. Consider the following questions for cardinals $\kappa$, $\lambda$, and $\theta$.

\begin{enumerate}
\item What is the strength of the hypothesis that $\kappa$ is $\lambda$-supercompact and $\GCH$ fails at $\kappa$?
\item What is the strength of the hypothesis that $\kappa$ is $\lambda$-supercompact and $\GCH$ fails at $\kappa$ with $2^\kappa\geq\theta$?
\item What is the strength of the hypothesis that $\kappa$ is $\lambda$-supercompact and $\GCH$ fails at $\lambda$?
\item What is the strength of the hypothesis that $\kappa$ is $\lambda$-supercompact and $\GCH$ fails at $\lambda$ with $2^\lambda\geq\theta$?
\end{enumerate}

Note that Woodin's theorem answers question (1) in the case that $\lambda=\kappa$. 

The following theorem, together with Woodin's result, provides complete answers to questions (1) - (4).

\begin{theorem}\label{theoremdegsupcomp} Suppose $\lambda$ and $\theta$ are cardinals. 
\begin{enumerate}
\item For $\lambda>\kappa$, the existence of a $\lambda$-supercompact cardinal $\kappa$ such that $\GCH$ fails at $\kappa$ is equiconsistent with the existence of a $\lambda$-supercompact cardinal.
\item The existence of a $\lambda$-supercompact cardinal $\kappa$ such that $2^\kappa\geq\theta$ is equiconsistent with the existence of a $\lambda$-supercompact cardinal that is also $\theta$-tall.
\item The existence of a $\lambda$-supercompact cardinal $\kappa$ such that $\GCH$ fails at $\lambda$ is equiconsistent with the existence of a $\lambda$-supercompact cardinal that is $\lambda^{++}$-tall.
\item The existence of a $\lambda$-supercompact cardinal $\kappa$ such that $2^\lambda\geq\theta$ is equiconsistent with the existence of a $\lambda$-supercompact cardinal that is $\theta$-tall.
%\item Suppose $\lambda>\kappa$, $\lambda^{<\lambda}=\lambda$, and $\GCH$ holds. The existence of a $\lambda$-supercompact cardinal $\kappa$ such that $\GCH$ holds on $[\kappa,\lambda)$ and fails at $\lambda$ is equiconsistent with the existence of a $\lambda$-supercompact cardinal that is $\theta$-tall.
\end{enumerate}
In each case above, the term ``equiconsistent'' is intended to mean that, in the forward direction the same cardinal witnessing the hypothesis also witnesses the conclusion; and in the reverse direction, the same cardinal witnessing the hypothesis witnesses the conclusion in a forcing extension.

%In (2) - (4) the term ``equiconsistent'' is to be understood in the sense that the former implies the latter and the latter implies the former in a forcing extension.
\end{theorem}

%In fact, we get more than what is stated in the theorem. We will show that in the forward direction, the same cardinal witnessing the hypothesis also witnesses the conclusion; and, in the reverse direction, the same cardinal witnessing the hypothesis witnesses the conclusion in a forcing extension.

The details of cardinal preservation in the various forcing extensions in parts (1) - (4) of the main theorem will be worked out below.

Questions (1) - (4) above can be seen as a special case to a more general question: 
\begin{enumerate}
\item[(5)] What kind of $\GCH$ patterns are consistent with a $\lambda$-supercompact cardinal from what type of large cardinal assumption? 
\end{enumerate}
There are some obvious restrictions, such as if $\GCH$ fails at $\kappa$, a $\lambda$-supercom\-pact cardinal, then it must fail unboundedly often below $\kappa$. Also, if $\lambda$ is a strong limit and $\GCH$ holds below and at $\kappa$ then $\GCH$ must hold up to $\lambda$. There are some more subtle issues in answering question (5) as well.

The backward direction of Theorem \ref{theoremdegsupcomp}(2) - (4) will be established by using forcing that achieves $2^\kappa>\lambda^+$, and hence $2^\lambda>\lambda^+$, and preserves the $\lambda$-supercompactness of $\kappa$, where $\lambda>\kappa$ is a cardinal. This suggests the question, can one force a violation of $\GCH$ at $\lambda$ while preserving $\GCH$ in the interval $[\kappa,\lambda)$ and preserving the $\lambda$-supercompactness of $\kappa$? It seems as though the method of surgical modification of a generic, due to Woodin, does not generalize to answer this question. However, Friedman and Honsik show in their forthcoming paper \cite{FriedmanHonsik:SupercompactnessAndFailuresOfGCH} that the answer to the previous question is yes by using generalized Sacks forcing and the tuning fork method.

Let me now give an outline of the rest of the chapter. I will prove Theorem \ref{theoremdegsupcomp}(1) in Section \ref{easysection}. In Section \ref{preliminariessection}, in order to prepare for the proof of Theorem \ref{theoremdegsupcomp}(2) - (4), I discuss the large cardinal concept of ``tallness with closure,'' which synthesizes the concepts of $\lambda$-supercompactness and $\theta$-tallness. I prove Theorem \ref{theoremdegsupcomp}(2) - (4) in Section \ref{mainproofsection}.

\section{Proof of Theorem \ref{theoremdegsupcomp}(1)}\label{easysection}

The proof of Theorem \ref{theoremdegsupcomp}(1) will use a preparatory forcing notion called the lottery preparation, which was introduced by Hamkins in \cite{Hamkins:TheLotteryPreparation}. The lottery preparation works uniformly as a generalized Laver preparation in a variety of large cardinal contexts. Here I give a brief introduction to the lottery preparation.

The \emph{lottery sum} of a collection of posets $\{(\Q_\alpha,\leq_\alpha)\mid\alpha{<}\kappa\}$ is
$$\bigoplus\{\Q_\alpha\mid\alpha{<}\kappa\}:=\{\1 \}\cup\bigcup_{\alpha{<}\kappa}\{(\alpha,q)\mid q\in \Q_\alpha\}$$
where the ordering on the lottery sum is defined by (1) $(\alpha,q)\leq\emptyset$ for all $\alpha{<}\kappa$ and $q\in \Q_\alpha$ and (2) $(\alpha,p)\leq(\beta,q)$ if and only if $\alpha=\beta$ and $p\leq_\alpha q$. As Hamkins says, a generic for the lottery sum of a collection of posets chooses a poset and forces with it. For a detailed account of the lottery preparation see \cite{Hamkins:TheLotteryPreparation}.

A poset $\Q$ is said to be \emph{allowed} at stage $\gamma$ if $\Q$ is ${<}\gamma$-strategically closed; note that ``${<}\gamma$-strategic closure'' will not play a role in the arguments to come, so the reader who is unfamiliar with this concept may take this to simply mean ${<}\gamma$-closed. For a partial function $f\subseteq\kappa\times\kappa$ the \emph{lottery preparation of $\kappa$ with respect to $f$} is defined to be the Easton support forcing iteration of length $\kappa$, such that if $\gamma{<}\kappa$ is inaccessible and $f"\gamma\subseteq\gamma$, then the stage $\gamma$ forcing is the lottery sum in $V^{\P_\gamma}$ of all allowed posets in $H(f(\gamma)^+)^{V^{\P_\gamma}}$ and otherwise the stage $\gamma$ forcing is trivial.  Suppose $\P$ is the lottery preparation of $\kappa$ with respect to a partial function $f\subseteq\kappa\times\kappa$. Let $p=\langle p_\alpha\mid\alpha{<}\kappa\rangle\in \P$ be the condition such that $p_\alpha=\emptyset$ for $\alpha\neq\gamma$ and $p_\gamma= 1_\Q$, where $1_\Q$ is the top element of $\Q$. Since forcing below $p$ provides a $V^{\P_\gamma}$-generic for $\Q$, Hamkins says that the condition $p$ \emph{opts} for $\Q$ at stage $\gamma$.

% We say that a condition $p\in \P$ \emph{opts} for a forcing notion $\Q$ at stage $\gamma{<}\kappa$ if $\forces_{\P_\gamma} p_\gamma\in \Q$ where $p=\langle p_\alpha\mid\alpha{<}\kappa\rangle$.

The lottery preparation $\P$ of some large cardinal $\kappa$ is usually used with respect to a partial function $f\subseteq\kappa\times\kappa$ with the Menas property, such as a function added by fast function forcing (see Section \ref{fffsection}). Using the lottery preparation with respect to such a function insures that $j(\P)$, where $j$ is an elementary embedding witnessing the large cardinal property at hand, has a tail with a high degree of closure.

I will now show that given a $\lambda$-supercompact cardinal $\kappa$, one may pump up the power set of $\kappa$ to have size at least $\lambda^+$ while maintaining the $\lambda$-supercompactness of $\kappa$. This will establish Theorem \ref{theoremdegsupcomp}(1) because $\lambda>\kappa$ and $2^\kappa\geq\lambda^+$ trivially implies $2^\kappa\geq\kappa^{++}$.

% if we assume $\lambda>\kappa$ then in a forcing extension we will have $2^\kappa\geq\lambda^+\geq\kappa^{++}$. We note that this result essentially follows from the methods of \cite{ApterHamkins:IndestructibilityAndTheLevelByLevelAgreement}.

\begin{theorem}\label{easy}
If $\kappa$ is $\lambda$-supercompact then there is a forcing extension preserving this in which $2^\kappa\geq\lambda^+$.
\end{theorem}

\begin{proof} 
Since every $\lambda$-supercompactness embedding is a $\lambda^{<\kappa}$-supercompactness embedding, it follows that one may assume that $\lambda^{{<}\kappa}=\lambda$. I also will assume $2^\kappa\leq\lambda$ since otherwise the theorem is trivial. Further, I assume that $2^\lambda=\lambda^+$ since this can be forced using ${\leq}\lambda$-distributive forcing, which easily preserves the $\lambda$-supercompactness of $\kappa$. Let $f\subseteq\kappa\times\kappa$ be a partial function and let $j:V\to M$ witness that $\kappa$ is $\lambda$-supercompact such that $j(f)(\kappa)>\lambda$ where
$$M=\{j(h)(j"\lambda)\mid h:P_\kappa\lambda\to V, h\in V\}.$$ Note that $(\lambda^+)^M=\lambda^+$ since $M$ is closed under $\lambda$ sequences in $V$. Thus, since $j(\kappa)$ is inaccessible in $M$ it follows that $j(\kappa)>\lambda^+$. Now let $\P$ be the lottery preparation of $\kappa$ relative to $f$. Let $G$ be $V$-generic for $\P$. Let $\Q:=\Add(\kappa,\lambda^+)^{V[G]}$ and let $H$ be $V[G]$-generic for $\Q$. 

By elementarity, $j(\P)$ is the lottery preparation of $j(\kappa)$ defined relative to $j(f)$. Since $M$ is closed under $\lambda$-sequences in $V$, the first $\kappa$ stages of $\P$ and $j(\P)$ agree. Furthermore, $j(f)\restrict\kappa=f$ implies that $j(f)"\kappa\subseteq\kappa$, and since $\kappa$ is inaccessible in $M$, it follows that the stage $\kappa$ forcing in $j(\P)$ is the lottery sum in $M[G]$ of all allowed posets in $H(j(f)(\kappa)^+)^{M[G]}$. Since $\Q$ is in $M[G]$ and also in $H(j(f)(\kappa)^+)^{M[G]}$, it follows that $\Q$ appears in the stage $\kappa$ lottery sum in $j(\P)$. Thus, $j(\P)$ factors below a condition $p$ that opts for $\Q$ at stage $\kappa$ as $j(\P)\restrict p\cong \P * \dot\Q * \dot\P_{tail}$ where $\dot\P_{tail}$ is a term for the forcing $j(\P)$ beyond stage $\kappa$. For example, $p$ could be the condition $\langle p_\alpha \mid \alpha<j(\kappa)\rangle$ such that $p_\kappa=1_\Q$ and $p_\alpha=\emptyset$ for every other $\alpha<j(\kappa)$. The next stage of nontrivial forcing in $j(\P)$ is beyond $\lambda$ since $j(f)(\kappa)>\lambda$. From this it follows that $\dot\P_{tail}$ is a $\P*\dot\Q$-name for ${\leq}\lambda$-closed forcing. Since $M\subseteq V$, it follows that $G$ is $M$-generic for $\P$ and $H$ is $M[G]$-generic for $\Q$. Thus $\P_{tail}$ is ${\leq}\lambda$-closed in $M[G][H]$. Furthermore, it follows from Lemma \ref{lemmachain} that $M[G][H]$ is closed under $\lambda$-sequences in $V[G][H]$ because $\P*\dot\Q$ is $\kappa^+$-c.c.. Since in $V$, $\P$ has at most $2^\kappa\leq\lambda$-many dense subsets, it follows that $\P_{tail}$ has at most $j(\lambda)$-many dense subsets in $M[G][H]$ where $|j(\lambda)|^V\leq (\lambda)^{\lambda^{{<}\kappa}}=\lambda^\lambda=2^\lambda=\lambda^+$. Thus, one can see that there is an $M[G][H]$-generic $G_{tail}$ for $\P_{tail}$ in $V[G][H]$ by building a descending sequence of conditions. Furthermore, $j"G\subseteq G*H*G_{tail}$, since each condition in $\P$ has support bounded below the critical point of $j$. This implies that the embedding lifts to $j:V[G]\to M[j(G)]$ where $j(G)=G*H*G_{tail}$ and the lifted embedding is a class of $V[G][H]$. It follows from Lemma \ref{lemmaground} that $M[j(G)]$ is closed under $\lambda$ sequences in $V[G][H]$.

Now it is shown, using the methods of \cite[Corollary 10]{ApterHamkins:IndestructibilityAndTheLevelByLevelAgreement}, that the embedding lifts through $\Q$. Let $A\subseteq j(\Q)=\Add(j(\kappa),j(\lambda^+))$ be a maximal antichain in $M[j(G)]$. Let $r\in j(\Q)$ be a condition that is compatible with every element of $j"H$. I will argue that there is a condition $r'\leq r$ deciding $A$ that is still compatible with every element of $j"H$. Since $j(\Q)$ is $j(\kappa^+)$-c.c. it follows that $|A|\leq j(\kappa)$ in $M[j(G)]$. Furthermore, $\sup j"\lambda^+=j(\lambda^+)$ follows from the fact that $\cf(\lambda^+)>\lambda$, and this implies $A\subseteq \Add(j(\kappa),j(\alpha))$ for some $\alpha<\lambda^+$. Fix such an $\alpha$ so that also $r\in \Add(j(\kappa),j(\alpha))$. Let $q=\bigcup(j"(H\cap\Add(\kappa,\alpha)))$. Since $j(p)=j"p$ for $p\in \Add(\kappa,\alpha)$ one has $|q|\leq\lambda<j(\kappa)$ and thus $q\in M[j(G)]$ is a master condition in $\Add(j(\kappa),j(\alpha))$ (which is a complete subposet of $\Add(j(\kappa),j(\lambda^+))$). Now since $r$ is compatible with every element of $j"H$ it follows that $r$ and $q$ are compatible in $\Add(j(\kappa),j(\alpha))$.  Choose $r'\in \Add(j(\kappa),j(\alpha))$ below $r$ and $q$ deciding $A$. Let me show that $r'$ remains compatible with $j"H$. Consider $j(p)$ for $p\in H$. One may split $p$ into two pieces: $p=p_0\cup p_1$ where $\dom(p_0)\subseteq \alpha\times\kappa$ and $\dom(p_1)\subseteq [\alpha,\lambda^+)\times\kappa$. Then $j(p)=j(p_0)\cup j(p_1)$ where the domain of $j(p_1)$ is disjoint from the domain of any element of $\Add(j(\kappa),j(\alpha))$. Thus, $r'$ is compatible with $j(p_1)$ in $\Add(j(\kappa),j(\lambda^+))$. Furthermore, $r'\leq q\leq j(p_0)$ and hence $r'$ is compatible with $j(p)$.

% Note in the above, $|q|\leq\lambda$ because for $p\in \Add(\kappa,\alpha)$ we have $j(p)=j"p$ and the set $j"(H\cap\Add(\kappa,\alpha))$ has size $\leq\lambda$.

Since $\Q$ has $\lambda^+$-many antichains, the above procedure may be iterated to choose a decreasing sequence of conditions in $V[G][H]$, meeting all the antichains of $\Add(j(\kappa),j(\lambda^+))$, such that each element of the sequence is compatible with $j"H$. Let $j(H)$ be the filter generated by this sequence. Then $j(H)$ is an $M[j(G)]$-generic filter for $\Add(j(\kappa),j(\lambda^+))$ with $j"H\subseteq j(H)$. Hence the embedding lifts to $j:V[G][H]\to M[j(G)][j(H)]$ in $V[G][H]$, which implies that $\kappa$ is $\lambda$-supercompact in $V[G][H]$. 
\end{proof}

\section{Tallness with Closure}\label{preliminariessection}

\subsection{Definitions and Basic Facts}\label{tallness}

Here I include some basic definitions and results about $\theta$-tall cardinals, and $\theta$-tall cardinals with closure $\lambda$, where $\lambda$ is some cardinal and $\theta$ is an ordinal. Both of these large cardinal notions are defined in \cite{Hamkins:TallCardinals}. A cardinal $\kappa$ is called \emph{$\theta$-tall} if there is a nontrivial elementary embedding $j:V\to M$ with critical point $\kappa$ such that $j(\kappa)>\theta$ and $M^\kappa\subseteq M$. Woodin and Gitik used such cardinals to determine the strength of the failure of $\GCH$ at a measurable cardinal (see \cite{Gitik:TheNegationOfTheSingularCardinalHypothesis}), and Hamkins has studied them in their own right in \cite{Hamkins:TallCardinals}. Hamkins says that $\kappa$ is \emph{$\theta$-tall with closure $\lambda$} if there is an elementary embedding $j:V\to M$ with $\cp(j)=\kappa$, $j(\kappa)>\theta$, and $M^\lambda\subseteq M$ in $V$.
By composing embeddings, one can see that a cardinal $\kappa$ is $\theta$-tall and $\lambda$-supercopmact if and only if it is $\theta$-tall with closure $\lambda$.

The following lemma will be required below.

\begin{lemma}\label{generators}
If $\kappa$ is $\theta$-tall with closure $\lambda$ then there is an embedding witnessing this $j:V\to M$ such that $$M=\{j(h)(j"\lambda,\alpha)\mid \alpha\leq\delta \textrm{ and } h:P_\kappa\lambda\times\kappa\to V \textrm{ is a function in $V$}\}$$ where $\delta=(\theta^\lambda)^M$. 
\end{lemma}

\begin{proof}
Let $j_0:V\to M_0$ witness the $\theta$-tallness with closure $\lambda$ of $\kappa$ and let $$X=\{j_0(h)(j_0"\lambda,\alpha)\mid \alpha\leq\delta \textrm{ and } h:P_\kappa\lambda\times\kappa\to V \textrm{ with } h\in V\}$$ where $\delta:=(\theta^\lambda)^M$. By Tarski-Vaught, it follows that $X\elesub M_0$. Let $\pi:X\to M$ be the Mostowski collapse of $X$ and define an elementary embedding $j:V\to M$ by $j= \pi \circ j_0$ an let $k:=\pi^{-1}:M\to X\subseteq M_0$. It follows that $j$ is the desired embedding.
\end{proof}

I will often make use of the easy fact that if $\kappa$ is $\theta$-tall with closure $\lambda$, then it is $\theta^\lambda$-tall with closure $\lambda^{{<}\kappa}$, which I demonstrate now. Suppose $j:V\to M$ witnesses the $\theta$-tallness with closure $\lambda$ of $\kappa$. I will argue that $j$ actually witnesses the $\theta^\lambda$-tallness with closure $\lambda^{{<}\kappa}$ of $\kappa$. If $\sigma\in P_\kappa\lambda$ then $j(\sigma)=j"\sigma$, and from this it follows that
$j"P_\kappa\lambda= P_\kappa(j"\lambda)$. Thus $j"P_\kappa\lambda\in M$, since $j"\lambda\in M$. By using a bijection from $P_\kappa\lambda$ to $\lambda^{{<}\kappa}$, it is routine to verify that $j"\lambda^{{<}\kappa}\in M$ and this implies that $M^{\lambda^{{<}\kappa}}\subseteq M$. Since $M^{\lambda^{{<}\kappa}}\subseteq M$, it follows that $(^{\lambda^{{<}\kappa}}\theta)^M=\,^{\lambda^{{<}\kappa}}\!\theta$. Furthermore, $j(\kappa)$ is inaccessible in $M$ and hence $\theta^{\lambda^{{<}\kappa}}\leq(\theta^{\lambda^{{<}\kappa}})^M<j(\kappa)$. This proves the following.

%Suppose $\kappa$ is $\lambda$-supercompact and $\theta$-tall in $V$. A standard argument shows that $\kappa$ is, in fact, $\lambda^{{<}\kappa}$-supercompact. Let $i:V\to N$ witness the $\lambda^{{<}\kappa}$-supercompactness of $\kappa$. Since $\kappa$ is $\theta$-tall in $V$, it follows that $\kappa$ is, in fact, $\theta^\kappa$-tall (see \cite[Lemma 2.1]{Hamkins:TallCardinals}). By elementarity, $i(\kappa)$ is $i(\theta^\kappa)$-tall. Let $k:N\to M$ be an $i(\theta^\kappa

%If $\kappa$ is $\lambda$-supercopmact it follows that $\kappa$ must be $\lambda^{{<}\kappa}$-supercompact. By composing embeddings we obtain a $j:V\to M$ witnessing that $\kappa$ is $\theta$-tall with closure $\lambda^{{<}\kappa}$. Since $j(\kappa)$ is inaccessible in $M$ and $M^{\lambda^{{<}\kappa}}\subseteq M$, we have $\theta^\lambda \leq (\theta^\lambda)^M <j(\kappa)$. Thus $j$ witnesses that $\kappa$ is $\theta^\lambda$-tall with closure $\lambda^{{<}\kappa}$.

By the remarks in the previous paragraph, given that $\kappa$ is $\theta$-tall with closure $\lambda$, in many arguments we will be able to assume without loss of generality that $\theta^\lambda=\theta$ and $\lambda^{{<}\kappa}=\lambda$. Then by Lemma \ref{generators} there is an embedding $j:V\to M$ witnessing that $\kappa$ is $\theta$-tall with closure $\lambda$ such that $$M=\{j(h)(j"\lambda,\alpha)\mid\textrm{$\alpha\leq\theta$ and $h:P_\kappa\lambda\times\kappa\to V$ is a function in $V$}\}.$$

\subsection{Fast Function Forcing and Tallness with Closure}\label{fffsection}

The goal of this section will be to prove that one can force to add a function with the Menas property with respect to $\theta$-tallness with closure $\lambda$. In other words, it is shown that if $j:V\to M$ witnesses that $\kappa$ is $\theta$-tall with closure $\lambda$, then one may force to add a partial function $f\subseteq \kappa\times\kappa$ with the property $j(f)(\kappa)>\theta$ and that $j$ lifts to $j:V[f]\to M[j(f)]$. In fact, one can arrange that $j(f)(\kappa)$ is equal to any ordinal less than $j(\kappa)$, the degree of tallness of $\kappa$. To accomplish this I will use a technique invented by Woodin called fast function forcing.

For a cardinal $\kappa$, define the \emph{fast function forcing} poset $\F_\kappa$ as follows. Conditions in $\F_\kappa$ are partial functions $p\subseteq\kappa\times\kappa$ such that the following two conditions hold.
\begin{enumerate}
\item Each $\gamma\in \dom(p)$ is inaccessible and $p"\gamma\subseteq\gamma$.
\item If $\gamma{\leq}\kappa$ is inaccessible then $|p\restrict\gamma|<\gamma$.
\end{enumerate}
The ordering on $\F_\kappa$ is given by $p\leq q$ if and only if $p\supseteq q$. For a fixed condition of the form $p:=\{(\gamma,\delta)\}$  the poset $\F_\kappa$ factors below $p$ as $\F_\kappa\restrict p \cong \F_{\gamma}\times\F_{[\lambda,\kappa)}$ where $\lambda$ is the next inaccessible beyond $\max(\gamma,\delta)$ and $\F_{[\lambda,\kappa)}:=\{p\in \F_\kappa\mid \dom(p)\subseteq [\lambda,\kappa)\}$. A generic $G$ for $\F_\kappa$ provides a partial function $f:=\bigcup G$ from $\kappa$ to $\kappa$. Since $f$ determines $G$, the forcing extension by the fast-function-forcing poset will be written as $V[f]$ from this point forward. For a more detailed account of fast-function-forcing see \cite{Hamkins:TheLotteryPreparation}.

\begin{lemma}\label{menaslemma}
Suppose $j:V\to M$ is a $\theta$-tallness embedding with closure $\lambda$ with critical point $\kappa$ where $\lambda\leq\theta$ (or merely $\lambda$ is less than the first inaccessible beyond $\theta$). Then there is a fast function forcing extension $V[f]$ such that $j$ lifts to $j:V[f]\to M[j(f)]$ witnessing the $\theta$-tallness with closure $\lambda$ in $V[f]$ such that $j(f)(\kappa)>\theta$. Furthermore, for any $\delta<j(\kappa)$ there is such a lift $j$ such that $j(f)(\kappa)=\delta$.
\end{lemma}

\begin{proof}
As mentioned at the end of Subsection \ref{tallness}, it can be assumed without loss of generality that $\lambda^{{<}\kappa}=\lambda$ and $\theta^\lambda=\theta$. It can further be assumed that $2^\lambda=\lambda^+$, since this can be accomplished using ${\leq}\lambda$-distributive forcing, which preserves the $\theta$-tallness with closure $\lambda$ of $\kappa$ by Lemma \ref{lemmalambdadist}. Let $f$ be $V$-generic for $\F_\kappa$ and let $j:V\to M$ be a $\theta$-tallness embedding with closure $\lambda$ such that $$M=\{j(h)(j"\lambda,\alpha)\mid\textrm{$\alpha\leq\theta$ and $h:P_\kappa\lambda\times\kappa\to V$ is a function in $V$}\}.$$ Let $\delta$ be an ordinal with $\theta<\delta<j(\kappa)$ and let $p:=\{(\kappa,\delta)\}$. The poset $j(\F_\kappa)$ factors below the $p$ as $j(\F_\kappa)\restrict p\cong \F_\kappa\times \F_{[\gamma,j(\kappa))}$ where $\gamma$ is the next inaccessible cardinal above $\delta$.

An $M$-generic filter for $\F_{[\gamma,j(\kappa))}$ will be constructed in $V$. Let $D$ be a dense subset of $\F_{[\gamma,j(\kappa))}$ in $M$. Then $D=j(h_{D})(j"\lambda,\alpha)$ for some $\alpha\leq\theta$ and $h_{D}:P_\kappa\lambda\times\kappa\to V$. Since $j"\lambda$ and $j(h_D)$ are both in $M$ it follows that $\vec{D}:=\langle j(h_{{D}})(j"\lambda,\alpha)\mid\alpha\leq\theta\rangle\in M$. Using the ${\leq}\theta$-closure of $\F_{[\gamma,j(\kappa))}$ in $M$ one may find a single condition in $\F_{[\gamma,j(\kappa))}$ meeting every dense set mentioned by $\vec{D}=\langle j(h_{{D}})(j"\lambda,\alpha)\mid\alpha\leq\theta\rangle$.

Now assume without loss of generality that 
$$h_{{D}}:P_\kappa\lambda\times\kappa\to\{\textrm{dense subsets of a tail of $\F_\kappa$}\}.$$ Since $|\F_\kappa|=\kappa$ it follows that there are $2^\kappa$  dense subsets of a tail of $\F_\kappa$. This implies that there are $(2^\kappa)^{\lambda^{{<}\kappa}}=2^\lambda=\lambda^+$ functions $h$ with domain $P_\kappa\lambda\times\kappa$ that represent dense subsets of a tail of $\F_\kappa$. In $V$, one may enumerate such $h$'s as $\vec{h}:=\langle h_\xi\mid\xi<\lambda^+\rangle$. Since every dense subset of $\F_{[\gamma,j(\kappa))}$ in $M$ is represented by a function $h_\xi$ on the list, an $M$-generic filter for $\F_{[\gamma,j(\kappa))}$ can be constructed in $V$ as follows. At successor stages $\xi$, by using the ${\leq}\theta$-closure of $\F_{[\gamma,j(\kappa))}$ in $M$, one can find a single condition $p_\xi\in \F_{[\gamma, j(\kappa))}$ below all previously constructed conditions meeting each dense set of the form $j(h_\xi)(j"\lambda,\alpha)$ for $\alpha\leq\theta$. At limit stages one may use the fact that $M^\lambda\subseteq M$ is ${\leq}\lambda$-closed in $V$ to find a condition below all previously constructed conditions. This defines a descending sequence of conditions in $V$. Let $f_{[\gamma,j(\kappa))}$ be the $M$-generic filter for $\F_{[\gamma,j(\kappa))}$ generated by the sequence. Since $f_{[\gamma,j(\kappa)}\in V$ and $f$ is $V$-generic for $\F_\kappa$, it follows from the product forcing lemma that $f\cup p \cup f_{[\gamma,j(\kappa))}$ is $M$-generic for $j(\F_\kappa)\restrict p$. Since $j"f\subseteq f\cup p\cup f_{[\gamma,j(\kappa))}$ the embedding lifts to $j:V[f]\to M[j(f)]$ where $j(f)=f\cup p \cup f_{[\gamma,j(\kappa))}$ and $j$ is a class of $V[f]$. Since $\F_\kappa$ is $\kappa$-c.c. and $f_{[\gamma,j(\kappa))}$ is in $V$, it follows by Lemmas \ref{lemmaground} and \ref{lemmachain} that $M[j(f)]$ is closed under $\lambda$-sequences in $V[f]$ and hence that the lifted embedding witnesses that $\kappa$ is $\theta$-tall with closure $\lambda$ in $V[f]$.
\end{proof}

\subsection{The Lottery Preparation and Tallness with Closure}

In \cite{Hamkins:TheLotteryPreparation}, Hamkins shows that the lottery preparation makes many large cardinals indestructible by a wide array of forcing notions. Here I will extend the results in \cite{Hamkins:TheLotteryPreparation} to include $\theta$-tallness with closure $\lambda$.

% We need $\lambda\geq\theta$ in the following because we need $\Q\in H(j(f)(\kappa)^+)$.

\begin{theorem}\label{theoremindestructibility}
Suppose $\kappa$ is $\theta$-tall with closure $\lambda$ where $\lambda\leq\theta$. Then after the lottery preparation of $\kappa$, the $\theta$-tallness with closure $\lambda$ of $\kappa$ is indestructible by ${<}\kappa$-directed closed forcing of size $\leq\lambda$.
\end{theorem}

\begin{proof}
Suppose $j:V\to M$ witnesses the $\theta$-tallness with closure $\lambda$ of $\kappa$. As before, without loss of generality assume that $\lambda^{{<}\kappa}=\lambda$, $\theta^\lambda=\theta$, $2^\lambda=\lambda^+$, and $$M=\{j(h)(j"\lambda,\alpha)\mid\textrm{$\alpha\leq\theta$ and $h:P_\kappa\lambda\times\kappa\to V$ is a function in $V$}\}.$$
Let me remark that the forcing to obtain $2^\lambda=\lambda^+$ collapses cardinals in the interval $[\lambda^+,2^\lambda]$ to $\lambda^+$. By Lemma \ref{menaslemma}, assume that there is a fast function $f\subseteq\kappa\times\kappa$ with $j(f)(\kappa)>\theta$. Let $\P$ be the lottery preparation defined relative to $f$ and let $G$ be $V$-generic for $\P$. Let $\Q$ be any ${<}\kappa$-directed closed forcing of size less than or equal to $\lambda$ in $V[G]$ and let $H$ be $V[G]$-generic for $\Q$.

Since $\Q$ could be trivial forcing it will suffice to lift $j$ to $V[G][H]$ in $V[G][H]$. Assume without loss of generality that $\Q\subseteq \ORD$. Since $|\P|^V=\kappa$ it follows from Lemma \ref{lemmachain} that $M[G]^\lambda\subseteq M[G]$ in $V[G]$ and hence $\Q\in M[G]$. By elementarity, $j(\P)$ is the lottery preparation of $j(\kappa)$ with respect to $j(f)$. Since $M$ is closed under $\lambda$-sequences in $V$, it follows that the first $\kappa$ stages in $\P$ and $j(\P)$ are the same. Since $\lambda\leq\theta$, one concludes that $\Q\in H(j(f)(\kappa)^+)^{M[G]}$ and thus $\Q$ appears in the lottery sum at stage $\kappa$ in $j(\P)$. Thus $j(\P)$ factors below a condition $p$ that opts for $\Q$ at stage $\kappa$ as $j(\P)\restrict p\cong \P*\dot\Q*\dot\P_{tail}$, where $\dot\P_{tail}$ is a term for the iteration beyond stage $\kappa$. It follows that $\dot\P_{tail}$ is a term for $\leq\theta$-closed forcing because $j(f)(\kappa)>\theta$. Since $|\Q|^{V[G]}\leq\lambda$, it follows that $M[G][H]^\lambda\subseteq M[G][H]$ in $V[G][H]$. I will show that one can construct an $M[G][H]$-generic for $\P_{tail}$ in $V[G][H]$. Let $D$ be a dense subset of $\P_{tail}$ in $M[G][H]$. Let $\dot{D}\in M$ be a $\P*\Q$-name for $D$, that is $\dot{D}_{G*H}=D$, and let $\dot{D}=j(h_{\dot{D}})(j"\lambda,\alpha)$ where $h_{\dot{D}}:P_\kappa\lambda\times\kappa\to V$, $h_{\dot{D}}\in V$, and $\alpha\leq\theta$. Since $j(h_{\dot{D}}),j"\lambda\in M$ the sequence of names $\vec{D}:=\langle j(h_{\dot{D}})(j"\lambda,\alpha)\mid\alpha\leq\theta\rangle$ is in $M$, and furthermore the sequence of dense subsets of $\P_{tail}$, $\vec{D}_{G*H}:=\langle j(h_{\dot{D}})(j"\lambda,\alpha)_{G*H}\mid\alpha\leq\theta\rangle$, is in $M[G][H]$. Since $\P_{tail}$ is $\leq\theta$-closed in $M[G][H]$ there is a single condition below every dense set mentioned by $\vec{D}_{G*H}$. Without loss of generality, assume that the range of $h_{\dot{D}}$ is contained in the set of nice names for dense subsets of a tail of $\P$. Working in $V[G][H]$, I will put a bound on the number of functions $$h:P_\kappa\lambda\times\kappa\to \{\textrm{nice names for dense subsets of a tail of $\P$}\}.$$ Since $|\P|=\kappa$ there are $2^\kappa$-many nice names for subsets of a tail of $\P$. Thus there are at most $(2^\kappa)^{\lambda^{{<}\kappa}}=2^\lambda=\lambda^+$-many such $h$'s. 
In $V[G][H]$ one may enumerate all such $h$'s as $\langle h_\xi\mid\xi<\lambda^+\rangle.$ Since every dense subset of $\P_{tail}$ in $M[G][H]$ has a nice $\P*\Q$-name and each nice $\P*\Q$-name is represented by one of the $h_\xi$'s on our list, one may construct an $M[G][H]$-generic descending sequence of conditions of $\P_{tail}$ as follows. At successor stages $\xi$, one work in $M[G][H]$ and use the fact that $\P_{tail}$ is $\leq\theta$-closed in $M[G][H]$ to find a condition of $\P_{tail}$ meeting every dense subset of $\P_{tail}$ which has a name on the list $\langle j(h_\xi)(j"\lambda,\alpha)\mid\alpha\leq\theta\rangle$. At limits $\xi<\lambda^+$, since $M[G][H]$ is closed under $\lambda$-sequences in $V[G][H]$ it follows that $\P_{tail}$ is $\leq\lambda$-closed in $V[G][H]$, and hence, in $V[G][H]$, there is a condition of $\P_{tail}$ below all previously constructed conditions. This defines a descending $\lambda^+$-sequence of conditions in $\P_{tail}$. Let $G_{tail}$ be the filter generated by this sequence of conditions. Clearly $G_{tail}\in V[G][H]$ is an $M[G][H]$-generic filter for $\P_{tail}$. Thus the embedding lifts to $j:V[G]\to M[j(G)]$ where $j(G):=G*H*G_{tail}$ and $j$ is a class of $V[G][H]$. Since $\P*\Q$ is $\lambda^+$-c.c. it follows from Lemma \ref{lemmachain} that $M[G][H]^\lambda\subseteq M[G][H]$ in $V[G][H]$. Furthermore, since $G_{tail}\in V[G][H]$, it follows from Lemma \ref{lemmaground} that $M[j(G)]$ is closed under $\lambda$-seqences in $V[G][H]$.

Now I will show that $j$ lifts to $V[G][H]$. Since $H$ and $j"\Q$ are both in $M[j(G)]$ it follows that $j"H$ is in $M[j(G)]$. Since $j(\Q)$ is ${<}j(\kappa)$-directed closed in $M[j(G)]$ it follows that in $M[j(G)]$, there is a master condition $r\in j(\Q)$ below each element of $j"H$. Let me now construct an $M[j(G)]$-generic filter for $j(\Q)$ in $V[G][H]$. Let $D\in M[j(G)]$ be a dense subset of $j(\Q)$. Then $D$ is of the form $D=j(h_D)(j"\lambda,\alpha)$ where $h_D\in V[G]$ is a function from $P_\kappa\lambda\times\kappa$ to the collection of dense subsets of $\Q$ and $\alpha\leq\theta$. Now let $\vec{D}:=\langle j(h_D)(j"\lambda,\alpha)\mid\alpha\leq\theta\rangle$. Since $j"\lambda$ and $j(h_D)$ are in $M[j(G)]$, it follows that $\vec{D}\in M[j(G)]$. Since $j(\Q)$ is ${<}j(\kappa)$-directed closed in $M[j(G)]$, one can find, via an internal argument in $M[j(G)]$, a single condition that meets every dense set mentioned by $\vec{D}$. In $V[G]$, $|\Q|=\lambda$ and this implies that there are at most $(2^\lambda)^{\lambda^{{<}\kappa}}=\lambda^+$-many functions $h\in V[G]$ that represent dense subsets of $j(\Q)$ in $M[j(G)]$. As before, one can enumerate these functions as $\langle h_\xi\mid\xi<\lambda^+\rangle$ and define a descending sequence of conditions meeting every dense subset of $j(\Q)$. Start the descending sequence with the master condition, $r$. If $\xi$ is a successor, use the ${<}j(\kappa)$-directed closure of $j(\Q)$ in $M[j(G)]$ to meet all dense sets mentioned in $\langle j(h_\xi)(j"\lambda,\alpha)\mid\alpha\leq\theta\rangle$ with a single condition that is also below $r$. At limit stages $\xi<\lambda^+$, since $M[j(G)]$ is closed under $\lambda$-sequences in $V[G][H]$, it follows that $j(\Q)$ is ${\leq}\lambda$-closed in $V[G][H]$, and hence one may find a condition of $j(\Q)$ below all previously constructed conditions. This defines a descending $\lambda^+$-sequence of conditions below the master condition $r$. Let $j(H)$ be the generic filter generated by this sequence. Since $r$ is stronger than every element of $j"H$ and $r\in j(H)$ it follows that $j"H\subseteq j(H)$ and thus the embedding lifts to $j:V[G][H]\to M[j(G)][j(H)]$, where the lifted embedding is a class in $V[G][H]$.

This shows that the $\theta$-tallness with closure $\lambda$ of $\kappa$ is indestructible by any ${<}\kappa$-directed closed forcing of size ${\leq}\lambda$ in $V[G]$.
\end{proof}

Let me now give a quick application of Theorem \ref{theoremindestructibility}. The following corollary can be proven using a master condition argument, but it also follows directly from Theorem \ref{theoremindestructibility}.

\begin{corollary}\label{GCHoninterval}
Suppose $\GCH$ holds and $\kappa$ is $\theta$-tall with closure $\lambda$ where $\lambda\leq\theta$. Then there is a forcing extension in which $\kappa$ is $\theta$-tall with closure $\lambda$ and $2^\kappa=\lambda$.
\end{corollary}

\section{Proof of Theorem \ref{theoremdegsupcomp}(2) - (4)}\label{mainproofsection}

Let me now argue that the equiconsistencies in the forward directions in Theorem \ref{theoremdegsupcomp}(2) - (4) are actually implications. For Theorem \ref{theoremdegsupcomp}(2), suppose $j:V\to M$ witnesses that $\kappa$ is $\lambda$-supercompact and $2^\kappa\geq\theta$. Since $j(\kappa)$ is inaccessible in $M$ it follows that $\theta\leq2^\kappa\leq (2^\kappa)^M<j(\kappa)$. Hence $j$ is a $\theta$-tallness embedding with closure $\lambda$. The forward directions in Theorem \ref{theoremdegsupcomp}(3) and (4) are similar.

It remains to prove the backward directions of Theorem \ref{theoremdegsupcomp}(2) - (4). To do this, start with an embedding $j:V\to M$ witnessing the $\theta$-tallness with closure $\lambda$ of $\kappa$, force to violate $\GCH$ at either $\kappa$ or $\lambda$, and then lift the embedding to the forcing extension. In order to lift the embedding, I will use Woodin's method of surgery to modify a certain generic $g$ to obtain $g^*$ with the pullback property $j"H\subseteq g^*$. The following lemma, due to Woodin, will imply that $g^*$ is still a generic filter.

\begin{keylemma}\label{surgery}
Suppose $N$ and $M$ are transitive inner models of $\ZFC$ and $j:N\to M$ is a nontrivial elementary embedding with critical point $\kappa$ that is continuous at regular cardinals $\geq\lambda^+$ where $\lambda\geq\kappa$. If $A\in M$ is such that $|A|^M\leq j(\lambda)$ then $|A\cap \ran(j)|^V\leq\lambda$.
\end{keylemma}

\begin{proof}
Let $j:N\to M$ and $A\in M$ be as above; that is, $|A|\leq j(\lambda)$. 

First I will argue that it suffices to consider the case where $A$ is a set of ordinals. Let $\vec{B}:=\langle b_\alpha\mid\alpha<\beta\rangle\in N$ be a sequence of length $\beta$ such that $A\subseteq \ran(j(\vec{B}))$; for example, $\vec{B}$ could be an enumeration of some sufficiently large $V_\theta^N$ so that $j(\vec{B})$ is an enumeration of $V_{j(\theta)}^M$. Clearly $j(\vec{B})$ is a sequence of length $j(\beta)$ in $M$, write $j(\vec{B})=\langle b'_\alpha\mid\alpha<j(\beta)\rangle$. Let $A_0=\{\alpha<j(\beta)\mid b'_\alpha\in A\}$. Then $A_0\in M$ and $|A_0|^M=|A|^M$. Clearly $b'_\alpha\in\ran(j)$ if and only if for some $\xi<\beta$ it holds that $b'_\alpha=j(b_\xi)=b'_{j(\xi)}$. In other words, $b'_\alpha\in \ran(j)$ if and only if $\alpha\in\ran(j)$. It follows that $|A\cap\ran(j)|^V=|A_0\cap\ran(j)|^V$, and hence it will suffice to consider the case in which $A$ is a set of ordinals.

Suppose $A\in M$ is a set of ordinals with $|A|^M\leq j(\lambda)$ and $|A\cap\ran(j)|^V\geq\lambda^+$. Then $A$ contains $\lambda^+$-many elements of the form $j(\alpha)$. That is, $A$ contains elements of the form $j(\beta_\alpha)$ where $\alpha<\lambda^+$ and $\langle\beta_\alpha\mid\alpha<\lambda^+\rangle\in V$ is a strictly increasing sequence of ordinals which is not necessarily in $N$ since it was defined using $A\in M$. Now let $\delta=\sup\langle\beta_\alpha\mid\alpha<\lambda^+\rangle$. It follows that $\cf(\delta)^V=\lambda^+$ and hence $\cf(\delta)^N\geq\lambda^+$. By elementarity this implies that $\cf(j(\delta))^M\geq j(\lambda^+)$. Since $j$ is continuous at regular cardinals $\geq\lambda^+$, and thus at $\cf(\delta)^N$, it follows that $A$ contains unboundedly many $j(\beta_\alpha)$ less than $j(\delta)$. So in $M$, $A$ is unbounded in $j(\delta)$, but this implies that $|A|^M\geq j(\lambda^+)$ which contradicts our assumption that $|A|^M\leq j(\lambda)$.
\end{proof}

%In order to prove main theorem (2) we will start with a cardinal $\kappa$ that is $\theta$-tall with closure $\lambda$, pump up the size of the powerset of $\kappa$ to have size $\geq\theta$, and then lift the embedding using surgery.

The following theorem suffices to finish the proof of Theorem\ref{theoremdegsupcomp}(2) - (4).

\begin{theorem}\label{alllambda}
For any cardinals $\kappa\leq\lambda\leq\theta$, if $\kappa$ is $\lambda$-supercompact and $\theta$-tall then there is a forcing extension in which $\kappa$ is $\lambda$-supercompact and $2^\kappa\geq\theta$; and hence also $2^\lambda\geq\theta$. Indeed, the forcing preserves cardinals on $[\kappa,\lambda^+]\cup (2^\lambda,\infty)$ and assuming $\GCH$ holds at $\lambda$, all cardinals $\geq\kappa$ are preserved.
\end{theorem}

% The way in which theorem \ref{alllambda} establishes main theorem (3) and (4), by violating $\GCH$ at $\kappa$ in order to violate $\GCH$ at $\lambda$, raises the following question. Under $\GCH$ and from the hypothesis that $\kappa$ is $\lambda$-supercompact and $\lambda^{++}$-tall, can we obtain a forcing extension preserving the $\lambda$-supercompact\-ness of $\kappa$ in which $\GCH$ fails at $\lambda$ and holds at $\kappa$? We will show in a forthcoming article that under a mild assumption on $\lambda$ the answer to this question is yes.

In the following proof of Theorem \ref{alllambda}, I will use Woodin's method of surgery referred to just before the Key Lemma above.

\begin{proof}[Proof of Theorem \ref{alllambda}]\

\par

\subsection{Setup}

Let $\kappa$ be $\lambda$-supercompact and $\theta$-tall. As before, by the remarks at the end of Subsection \ref{tallness}, one may assume without loss of generality that $\lambda^{{<}\kappa}=\lambda$ and $\theta^\lambda=\theta$. One may further assume that $2^\lambda=\lambda^+$ since the forcing to achieve this is $\leq\lambda$-distributive and thus preserves the $\lambda$-supercompactness and $\theta$-tallness of $\kappa$. By Lemma \ref{generators}, there is an elementary embedding $j:V\to M$ with $\cp(j)=\kappa$, $j(\kappa)>\theta$, $M^\lambda\subseteq M$, and 
$$M=\{j(h)(j"\lambda,\alpha)\mid\alpha\leq\theta \textrm{ and } h:P_\kappa\lambda\times\kappa\to V\textrm{ is a function in } V\}.$$ 
By Lemma \ref{menaslemma}, one can assume without loss of generality that there is a partial function $f\subseteq\kappa\times\kappa$ in $V$ such that $j(f)(\kappa)>\theta$. Let $\P$ be the lottery preparation relative to $f$. Let $G\subseteq\P$ be $V$-generic and let $\Q=\Add(\kappa,\theta)^{V[G]}$. Let $H\subseteq \Q$ be $V[G]$-generic. Notice that since $\P$ has size $\kappa$ and $\Q$ is $\kappa^+$-c.c. it follows that $\P*\Q$ is $\kappa^+$-c.c., and thus by Lemma \ref{lemmachain} that $M[G][H]$ is closed under $\lambda$-sequences in $V[G][H]$.

\subsection{Lifting $j$ Through the Lottery Preparation}

By elementarity $j(\P)$ is the lottery preparation of length $j(\kappa)$ relative to $j(f)$ as defined in $M$. Since $M$ is closed under $\lambda$-sequences in $V$ it follows that the iterations $\P$ and $j(\P)$ agree up to stage $\kappa$ and since $\Q\in M[G]$ is ${<}\kappa$-closed it appears in the stage $\kappa$ lottery in $j(\P)$. Hence $j(\P)$ factors below a condition $p\in j(\P)$ that opts for $\Q$ at stage $\kappa$ as $j(\P)\restrict p\cong \P*\Q*\P_{tail}$. Since $j(f)(\kappa)>\theta$ it follows that the next nontrivial stage of forcing in $j(\P)$ is beyond $\theta$ and hence that $\P_{tail}$ is a term for ${\leq}\theta$-closed forcing. As in the proof of Theorem \ref{theoremindestructibility} it will be shown that one may construct a descending $\lambda^+$-sequence of conditions in $V[G][H]$ that meets every dense subset of $\P_{tail}$ in $M[G][H]$. Let $D$ be a dense subset of $\P_{tail}$ in $M[G][H]$ and let $\dot{D}\in M$ be a nice $\P*\Q$-name for $D$. Then $D=j(h_{\dot{D}})(j"\lambda,\alpha)_{G*H}$ for some $\alpha\leq\theta$ and some function $h_{\dot{D}}$ with domain $P_\kappa\lambda\times\kappa$ and range contained in the set of nice names for dense subsets of a tail of $\P$. Since the sequence of names $\langle j(h_{\dot{D}})(j"\lambda,\alpha)\mid\alpha\leq\theta\rangle$ is in $M$ and $\P_{tail}$ is $\leq\theta$-closed in $M[G][H]$, there is a condition in $\P_{tail}$ meeting every dense set mentioned by the sequence $\langle j(h_{\dot{D}})(j"\lambda,\alpha)\mid\alpha\leq\theta\rangle$. Since there are ${\leq}\lambda^+$-many functions from $P_\kappa\lambda\times\kappa$ to the set of nice names for dense subsets of a tail of $\P$, it follows from the fact that $M[G][H]$ is closed under $\lambda$-sequences in $V[G][H]$ that one can construct a descending $\lambda^+$-sequence in $V[G][H]$ that meets each dense subset of $\P_{tail}$ in $M[G][H]$ as in the proof of Theorem \ref{theoremindestructibility}. Let $G_{tail}$ be the $M[G][H]$-generic filter generated by this descending sequence. Then the embedding lifts in $V[G][H]$ to $j:V[G]\to M[j(G)]$ where $j(G)=G*H*G_{tail}$ and since $G_{tail}\in V[G][H]$ it follows from Lemma \ref{lemmaground} that  $M[j(G)]$ is closed under $\lambda$-sequences in $V[G][H]$.

%As in the proof of lemma GCHatkappa we may construct a descending $\lambda^+$-sequence in $V[G][H]$ which generates an $M[G][H]$-generic for $\P_{tail}$.

\subsection{The Factor Diagram}

Let $X=\{j(h)(j"\lambda,\theta)\mid h:P_\kappa\lambda\times\kappa\to V[G] \textrm{ where $h\in V[G]$}\}$. Then it follows that $X\elesub M[j(G)]$. Let $k:M_0'\to M[j(G)]$ be the inverse of the Mostowski collapse $\pi:X\to M_0'$ and let $j_0:V[G]\to M_0'$ be defined by $j_0:=k^{-1}\circ j$. It follows that $j_0$ is the ultrapower embedding by the measure $\mu:=\{X\subseteq P_\kappa\lambda\times\kappa\mid (j"\lambda,\theta)\in j(X)\}$ where $\mu\in V[G][H]$. By elementarity, $M_0'$ is of the form $M_0[j_0(G)]$, where $M_0\subseteq M_0'$ and $j_0(G)\subseteq j_0(\P)\in M_0'$ is $M_0$-generic. Furthermore, $j_0(G)=G*H_0*G^0_{tail}$ where $H_0$ is $M_0[G]$-generic for $\Add(\kappa,\pi(\theta))^{M_0[G]}$ and $G^0_{tail}$ is $M[G][H_0]$-generic for the tail of the iteration $j_0(\P)$ above $\kappa$. The following diagram is commutative.
$$
\xymatrix{
V[G] \ar[r]^j \ar[rd]_{j_0} & M[j(G)] \\
						 & M_0[j_0(G)] \ar[u]_k \\
}
$$
It follows that $j_0$ is a class of $V[G][H_0]$ and that $M_0[j_0(G)]$ is closed under $\lambda$-sequences in $V[G][H_0]$ and that $j_0(\kappa)>\pi(\theta)$.

%Let us argue that $M_0[j_0(G)]$ is closed under $\lambda$-sequences in $V[G]$. Since each $x\in X$ is of the form $x=j(h)(j"\lambda,\theta)$ for some $h\in V[G]$ with $\dom(h)=P_\kappa\lambda\times\kappa$ and since the critical point of $k$ is $>\lambda$, it follows that each element of $M_0[j_0(G)]$ is of the form $\pi\left(j(h)(j"\lambda,\theta)\right)=j_0(h)(j_0"\lambda,\pi(\theta))$ where $h$ is as above. Let $a=(j_0"\lambda,\pi(\theta))$, then each element of $M_0[j_0(G)]$ is of the form $j_0(h)(a)$ where $h\in V[G]$ with $\dom(h)=P_\kappa\lambda\times\kappa$. Let $\langle z_\alpha\mid\alpha<\lambda\rangle\in M_0[j_0(G)]^\lambda\cap V[G]$ then for each $\alpha<\lambda$ we have $z_\alpha=j_0(h_\alpha)(a)$ for some $h_\alpha\in V[G]$ with $\dom(h_\alpha)=P_\kappa\lambda\times\kappa$. Then $j_0(\langle h_\alpha\mid\alpha<\lambda\rangle)\in M_0[j_0(G)]$ is a $j_0(\lambda)$-sequence and by restricting this sequence to $j_0"\lambda$ and reindexing we see that $\langle j_0(h_\alpha)\mid\alpha<\lambda\rangle\in M_0[j_0(G)]$. This implies that $\langle j_0(h_\alpha)(a)\mid\alpha<\lambda\rangle=\langle z_\alpha\mid\alpha<\lambda\rangle\in M_0[j_0(G)]$. Hence $M_0[j_0(G)]$ is closed under $\lambda$-sequences in $V[G]$.

\subsection{Outline of the Rest of the Proof}\label{outline}

One would like to lift $j$ through the stage $\kappa$ forcing, $\Q$. This cannot be accomplished using a master condition argument since $\lambda$ may be less than $|j"H|^{V[G]}=\theta$. In order to lift the embedding, force with $j_0(\Q)$ over $V[G][H]$ to obtain a generic $g_0$ for $j_0(\Q)$. In Subsection \ref{obtainingsection}, I will argue that $k"g_0$ generates an $M[j(G)]$-generic $g$ for $j(\Q)$. However, one has no reason to expect that $j"H\subseteq g$, and thus one needs to do more work in order to lift the embedding. In Subsection \ref{surgerysection} I will use Woodin's method of surgery to modify the filter $g$ to obtain an $M[j(G)]$-generic $g^*$ for $j(\Q)$ with $j"H\subseteq g^*$. Then the embedding lifts to 
\begin{align}
j:V[G][H]\to M[j(G)][g^*]\label{surgerylift}
\end{align}
in $V[G][H][g_0]$ where $j(H)=g^*$.

The embedding (\ref{surgerylift}) does not witness that $\kappa$ is $\lambda$-supercompact in $V[G][H]$ because the embedding is a class of $V[G][H][g_0]$. Under the assumption that the embedding lifts to $V[G][H]$ as in (\ref{surgerylift}), I will now show that the embedding lifts further to the final model $V[G][H][g_0]$ witnessing that $\kappa$ is $\theta$-tall and $\lambda$-supercompact in $V[G][H][g_0]$. Furthermore, I will show in Subsection \ref{preservescardinals} that 
\begin{align}
(2^\kappa\geq\theta)^{V[G][H][g_0]}.\label{GCHresultsalllambda}
\end{align}

Let me first argue, assuming that $g^*$ is as above, that $M[j(G)][g^*]$ is closed under $\lambda$-sequences in $V[G][H][g_0]$. I will now show that $j_0(\Q)$ is ${\leq}\lambda$-distributive in $V[G][H]$. Since $j_0(\kappa)>\lambda$ it follows that $j_0(\Q)$ is ${\leq}\lambda$-closed in $M_0[j_0(G)]$. Since $M_0[j_0(G)]$ is closed under $\lambda$-sequences in $V[G][H_0]$ it follows that $j_0(\Q)$ is ${\leq}\lambda$-closed in $V[G][H_0]$ and since ${\leq}\lambda$-closed forcing remains ${\leq}\lambda$-distributive in $\lambda^+$-c.c. forcing extensions, it follows that $j_0(\Q)$ is ${\leq}\lambda$-distributive in $V[G][H]$. Since $M[j(G)]$ is closed under $\lambda$-sequences in $V[G][H]$ and $j_0(\Q)$ is ${\leq}\lambda$-distributive in $V[G][H]$ it easily follows that $M[j(G)]$ is closed under $\lambda$-sequences in $V[G][H][g_0]$. Since $g^*$ is constructed from $g_0$ one concludes that $g^*\in V[G][H][g_0]$, and from this it follows that $M[j(G)][g^*]$ is closed under $\lambda$-sequences of ordinals in $V[G][H][g_0]$. By using a well ordering of a sufficient initial segment of the universe $M[j(G)][g^*]$, it follows that $M[j(G)][g^*]$ is closed under $\lambda$-sequences in $V[G][H][g_0]$.

Now I show that the embedding (\ref{surgerylift}) lifts through $j_0(\Q)$. Every element of $M[j(G)][g^*]$ is of the form $j(h)(j"\lambda,\alpha)$ where $h:P_\kappa\lambda\times\kappa\to V[G][H]$ is in $V[G][H]$ and $\alpha\leq\theta$. In other words, $M[j(G)][g^*]$ is generated by $\{(j"\lambda,\alpha)\mid\alpha\leq\theta\}\subseteq P_\kappa\lambda\times\kappa$ over $V[G][H]$. Since $P_\kappa\lambda\times\kappa$ has size $\lambda$ and $j_0(\Q)$ is ${\leq}\lambda$-distributive, it follows by Lemma \ref{lemmalambdadist} that $j"g_0$ generates an $M[j(G)][g^*]$-generic filter $j(g_0)$ for the poset $j(j_0(\Q))$. Thus $j$ lifts in $V[G][H][g_0]$ to $$j:V[G][H][g_0]\to M[j(G)][g^*][j(g_0)].$$
Since $j(g_0)\in V[G][H][g_0]$ and $M[j(G)][g^*]$ is closed under $\lambda$-sequences in $V[G][H][g_0]$ it follows that $M[j(G)][g^*][j(g_0)]$ is closed under $\lambda$-sequences in $V[G][H][g_0]$. Since $j$ is a lift of the original embedding it still satisfies $j(\kappa)>\theta$. Hence $j$ witnesses that $\kappa$ is $\theta$-tall with closure $\lambda$ in $V[G][H][g_0]$.

To complete the proof of Theorem \ref{alllambda} it remains to carry out the surgery argument and to show that (\ref{GCHresultsalllambda}) holds.

\subsection{Obtaining the Generic for Use in Surgery}\label{obtainingsection}

Let $g_0$ be as in Subsection \ref{outline}; that is, $g_0$ is $V[G][H]$-generic for $j_0(\Q)$. In this Subsection I will argue that $k"g_0$ generates an $M[j(G)]$-generic for $j(\Q)$. Each $x\in M[j(G)]$ is of the form $x=j(h)(j"\lambda,\alpha)$ for some $\alpha\leq\theta$ and some $h:P_\kappa\lambda\times\kappa\to V[G]$ with $h\in V[G]$. Since $j_0"\lambda\in M_0$ it follows that each $x\in M[j(G)]$ is of the form $k(h)(\alpha)$ for some $\alpha\leq\theta$ and some $h:j_0(\kappa)\to M_0[j_0(G)]$ with $h\in M_0[j_0(G)]$; in fact, since $k(h\restrict\pi(\theta))$, where $\pi(\theta)$ is the collapse of $\theta$, still has every $\alpha\leq\theta$ in its domain, one may assume that $h:\pi(\theta)\to M_0$. Let $D$ be an open dense subset of $j(\Q)$ in $M[j(G)]$. Then $D=k(\vec{D})(\alpha)$ for some fixed $\alpha\leq\theta$ where $\vec{D}=\langle D_\beta \mid \beta<\pi(\theta)\rangle$ is a sequence of dense open subsets of $j_0(\Q)$. Since $j_0(\Q)$ is ${\leq}\pi(\theta)$-distributive one sees that $\bar{D}:=\bigcap_{\beta<\pi(\theta)}(D_\beta)$ is open dense in $j_0(\Q)$. Hence there is a condition $p\in g_0\cap \bar{D}$. Then $k(p)\in k"g_0\cap k(\bar{D})$. Now $\bar{D}\subseteq D_\beta=\vec{D}(\beta)$ for each $\beta<\pi(\theta)$ and this implies $k(\bar{D})\subseteq k(\vec{D})(\beta)$ for each $\beta<k(\pi(\theta))=\theta$. It follows that 
$$ k(\bar{D})\subseteq k(\vec{D})(\alpha)=D$$
and hence $k(p)\in k"g_0\cap k(\bar{D})\subseteq D$. Therefore $k"g_0$ generates an $M[j(G)]$-generic filter for $j(\Q)$.

% Some details I decided to omit:

%Each $x\in M[j(G)]$ is of the form $x=k(j_0(h))(k\circ j_0"\lambda,\alpha)$ $\implies$ each $x\in M[j(G)]$ is of the form $x=k(h')(k"j_0"\lambda,\alpha')$ where $h': j_0(P_\kappa\lambda)\times j_0(\kappa)\to M_0[j_0(G)]$, $h'\in M_0[j_0(G)]$, $\alpha'<\theta$. Let $h'':j_0(\kappa)\to M_0[j_0(G)]$ be defined by $h''(\xi):=h'(j_0"\lambda,\xi)$. Then $k(h'')(\alpha)=k(h')(k(j_0"\lambda),\alpha)=k(h')(k"j_0"\lambda,\alpha)$.

% Here we argue that the collapse $\pi(\theta)$ is $>\lambda$, or in other words, $\cp(k)>\lambda$. Clearly, $j"\lambda\in X$ and furthermore $j(\alpha)\in X$ for $\alpha\leq\lambda$ because $\ran(j)\subseteq X$. This implies that $\sigma:=j"\lambda\cap j(\alpha)\in X$. Now, suppose $X\models \ot(\sigma)=\delta$ where $\delta$ is an ordinal in $X$. Then since $X\elesub M$ we see that $\delta$ is an ordinal in $M$ and $M\models\ot(\sigma)=\delta$. Since $M\models\ot(\sigma)=\alpha$ we conclude that $\ot(\sigma)=\alpha\in X$. Hence each $\alpha\leq\lambda$ is in $X$ and this implies that $\cp(k)>\lambda$. In other words, when we collapse the seed hull $X$ to $M_0$, no ordinals are collapsed to $\lambda$ or below.

% The last paragraph implies that $k(j_0"\lambda)=k"j_0"\lambda$.

\subsection{Surgery}\label{surgerysection}

From the work above, $g$ is $M[j(G)]$-generic for $j(\Q)$. I now use Woodin's method of surgery to obtain an $M[j(G)]$-generic $g^*$ for $j(\Q)$ with $j"H\subseteq g^*$. Define $g^*$ in terms of $g$ and $j"H$ in the following way. Let $\Delta$ be the set of coordinates $(\alpha,\beta)\in j(\theta)\times j(\kappa)$ such that there is a $p\in H$ such that $(\alpha,\beta)\in\dom(p)$ and $j(p)(\alpha,\beta)\neq g(\alpha,\beta)$ and let $\pi: j(\Q)\to j(\Q)$ be the automorphism induced by flipping bits over coordinates in $\Delta$. Let $g^*:=\pi"g$. In other words, one obtains the modified generic $g^*$ simply by using $g$, except that whenever $g$ and $j"H$ disagree, one changes $g$ to match $j"H$.

%\begin{center}

\begin{figure}
\centering
\begin{tikzpicture}

%Draw crosshair at the origin:
%\draw (-.1,0) -- (.1,0); 
%\draw (0,-.1) -- (0,.1);

\filldraw[thick,fill=gray!30] (0, 1.5) node [left] {$\kappa$} --  (1.5,1.5) -- (1.5,0) node [below] {$\kappa$} -- (0,0) -- (0,1.5);
\filldraw[thick,fill=gray!30] (2.5,0) node [below] {$j(\kappa)$} -- (2.5,1.5) -- (3.2,1.5) -- (3.2,0);
\filldraw[thick,fill=gray!30] (3.8,0) -- (3.8,1.5) -- (4.5,1.5) -- (4.5,0);

%Shade using grid:
%\draw[step=.2cm,gray,very thin] (-4, -2) grid (-2, 0);
%\draw[step=.2cm,gray,very thin] (0, -2) grid (1, 0);
%\draw[step=.2cm,gray,very thin] (2, -2) grid (3, 0);

\node at ( .75,.75) [circle,draw=none] {$j"H$};
\node at ( 2,2) [circle,draw=none] {$g$};

\draw[thick] (0, 2.5) -- (5, 2.5) -- (5,0) node [below] {$j(\theta)$} -- (0,0) -- (0,2.5)
	node [left] {$j(\kappa)$};

%\draw (0,0) .. controls (1,1) and (2,1) .. (2,0) .. controls (2,-1) and (3,-1) .. (4,0);
\end{tikzpicture}\caption{The domain of $j(\Q)$.}\label{figuredomain1}
\end{figure}
%\end{center}

%Let $g^*(j(\alpha),j(\beta))= H(\alpha,\beta)$ for $(\alpha,\beta)\in\theta\times\kappa$. Otherwise, if $(\alpha,\beta)\in j(\theta)\times j(\kappa)\setminus \ran(j)$, let $g^*(\alpha,\beta)=g(\alpha,\beta)$. So to get $g^*$ we change the generic $g$ on just those coordinates that will ensure that $j"H\subseteq g^*$. In terms of conditions, for $p\in g$ we define $p^*$ so that $\dom(p^*)=\dom(p)$ as follows. If $(\alpha,\beta)\in \theta\times\kappa$ and $(j(\alpha),j(\beta))\in \dom(p)$ let $p^*(j(\alpha),j(\beta))=H(\alpha,\beta)$ and if $(\alpha,\beta)\in (j(\theta)\times j(\kappa))\cap\dom(p)\setminus \ran(j)$ then we let $p^*(\alpha,\beta)=p(\alpha,\beta)$. Then we have $g^*=\{p^*\mid p\in g\}$.

Since $j$ is continuous at regular cardinals $\geq\lambda^+$ the key lemma applies and will be used to show that $g^*$ is a generic filter on $j(\Q)$. First note that if $p\in j(\Q)$ then $|p|^{M[j(G)]}<j(\kappa)$ and so the set of coordinates on which $p^*:=\pi(p)\neq p$ has size $\leq\lambda$ by the key lemma and is thus in $M[j(G)]$ since $M[j(G)]^\lambda\subseteq M[j(G)]$ in $V[G][H][g_0]$. This implies that $p^*\in M[j(G)]$ and thus that $g^*$ defines a filter in $M[j(G)]$.

Now I show that $g^*$ is $M[j(G)]$-generic for $j(\Q)$. Let $A\subseteq j(\Q)$ be a maximal antichain in $M[j(G)]$. Since $j(\Q)$ has the $j(\kappa)^+$-c.c. it follows that $|A|^{M[j(G)]}\leq j(\kappa)$. Furthermore, each $p\in A$ has $|p|^{M[j(G)]}<j(\kappa)$. Hence $|\bigcup_{p\in A}\dom(p)|^{M[j(G)]}\leq j(\kappa)$. By the key lemma, the set of coordinates mentioned by conditions in $A$ that were involved in the changes made in going from $g$ to $g^*$ has size $\leq \lambda$, call this set $\Delta_A$. In other words, $\Delta_A:=\Delta\cap\left(\bigcup_{p\in A}\dom(p)\right)$. Let $\pi_A:j(\Q)\to j(\Q)$ be the automorphism induced by flipping bits over coordinates in $\Delta_A$. The coordinates of bits that get flipped by $\pi_A$ are contained in the domain of the antichain (see the shaded region in the figure below).

%\begin{center}
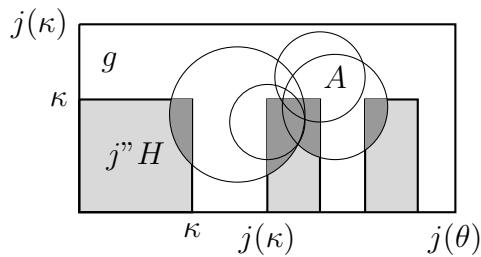
\begin{figure}
\centering
\begin{tikzpicture}

%Draw crosshair at the origin:
%\draw (-.1,0) -- (.1,0); 
%\draw (0,-.1) -- (0,.1);

\def\recta{(0, 1.5) node [left] {$\kappa$} --  (1.5,1.5) -- (1.5,0) node [below] {$\kappa$} -- (0,0) -- (0,1.5)}
	
\def\rectb{(2.5,0) node [below] {$j(\kappa)$} -- (2.5,1.5) -- (3.2,1.5) -- (3.2,0)}

\def\rectc{(3.8,0) -- (3.8,1.5) -- (4.5,1.5) -- (4.5,0)}

\filldraw[thick,fill=gray!30] \recta;

\filldraw[thick,fill=gray!30] \rectb;

\filldraw[thick,fill=gray!30] \rectc;

%Shade using grid:
%\draw[step=.2cm,gray,very thin] (-4, -2) grid (-2, 0);
%\draw[step=.2cm,gray,very thin] (0, -2) grid (1, 0);
%\draw[step=.2cm,gray,very thin] (2, -2) grid (3, 0);

\node at ( .75,.75) [circle,draw=none] {$j"H$};
\node at ( 0.4,2) [circle,draw=none] {$g$};

\draw[thick] (0, 2.5) -- (5, 2.5) -- (5,0) node [below] {$j(\theta)$} -- (0,0) -- (0,2.5)
	node [left] {$j(\kappa)$};

%\draw (0,0) .. controls (1,1) and (2,1) .. (2,0) .. controls (2,-1) and (3,-1) .. (4,0);

\def\circlea{(3.4,1.4) circle (0.7)}
\def\circleb{(2.6,1.5) circle (0.5)}
\def\circlec{(2.5,1.2) circle (0.5)}
\def\circled{(2.1,1.3) circle (0.9)}
\def\circlee{(3.2,1.8) circle (0.6)}

\begin{scope}
\clip \circled;
\fill[gray!80] \recta;
\end{scope}

\begin{scope}
\clip \circled;
\fill[gray!80] \rectb;
\end{scope}

\begin{scope}
\clip \circlea;
\fill[gray!80] \rectb;
\end{scope}

\begin{scope}
\clip \circlea;
\fill[gray!80] \rectc;
\end{scope}

\draw \circlea;
%\draw \circleb;
\draw \circlec;
\draw \circled;
\draw \circlee;

\node at ( 3.4,1.8) [circle,draw=none] {$A$};

\end{tikzpicture}\caption{The domain of $A$.}\label{figuredomain2}
\end{figure}

%\end{center}
%So, if a condition $p$ doesn't mention any of the coordinates of conditions in $A$ that were involved in the changes, then $\pi_A(p)=p$. If $p$ does mention some coordinates of conditions in $A$ that were involved in the changes, then $\pi_A$ flips just these bits to produce $\pi_A(p)$. 
Since $|\Delta_A|\leq\lambda$ one has $\Delta_A\in M[j(G)]$ and it follows that $\pi_A\in M[j(G)]$. Then $\pi_A^{-1}"A$ is a maximal antichain of $j(\Q)$ and by genericity of $g$ there is a condition $p\in  g$ that decides $\pi_A^{-1}"A$. It follows that $\pi(p)\in g^*$  decides $A$ since $\pi"A = \pi_A"A$. This establishes that $g^*$ is $M[j(G)]$-generic for $j(\Q)$.

% Joel's way of justifying $\pi_A\in M[j(G)]$:

%Since $j"\lambda\in M[j(G)]$ and $H\in M[j(G)]=M[G][H][G_{tail}]$, and since $\pi_A$ can be built from (is definable from) these, it follows that $\pi_A\in M[j(G)]$. Since $\pi_A^{-1}"A\in M[j(G)]$ is still a maximal antichain of $j(\Q)$ there is a condition $p\in \pi_A^{-1}"A\cap g$. This implies that $\pi_A(p)\in A\cap g^*$. Thus, $g^*$ intersects every maximal antichain of $j(\Q)$ in $M[j(G)]$.

Since $j"H\subseteq g^*$ was arranged by definition, one may use Lemma \ref{lemmaliftingcriterion} to lift the embedding to $j:V[G][H]\to M[j(G)][j(H)]$ where $j(H)=g^*$. Since $g_0$ was used to define $g^*$, this lift is a class of $V[G][H][g_0]$. It was shown above (in the outline given in Section \ref{outline}) that the embedding lifts further through the $g_0$ forcing. So all that remains is to show that $j_0(\Q)$ preserves cardinals and that $2^\kappa\geq\theta$ in $V[G][H][g_0]$.

\subsection{Preserving $2^\kappa\geq\theta$ in $V[G][H][g_0]$}\label{preservescardinals}

I have already argued that $j_0(\Q)$ is $\leq\lambda$-distributive in $V[G][H]$ and I will now argue that $j_0(\Q)$ is $\lambda^{++}$-c.c. From this it follows that $j_0(\Q)$ preserves cardinals over $V[G][H]$ and $2^\kappa\geq\theta$ in $V[G][H]$. Each condition $p\in j_0(\Q)$ is in $M_0[j_0(G)]$ and is thus of the form $p=j_0(h_p)(j_0"\lambda,\theta)$ for some $h_p:P_\kappa\lambda\times\kappa\to \Q$ with $h\in V[G]$. For each $p\in j_0(\Q)$, $\dom(h_p)$ has size $\lambda$ in $V[G]$ and thus $h_p$ leads to a function $\bar{h}_p:\lambda\to \Q$, which can be viewed as a condition in the full support product of $\lambda$-many copies of $\Q$ as defined in $V[G]$, which I denote by $\overline{\Q}$. I will show that $j_0(\Q)$ is $\lambda^{++}$-c.c. in $V[G][H]$ by arguing that $\bar{\Q}$ is $\lambda^{++}$-c.c. in $V[G][H]$ and that an antichain of $j_0(\Q)$ of size $\lambda^{++}$ in $V[G][H]$ would lead to an antichain of $\bar{\Q}$ of size $\lambda^{++}$ in $V[G][H]$.

\begin{claim}\label{Qbarlemma}
$\bar{\Q}$ is $\lambda^{++}$-c.c. in $V[G][H]$
\end{claim}
\begin{proof}[Proof of claim]

By a delta system argument $\bar{\Q}$ is $\lambda^{++}$-c.c. in $V[G]$. Suppose $A\in V[G][H]$ is an antichain of $\bar{\Q}$ with $|A|=\delta$. I will show that $A$ leads to an antichain of size $\delta$ of $\bar{\Q}\cong \Q\times\bar{\Q}$ in $V[G]$ and thus that $\delta<\lambda^{++}$. Let 
$$q\forces \dot{A}\textrm{ is an antichain of $\bar{\Q}$ and }\dot{f}:\delta\to \dot{A}\textrm{ is bijective}$$ where $q\in \Q\cap H$ and $\dot{A}_H=A$. For each $\alpha<\delta$ let $q_\alpha\leq q$ be such that $q_\alpha\forces \dot{f}(\check{\alpha})=\check{p}_\alpha$ where $p_\alpha\in \bar{\Q}$. It follows that $\bar{\Q}\cong\Q\times\bar{\Q}$ in $V[G]$ and I now show that $W:=\{(q_\alpha,p_\alpha)\in \Q\times\bar{\Q}\mid\alpha<\delta\}$ is an antichain of size $\delta$ of $\Q\times\bar{\Q}$ in $V[G]$. Clearly $W\in V[G]$ because in choosing the pairs $(q_\alpha,p_\alpha)$ in $W$ one only needs to use the forcing relation $\forces_{\Q}$. Suppose for a contradiction that $W$ is not an antichain, i.e. that $(q^*,p^*)\leq(q_\alpha,p_\alpha),(q_\beta,p_\beta)$ for some $\alpha,\beta<\delta$ with $\alpha\neq\beta$ and some $(q^*,p^*)\in \Q\times\bar{\Q}$. Let $H^*$ be $V[G]$-generic for $\Q$ with $q^*\in H^*$. Since $q^*\leq q$ it follows that $\dot{f}_{H^*}$ enumerates an antichain. Furthermore, it follows that $\dot{f}_{H^*}(\alpha)=p_\alpha$, $\dot{f}_{H^*}(\beta)=p_\beta$, and $p^*\leq p_\alpha, p_\beta$, a contradiction. Hence one concludes that $W$ is an antichain of $\Q\times\bar{\Q}$ in $V[G]$ and since $\bar{\Q}\cong\Q\times\bar{\Q}$ it follows that $W$ leads to an antichain of $\bar{\Q}$ of size $\delta$ in $V[G]$. Therefore, $\delta<\lambda^{++}$. Hence, an antichain $A$ of $\bar{\Q}$ in $V[G][H]$ must have size $<\lambda^{++}$.
\end{proof}

Now I complete the proof of Theorem \ref{alllambda} by showing that $j_0(\Q)$ is $\lambda^{++}$-c.c. in $V[G][H]$. Suppose that in $V[G][H]$, $j_0(\Q)$ has an antichain $A$ of size $\delta$. For each $p\in A\subseteq j_0(\Q)$ let $h_p:P_\kappa\lambda\times\kappa\to \Q$ be such that $p=j_0(h_p)(j"\lambda,\theta)$ where $h_p\in V[G]$. As above each $h_p$ yields a condition in $\bar{\Q}$, call it $\bar{h}_p$. For $p,q\in A$ it follows that $j_0(h_p)(j"\lambda,\theta)\perp_{j_0(\Q)} j_0(h_q)(j"\lambda,\theta)$, and thus by elementarity, there is a $(\sigma,\alpha)\in P_\kappa\lambda\times\kappa$ such that $h_p(\sigma,\alpha)\perp h_q(\sigma,\alpha)$. This implies that $\bar{A}:=\{\bar{h}_p\mid p\in A\}$ is an antichain in $\bar{\Q}$ where $|\bar{A}|=\delta$. By Claim \ref{Qbarlemma}, $\delta<\lambda^{++}$. Thus $j_0(\Q)$ is $\lambda^{++}$-c.c. in $V[G][H]$.

Thus in $V[G][H][g_0]$, $\kappa$ is $\lambda$-supercompact and $\theta$-tall, and $2^\lambda\geq\theta$.

Let me argue that cardinals in $[\kappa,\lambda^+]\cup(2^\lambda,\infty)$ are preserved. I started with a model and forced $2^\lambda=\lambda^+$ which may have collapsed cardinals in $(\lambda^+, 2^\lambda]$. I then add a fast function using $\kappa^+$-c.c. forcing which preserves cardinals $\geq\kappa$. The remaining forcing is $\P*\Q*j_0(\Q)$ where $\P*\Q$ is $\kappa^+$-c.c. and $j_0(\Q)$ preserves cardinals over $V[G][H]$. Thus in the final model $V[G][H][g_0]$ cardinals are preserved in $[\kappa,\lambda^+]\cup (2^\lambda,\infty)$.
\end{proof}

%===========================================================
%===========================================================
%===========================================================
%===========================================================
%===========================================================
%===========================================================
%===========================================================
%
%
%  THE LEVINSKI PROPERTY
%
%===========================================================
%===========================================================%===========================================================
%===========================================================

\chapter{The Levinski Property}\label{chapterlevinski}

Levinski proved in \cite{Levinski:FiltersAndLargeCardinals}, from merely the existence of a measurable cardinal, say $\kappa$, that there is a forcing extension in which $\kappa$ remains measurable and yet $\kappa$ is the least regular cardinal at which $\GCH$ holds. This result is in contrast with Scott's well known theorem that if $\GCH$ fails at a measurable cardinal $\kappa$, then $\{\delta<\kappa\mid 2^\delta>\delta^+\}$ has measure one with respect to some normal measure on $\kappa$. In preparation for what follows let me give the following definition.

\begin{definition}
A cardinal $\kappa$ is said to have the \emph{Levinski property} if $\kappa$ is the least regular cardinal such that $2^\kappa=\kappa^+$. Additionally, $\kappa$ has the \emph{strict Levinski property} if $\kappa$ has the Levinski Property and both (1) for each regular cardinal $\delta<\kappa$ one has $2^\delta=\delta^{++}$ and (2) for each regular $\delta\geq\kappa$ one has $2^\delta=\delta^+$.
\end{definition}

Clearly, the strict Levinski property implies the Levinski property. I make the distinction because forcing the strict Levinski property requires controlling the continuum function in a more specific manner than forcing the Levinski property.

In this chapter, I recall a proof of Levinski's theorem suggested in \cite{ApterCummings:BlowingUpThePowerSetOf}, and I generalize Levinski's theorem to other large cardinal contexts. It will easily follow that one may force inaccessible cardinals, Mahlo cardinals, and weakly compact cardinals to have the Levinski property. Due to the fact that many of the stronger large cardinals are $\Sigma_2$-reflecting, it is impossible for them to have the Levinski property. For example, supercompact cardinals and huge cardinals are $\Sigma_2$-reflecting and thus cannot have the Levinski property. However, I will show that the proof of Levinski's theorem can be extended to show that, among other things, the degree of hugeness can be forced to have the Levinski property; more precisely, if $\kappa$ is an $n$-huge cardinal and $\GCH$ holds, I will show that there is a forcing extension in which $\kappa$ remains $n$-huge and $j^{n}(\kappa)$ has the strong Levinski property.

\section{Levinski's Theorem and a Generalization}

\subsection{The Forcing}\label{sectionlevinskiforcing}

Here I provide a discussion of the forcing notion due to Levinski that will produce, from a measurable cardinal $\kappa$, a forcing extension in which $\kappa$ remains measurable and $\kappa$ has the Levinski property. The forcing will be an Easton support iteration of Easton support products. First, define the length $\kappa$ iteration $\P_\kappa=\langle (\P_{\alpha},\dot{\Q}_{\alpha}) : \alpha<\kappa \rangle$ as follows. For a cardinal $\lambda$, let $\bar{\lambda}$ denote the least inaccessible cardinal in the interval $(\lambda,\kappa)$, if such a cardinal exists, and $\bar{\lambda}=\kappa$ otherwise. In general, for a set $I\subseteq\ORD$, define
$$\Q_{I}:=\prod_{\gamma\in I\cap\REG}\Add(\gamma,\gamma^{++})$$
where the product has Easton support. The \emph{Levinski iteration up to $\kappa$}, denoted by $\P_\kappa$, is defined by the following.
\begin{enumerate}
\item The first stage of nontrivial forcing is $\Q_\omega:=\Q_{[\omega,\bar{\omega})}$.
\item If $\alpha<\kappa$ is an inaccessible cardinal, then $\dot{\Q}_\alpha$ is a $\P_\alpha$-name for the forcing $\Q_{[\alpha,\bar{\alpha})}$ as defined in $V^{\P_\alpha}$ and $\P_{\alpha+1}:=\P_\alpha*\dot{\Q}_\alpha$.
\item If $\alpha<\kappa$ is a singular limit of inaccessible cardinals, then $\dot{\Q}_\alpha$ is a $\P_\alpha$-name for the forcing $\Q_{[\alpha,\bar{\alpha})} = \Q_{[\alpha^+,\bar{\alpha})}$ as defined in $V^{\P_\alpha}$ and $\P_{\alpha+1}:=\P_\alpha*\dot{\Q}_\alpha$.
\item Otherwise, for $\alpha<\kappa$, $\dot{\Q}_\alpha$ is a $\P_\alpha$-name for trivial forcing and $\P_{\alpha+1}:=\P_\alpha *\dot{\Q}_\alpha$.
\item The iteration uses Easton support. 
\end{enumerate}

Condition 3 above is necessary in order to, for example, force $\GCH$ to fail at the successor of the least singular limit of inaccessible cardinals. Notice that if $\alpha$ is a singular limit of inaccessible cardinals, then $\forces_{\P_\alpha}$ ``$\dot{\Q}_\alpha$ is ${\leq}\alpha$-closed,'' since $\dot{\Q}_\alpha$ is a $\P_\alpha$-name for $\Q_{[\alpha^+,\bar{\alpha})}$. 

\begin{remark}[Notation]\ 
\begin{itemize}
\item If $\delta<\kappa$ is either an inaccessible cardinal or a singular limit of inaccessible cardinals, then the Levinski iteration can be factored as $\P_\kappa\cong \P_\delta*\dot{\Q}_{[\delta,\bar{\delta})}*\dot{\P}_{[\bar{\delta},\kappa)}$. Furthermore, working in $V^{\P_\delta}$, for each regular cardinal $\gamma\in[\delta,\bar{\delta})$ let $\R_\gamma$ denote the $\gamma^{th}$-factor of the product $\Q_{[\delta,\bar{\delta})}$ so that $\Q_{[\delta,\bar{\delta})}\cong \R_\delta\times\Q_{[\delta^+,\bar{\delta})}$.
\item In the results below, I will typically be lifting elementary embeddings of the form $j:V\to M$ with critical point $\kappa$ through the Levinski iteration. Thus, $j(\P_\kappa)$ will be the Levinski iteration up to $j(\kappa)$ as defined in $M$. Working in $M$, suppose $\delta<j(\kappa)$ is an inaccessible cardinal or a singular limit of inaccessible cardinals. In $M^{\P_\delta}$, let $\widetilde{\Q}_{[\delta,\bar{\delta}^M)}$ denote the stage $\delta$ forcing in $j(\P_\kappa)$, where $\bar{\delta}^M$ is the least $M$-inaccessible cardinal greater than $\delta$. As above, I will write $\widetilde{\Q}_{[\delta,\bar{\delta}^M)}\cong\widetilde{\R}_\delta\times\widetilde{\Q}_{[\delta^+,\bar{\delta}^M)}$.
\end{itemize}
\end{remark}

As it will be used many times in later sections, let me isolate the following lemma which asserts that the Levinski iteration up to $\kappa$ preserves cardinals and forces $\kappa$ to have the Levinski property. This lemma shows that in the results that follow, one only needs to be concerned with preserving the large cardinal property under consideration. The proof of the next lemma is omitted, as it is completely standard.

\begin{lemma}\label{lemmalevinski}
If $\GCH$ holds then the Levinski iteration up to $\kappa$, written as $\P_\kappa$, preserves cofinalities and forces $\kappa$ to have the strict Levinski property.
\end{lemma}

\subsection{Proof of Levinski's Theorem}\label{sectionmeasurable}

I now state and prove Levinski's result which originally appeared in \cite{Levinski:FiltersAndLargeCardinals} and is discussed in \cite{ApterCummings:BlowingUpThePowerSetOf}. The proof presented below, while substantially different from Levinski's, is suggested by the discussion in \cite{ApterCummings:BlowingUpThePowerSetOf}, and is similar to the proof of Theorem \ref{easy}.

\begin{theorem}[Levinski]\label{theoremlevinski}
Suppose $\kappa$ is measurable and $\GCH$ holds. Then there is a cofinality-preserving forcing extension in which $\kappa$ remains measurable and $\kappa$ has the strict Levinski property, i.e. $2^\delta=\delta^{++}$ for each regular cardinal $\delta<\kappa$, and $2^\kappa=\kappa^+$.
\end{theorem}

\begin{proof}

Let $G_\kappa$ be $V$-generic for the Levinski iteration $\P_\kappa$ define above and let $H_\kappa$ be $V[G_\kappa]$-generic for $\R_\kappa:=\Add(\kappa,\kappa^+)^{V[G_\kappa]}$.

It is routine to check that $\P_\kappa*\R_\kappa$ preserves cardinals and that in $V[G_\kappa][H_\kappa]$, $2^\delta=\delta^{++}$ for each regular $\delta<\kappa$, and $2^\kappa=\kappa^+$. It remains to show that $\kappa$ is measurable in $V[G_\kappa][H_\kappa]$.

Let $j:V\to M$ be an ultrapower by a normal measure on $\kappa$. Then $M=\{j(h)(\kappa)\mid h:\kappa\to V\textrm{ and } h\in V\}$ and $M^\kappa\subseteq M$ in $V$. This implies that 
\begin{align*}
j(\P_\kappa)&\cong\P_\kappa * \dot{\widetilde{\Q}}_{[\kappa,\bar{\kappa}^M)}*\dot{\widetilde{\P}}_{[\bar{\kappa}^M,j(\kappa))}\\
	&\cong \P_\kappa * (\dot{\widetilde{\R}}_\kappa\times\dot{\widetilde{\Q}}_{[\kappa^+,\bar{\kappa}^M)}) * \dot{\widetilde{\P}}_{[\bar{\kappa}^M,j(\kappa))}
\end{align*}
where $\bar{\kappa}^M$ is the next $M$-inaccessible cardinal greater than $\kappa$ and $\widetilde{\Q}_{[\kappa,\bar{\kappa}^M)}$ is the Easton support product violating $\GCH$ at regulars on the interval $[\kappa,\bar{\kappa}^M)$ as defined in $M[G_\kappa]$ and $\widetilde{\R}_\kappa=\Add(\kappa,\kappa^{++})^{M[G_\kappa]}$.

Since $\P_\kappa$ is $\kappa$-c.c., it follows that $M[G_\kappa]^\kappa\subseteq M[G_\kappa]$. Furthermore, since $\widetilde{\Q}_{[\kappa^+,\bar{\kappa}^M)}$ is $\leq\kappa$-closed in $M[G_\kappa]$ and has at most $2^\kappa=\kappa^+$ many dense subsets in $M[G_\kappa]$ one can build an $M[G_\kappa]$-generic filter $\widetilde{H}_{[\kappa^+,\bar{\kappa}^M)}$ for $\widetilde{\Q}_{[\kappa^+,\bar{\kappa}^M)}$ in $V[G_\kappa]$. 

Since $j(\kappa)$ has size $\kappa^+$ it follows that $(\kappa^{++})^M<\kappa^{++}$. Since $M$ and $V$ agree on subsets of $(\kappa^{++})^M$ of size $<\kappa$, there is an isomorphism $\Add(\kappa,\kappa^{++})^{M[G_\kappa]}\stackrel{i}{\cong} \Add(\kappa,\kappa^+)^{V[G_\kappa]}$ in $V[G_\kappa]$. By using $i$ to rearrange $H_\kappa$ and by noting that $\widetilde{H}_{[\kappa^+,\bar{\kappa}^M)}\in V[G_\kappa]$ it follows that there is an $M[G_\kappa][\widetilde{H}_{[\kappa^+,\bar{\kappa}^M)}]$-generic $\widetilde{H}_\kappa$ for $\widetilde{\R}_\kappa$ in $V[G_\kappa][H_\kappa]$. By the product forcing lemma, $H_{[\kappa,\bar{\kappa}^M)}:=\widetilde{H}_\kappa\times\widetilde{H}_{[\kappa^+,\bar{\kappa}^M)}$ is $M[G_\kappa]$-generic for $\widetilde{\R}_\kappa\times\widetilde{\Q}_{[\kappa^+,\bar{\kappa}^M)}$. Since $\P_\kappa * \R_\kappa$ is $\kappa^+$-c.c. it follows that $M[G_\kappa][\widetilde{H}_\kappa]$ is closed under $\kappa$-sequences in $V[G_\kappa][H_\kappa]$. Since $\widetilde{H}_{[\kappa^+,\bar{\kappa}^M)}\in V[G_\kappa][H_\kappa]$ it follows that $M[G_\kappa][\widetilde{H}_\kappa][\widetilde{H}_{[\kappa^+,\bar{\kappa}^M)}]$ is closed under $\kappa$-sequences in $V[G_\kappa][H_\kappa]$.

The forcing $\widetilde{\P}_{[\bar{\kappa}^M,j(\kappa))}$ is $\leq\kappa$-closed and has at most $\kappa^+$ dense subsets in $M[G_\kappa][\widetilde{H}_{[\kappa,\bar{\kappa}^M)}]$. This implies that there is an $M[G_\kappa][H_{[\kappa,\bar{\kappa}^M)}]$-generic $\widetilde{G}_{[\bar{\kappa}^M,j(\kappa))}$ for $\widetilde{\P}_{[\bar{\kappa}^M,j(\kappa))}$ in $V[G_\kappa][H_\kappa]$.

Since conditions in $G_\kappa$ have support bounded below $\kappa$ it follows that $j"G_\kappa\subseteq j(G_\kappa):=G_\kappa*\widetilde{H}_{[\kappa,\bar{\kappa}^M)}*\widetilde{G}_{[\bar{\kappa}^M,j(\kappa))}$ and thus the embedding lifts to $j:V[G_\kappa]\to M[j(G_\kappa)]$ in $V[G_\kappa][H_\kappa]$ where $M[j(G_\kappa)]$ is closed under $\kappa$-sequences in $V[G_\kappa][H_\kappa]$.

It remains to lift $j$ through $\R_\kappa=\Add(\kappa,\kappa^+)^{V[G_\kappa]}$. Notice that since $\R_\kappa$ has size $\kappa^+$ a master condition argument is not possible here. I will build a descending sequence of increasingly masterful conditions in $j(\R_\kappa)=\Add(j(\kappa),j(\kappa^+))^{M[j(G_\kappa)]}$.

Suppose $A\in M[j(G_\kappa)]$ is a maximal antichain of $j(\R_\kappa)$ and $p\in j(\R_\kappa)$ is a condition that is compatible with every element of $j"H_\kappa$. I will first argue that $p$ extends to some $p'\leq p$ that decides $A$ and remains compatible with every element of $j"H_\kappa$. Since $j(\R_\kappa)$ is $j(\kappa^+)$-c.c., $A$ has size at most $j(\kappa)$ in $M[j(G_\kappa)]$. Since $\sup j"\kappa^+ = j(\kappa^+)$ and $j(\kappa^+)$ is regular in $M[j(G_\kappa)]$ it follows that there is an $\alpha<\kappa^+$ such that $A\subseteq \Add(j(\kappa),j(\alpha))$. Let me fix such an $\alpha$ with $p\in \Add(j(\kappa),j(\alpha))$ as well. Since $j"(H_\kappa\cap\Add(\kappa,\alpha))$ is a $\kappa$-sequence of elements of $M[j(G_\kappa)]$ which is in $V[G_\kappa][H_\kappa]$, it follows that $q:=\bigcup j"(H_\kappa\cap\Add(\kappa,\alpha))$ is a \emph{master} condition in $\Add(j(\kappa),j(\alpha))$. Since $p$ is compatible with every element of $j"H_\kappa$ it follows that $p$ and $q$ are compatible in $\Add(j(\kappa),j(\alpha))$. Let $p'$ be a condition in $\Add(j(\kappa),j(\alpha))$ that extends $p$ and $q$ and decides the maximal antichain $A$. Let me argue that $p'$ remains compatible with every element of $j"H_\kappa\subseteq \Add(j(\kappa),j(\kappa^+))$. Suppose $j(r)\in j"H_\kappa$. Then $j(r)=r_0\cup r_1$ where $\dom(r_0)\subseteq j(\alpha)\times j(\kappa)$ and $\dom(r_1)\subseteq [j(\alpha),j(\kappa^+))\times j(\kappa)$. It follows that $p'$ is compatible with $r_0$ since $p'\leq q\leq r_0$ and $p'$ is compatible with $r_1$ since $\dom(p')\cap\dom(r_1)=\emptyset$. Hence $p'$ and $r$ are compatible in $j(\R_\kappa)=\Add(j(\kappa),j(\kappa^+))$.

Since $\R_\kappa$ has at most $\kappa^+$ maximal antichains it follows that $j(\R_\kappa)$ has at most $j(\kappa^+)$ maximal antichains in $M[j(G_\kappa)]$. Since $|j(\kappa^+)|^V\leq (\kappa^+)^\kappa = \kappa^+$ it follows that $V[G_\kappa]$ has an enumeration $\langle A_\xi\mid\xi<\kappa^+\rangle$ of all maximal antichains of $j(\R_\kappa)$ in $M[j(G_\kappa)]$. By iterating the procedure in the previous paragraph and using the facts that $j(\R_\kappa)$ is $\leq\kappa$-closed in $M[j(G_\kappa)]$ and $M[j(G_\kappa)]$ is closed under $\kappa$ sequences in $V[G_\kappa][H_\kappa]$, one can build a descending sequence of conditions $\langle p_\xi\mid\xi<\kappa^+\rangle$ in $V[G_\kappa][H_\kappa]$ such that $p_\xi\in j(\R_\kappa)$ decides $A_\xi$ and is compatible with every element of $j"H_\kappa$. Let $j(H_\kappa)$ be the filter generated by this descending sequence. It follows that $j(H_\kappa)\in V[G_\kappa][H_\kappa]$ and that $j(H_\kappa)$ is $M[j(G_\kappa)]$-generic for $j(\R_\kappa)$. Suppose $p\in H_\kappa$ then since every element of $j(H_\kappa)$ is compatible with $j(p)$ it follows from the fact that generic filters are maximal that $j(p)\in j(H_\kappa)$. Hence the embedding lifts to $j:V[G_\kappa][H_\kappa]\to M[j(G_\kappa)][j(H_\kappa)]$ in $V[G_\kappa][H_\kappa]$, witnessing that $\kappa$ is measurable in $V[G_\kappa][H_\kappa]$.
\end{proof}

Let me note here that by modifying the definition of the forcing used in the proof of Theorem \ref{theoremlevinski} one may in fact show that from a measurable cardinal $\kappa$, any Easton function $F:\REG\to\CARD$ with $F"\kappa\subseteq\kappa$ can be realized as the continuum function below $\kappa$. In a related result, Friedman and Honsik have shown in \cite{FriedmanHonzik:EastonsTheoremAndLargeCardinals} that if $F:\REG\to\CARD$ is an Easton function with $F"\kappa\subseteq\kappa$ and $\kappa$ is $F(\kappa)$-hypermeasurable witnessed by an embedding $j$ with $j(F)(\kappa)\geq F(\kappa)$, then there is a forcing extension in which $\kappa$ is measurable and $2^\delta=F(\delta)$ for every regular cardinal.

\subsection{A Generalization of Levinski's Theorem Involving Measurable Cardinals}

I will now generalize Levinski's theorem by starting with a class of measurable cardinals, force to a model in which $\GCH$ fails at every nonmeasurable regular cardinal, $\GCH$ holds at every measurable cardinal, every measurable cardinal is preserved, and no new measurable cardinals are created.

\begin{theorem}\label{theoremeverymeasurable}
Suppose there is a measurable cardinal and $\GCH$ holds. There is a forcing extension in which all measurable cardinals are preserved, no new measurable cardinals are created, and $\GCH$ holds at every measurable cardinal and fails at every nonmeasurable regular cardinal.
\end{theorem}

\begin{proof}

Let $\P$ be the Easton support iteration defined as follows. Let $\gamma_0$ denote the least inaccessible cardinal. Note that the forcing which violates $\GCH$ on the interval $[\omega,\gamma_0)$ will be put off until later so that Hamkins' results in \cite{Hamkins:GapForcing} or \cite{Hamkins:Approximation}. The first stage of forcing in $\P$ is $\Add(\omega,1)*\Q_{[\gamma_0,\bar{\gamma}_0)}$. Here I will use the notation $\Q_I$ to denote the Easton support product forcing over an interval $I$ as defined at the beginning of Subsection \ref{sectionlevinskiforcing}. If $\gamma$ is a nonmeasurable inaccessible cardinal or a singular limit of inaccessible cardinals then force at stage $\gamma$ with $\Q_{[\gamma,\bar{\gamma})}$ as defined in $V^{\P_\gamma}$. If $\gamma$ is a measurable cardinal then the stage $\gamma$ forcing is $\Add(\gamma,\gamma^+)\times\Q_{[\gamma^+,\bar{\gamma})}$ as defined in $V^{\P_\gamma}$. If there is a largest inaccessible cardinal, say $\delta$, then there is a last stage of forcing in the iteration $\P$ of the form $\Q_{[\delta,\infty)}$, where the first factor of $\Q_{[\delta,\infty)}$ depends as above on whether $\delta$ is a measurable cardinal or a nonmeasurable inaccessible cardinal. If there is no largest inaccessible cardinal, but the inaccessible cardinals are bounded, say $\delta=\sup\{\gamma\mid\textrm{$\gamma$ is inaccessible}\}$, then the last stage of forcing occurs at stage $\delta$ and is of the form $\Q_{[\delta,\infty)}=\Q_{[\delta^+,\infty)}$ since $\delta$ must be singular. If there is no largest inaccessible cardinal, then $\P$ is a class length iteration as defined above. Let $G$ be $V$-generic for $\P$ and let $H$ be $V[G]$-generic for $\Q_{[\omega,\bar{\omega})}$. Note that $\bar{\omega}=\gamma_0$. I will show that the conclusion of the theorem holds in $V[G][H]$.

By factoring the iteration one can verify that $\P$ preserves all cardinals and cofinalities. It is also easy to check that $\GCH$ holds at every cardinal that is measurable in $V$ and fails at every cardinal which is nonmeasurable and regular in $V$, starting at the least inaccessible. Since the forcing $\P$ factors as $\Add(\omega,1)*\P_1$ where $|\Add(\omega,1)|<\omega_1$ and $\P_1$ is $\leq\omega_1$ closed it follows by Hamkins' gap forcing theorem (see \cite{Hamkins:GapForcing}) that every measurable cardinal in $V[G]$ is measurable in $V$ and hence that $\GCH$ holds at every measurable cardinal in $V[G]$. Since the forcing $H\subseteq\Q_{[\omega,\bar{\omega})}$ is small relative to the least measurable cardinal, it neither creates nor destroys any measurable cardinals. Thus it will suffice to show that if $\kappa$ is a measurable cardinal in $V$, then it remains so in $V[G]$.

It remains to show that $\P$ preserves every measurable cardinal. Suppose $\kappa$ is a measurable cardinal and factor 
$\P\cong\P_\kappa*\dot{\Q}_{[\kappa,\bar{\kappa})}*\dot{\P}_{[\bar{\kappa},\infty)}.$ Let $j:V\to M$ be an ultrapower embedding by a normal measure on $\kappa$ such that $\kappa$ is not measurable in $M$. Since $j(\kappa)$ is measurable in $M$ there will be many $M$-inaccessible cardinals below $j(\kappa)$. It follows that 
$$j(\P_\kappa)\cong\P_\kappa* (\dot{\widetilde{\R}}_\kappa\times \dot{\widetilde{\Q}}_{[\kappa^+,\bar{\kappa}^M)})*\dot{\widetilde{\P}}_{tail}$$ 
where $\dot{\widetilde{\R}}_\kappa$ is a $\P_\kappa$-term for $\Add(\kappa,\kappa^{++})^{M[G_\kappa]}$, $\dot{\widetilde{\Q}}_{[\kappa^+,\bar{\kappa}^M)}$ is a $\P_\kappa$-term for the relevant Easton product on the interval $[\kappa^+, \bar{\kappa}^M)$ where $\bar{\kappa}^M$ is the next $M$-inaccessible cardinal above $\kappa$, and $\dot{\widetilde{\P}}_{tail}$ is a $\P_\kappa*(\dot{\widetilde{\R}}_\kappa\times\dot{\widetilde{\Q}}_{[\kappa^+,\bar{\kappa}^M)})$-term for $\leq\kappa$-closed forcing. It follows as in the proof of Theorem \ref{theoremlevinski}, that the embedding lifts to $j:V[G_\kappa]\to M[j(G_\kappa)]$ in $V[G_\kappa][H_\kappa]$ where $H_\kappa$ is $V[G_\kappa]$-generic for $\R_\kappa$.

As before, the embedding lifts further through $\R_\kappa=\Add(\kappa,\kappa^+)^{V[G_\kappa]}$ by building a descending sequence of increasingly masterful conditions in $j(\R_\kappa)$. Thus $j$ lifts to $j:V[G_\kappa][H_\kappa]\to M[j(G_\kappa)][j(H_\kappa)]$ in $V[G_\kappa][H_\kappa]$ witnessing that $\kappa$ is measurable in $V[G_\kappa][H_\kappa]$.

By Lemma \ref{lemmaeaston} the remaining forcing $\Q_{[\kappa^+,\bar{\kappa})} * \dot{\P}_{[\bar{\kappa},\infty)}$ is $\leq\kappa$ distributive in $V[G_\kappa][H_\kappa]$ and therefore will not kill the measurability of $\kappa$. So $\kappa$ remains measurable in $V[G]$.
\end{proof}

\section{Other Large Cardinals and the Levinski Property}\label{sectioncohen}

I now begin a survey in which various large cardinals are preserved through forcing that achieves the Levinski property at some cardinal relevant to the definition of the large cardinal property. For example, I will show that if $\kappa$ is $n$-huge then there is a forcing extension preserving this in which $j^n(\kappa)$ has the Levinski property. 
%In some cases, the large cardinal $\kappa$ can be preserved by the Levinski iteration either up to $\kappa$, or up to some other cardinal $\lambda$, which is relevant in the definition of the large cardinal property. 

\subsection{Various Small Large Cardinals and the Levinski Property}

I will begin with a few easy cases and work my way up through the large cardinal hierarchy. If $\kappa$ is inaccessible, it is easy to show, using standard arguments, that $\kappa$ remains inaccessible after forcing with the Levinski iteration up to $\kappa$. If $\kappa$ is Mahlo, standard arguments show that after forcing with $\P_\kappa$, each inaccessible cardinal less than $\kappa$ remains inaccessible and every new club contains an old club. Hence in $V^{\P_\kappa}$, every club subset of $\kappa$ contains an inaccessible and thus $\kappa$ remains Mahlo and has the Levinski property.

Let me now consider weakly compact cardinals. I will quickly review a characterization of weak compactness that is useful in preserving the property through forcing. First, a transitive set $M$ is a \emph{$\kappa$-model} if $|M|=\kappa$, $M^{<\kappa}\subseteq M$, and $M$ satisfies $\ZFC^-$ where $\ZFC^-$ is the theory consisting of the usual axioms of $\ZFC$, excluding the powerset axiom, and using the collection axiom in place of the replacement axiom. A cardinal $\kappa$ is \emph{weakly compact} if for every $A\subseteq \kappa$ there is a $\kappa$-model $M$ with $\kappa,A\in M$ and a transitive set $N$ with an embedding $j:M\to N$ with critical point $\kappa$.

\begin{theorem}
Suppose $\kappa$ is weakly compact and $\GCH$ holds. Then there is a cofinality-preserving forcing extension in which $\kappa$ has the strict Levinski property and remains weakly compact.
\end{theorem}

\begin{proof}
Assume without loss of generality that $\GCH$ holds since this can be forced without disturbing the weak compactness of $\kappa$. Let $\P_\kappa$ be the Levinski iteration up to $\kappa$ and let $G_\kappa$ be $V$-generic for $\P_\kappa$. Suppose $A\subseteq\kappa$ in $V[G_\kappa]$ and let $\dot{A}$ be a nice $\P_\kappa$-name for $A$ such that $\dot{A}\in H_{\kappa^+}$. Fix a $\kappa$-model $M$ with $\dot{A},V_\kappa\in M$ and an elementary embedding $j:M\to N$ where $N$ is also a $\kappa$-model. Since $V_\kappa\in M$ it follows that $\P_\kappa\in M$. I will lift the embedding to $j:M[G_\kappa]\to N[j(G_\kappa)]$ in $V[G_\kappa]$ thus showing that in $V[G_\kappa]$ the set $A$ can be put into a $\kappa$ model which is the domain of an elementary embedding with critical point $\kappa$.

Clearly $j(\P_\kappa)$ is the Levinski iteration up to $j(\kappa)$ as defined in $N$. Furthermore, $j(\P_\kappa)\cong\P_\kappa*\dot{\widetilde{\P}}_{[\kappa,j(\kappa))}$ where $\widetilde{\P}_{[\kappa,j(\kappa))}$ is the Levinski iteration over the interval $[\kappa,j(\kappa))$ as defined in $N[G_\kappa]$. It follows that $\widetilde{\P}_{[\kappa,j(\kappa))}$ is ${<}\kappa$-closed in $N[G_\kappa]$. Since $N$ has size $\kappa$ in $V$ it follows that $\widetilde{\P}_{[\kappa,j(\kappa))}$ has at most $\kappa$ dense subsets in $N[G_\kappa]$. Furthermore, since $N^{<\kappa}\subseteq N$ in $V$ and $\P_\kappa$ is $\kappa$-c.c. it follows by Lemma \ref{lemmachain} that $N[G_\kappa]^{<\kappa}\subseteq N[G_\kappa]$ in $V[G_\kappa]$. Thus one may diagonalize to build an $N[G_\kappa]$-generic filter $\widetilde{G}_{[\kappa,j(\kappa))}$ for $\widetilde{\P}_{[\kappa,j(\kappa))}$ in $V[G_\kappa]$. Since conditions in $\P_\kappa$ have support bounded below the critical point of $j$, it follows that $j"G_\kappa\subseteq G_\kappa * \widetilde{G}_{[\kappa,j(\kappa))}$ and hence the embedding lifts to $j:V[G_\kappa]\to M[j(G_\kappa)]$ in $V[G_\kappa]$ where $j(G_\kappa):=G_\kappa * \widetilde{G}_{[\kappa,j(\kappa))}$.
\end{proof}

\subsection{Partially Supercompact Cardinals and the Levinski Property}\label{sectionpartiallysupercompact}

%Recall that $\kappa$ is \emph{$\lambda$-supercompact} if there is a normal fine measure $U$ on $P_\kappa\lambda$. It is easy to show that the ultrapower embedding $j:V\to M=\Ult(V,U)$ has critical point $\kappa$, $j(\kappa)>\lambda$, and $M^\lambda\subseteq M$ in $V$. 

If $\kappa$ is even $\kappa^+$-supercompact and $2^\kappa=\kappa^+$ then $\GCH$ must hold on a measure one subset of $\kappa$. So, of course, there is no hope of having a nontrivially partially supercompact cardinal with the Levinski property. I will show that one can generalize the proof of Levinski's theorem to show that from a $\lambda$-supercompact cardinal there is a forcing extension in which $\kappa$ is $\lambda$-supercompact and $\lambda$ has the Levinski property, assuming $\lambda$ is regular.

\begin{theorem}\label{theorempartiallysupercompact}
Suppose $\kappa$ is $\lambda$-supercompact where $\lambda$ is regular and $\GCH$ holds. Then there is a cofinality-preserving forcing extension in which $\kappa$ remains $\lambda$-supercompact and $\lambda$ has the strict Levinski property.
\end{theorem}

\begin{proof}
Let $\P_\lambda$ be the Levinski forcing up to $\lambda$, that is force at stage $\gamma<\lambda$ with $\Q_{[\gamma,\bar{\gamma})}$ provided $\gamma$ is either inaccessible or a singular limit of inaccessibles. As before, if $\gamma$ is inaccessible or a singular limit of inaccessible cardinals and there is no inaccessible cardinal in the interval $(\gamma,\lambda)$, then $\P_\lambda$ has a last stage of forcing of the form $\Q_{[\gamma,\lambda)}$. Let $\R_\lambda=\Add(\lambda,\lambda^+)^{V^{\P_\lambda}}$. Let $G_\lambda * H_\lambda$ be $V$-generic for $\P_\lambda*\dot{\R}_\lambda$.

Let $j:V\to M$ be the ultrapower by a normal fine measure on $P_\kappa\lambda$. One may assume without loss of generality that $\lambda^{<\kappa}=\lambda$ since any $\lambda$-supercompactness embedding is also a $\lambda^{<\kappa}$-supercompactness embedding. Then $\lambda\in[\theta,\bar{\theta}^M)$ where $\theta$ is either an $M$-inaccessible cardinal or a singular limit of $M$-inaccessible cardinals. It follows that 
$$j(\P_\kappa)\cong \P_\kappa*\dot\P_{[\kappa,\theta)}*(\dot\Q_{[\theta,\lambda)}\times\dot{\widetilde{\R}}_\lambda\times\dot{\widetilde{\Q}}_{[\lambda^+,\bar{\theta}^M)})*\dot{\widetilde{\P}}_{[\bar{\theta}^M,j(\kappa))}$$
where $\widetilde{\R}_\lambda=\Add(\lambda,\lambda^{++})^{M[G_\theta]}$ and the iteration $j(\P_\kappa)$ agrees with $\P_\lambda:=\P_\theta*\dot\Q_{[\theta,\lambda)}$ up to $\lambda$ since $M$ is closed under $\lambda$-sequences. 

Let me argue that $\P_\theta$ is $\lambda^+$-c.c. If $\theta$ is $M$-inaccessible then since $M^\lambda\subseteq M$ in $V$, it follows that $\theta$ is inaccessible in $V$, and hence that $\P_\theta$ has size $\theta^{<\theta}=\theta$ and is thus $\theta^+$-c.c. Since $\lambda\in[\theta,\bar{\theta}^M)$, this implies $\P_\theta$ is $\lambda^+$-c.c.  If $\theta$ is a singular limit of $M$-inaccessible cardinals, then $V$ agrees on this. Now $\P_\theta$ has size at most $2^\theta=\theta^+$ and is thus $\theta^{++}$-c.c. Furthermore, $\theta^+\leq\lambda$ since $\lambda$ is regular and hence $\theta^{++}\leq\lambda^+$. It easily follows that $\P_\theta$ is $\lambda^+$-c.c.

Since $\P_\theta$ is $\lambda^+$-c.c. in either case above, it follows that $M[G_\theta]^\lambda\subseteq M[G_\theta]$ in $V[G_\theta]$ by Lemma \ref{lemmachain}. As before one can build an $M[G_\theta]$-generic $\widetilde{H}_{[\lambda^+,\bar{\theta}^M)}$ for $\widetilde{\Q}_{[\lambda^+,\bar{\theta}^M)}$ in $V[G_\theta]$ and it follows by Lemma \ref{lemmaground} that $M[G_\theta][\widetilde{H}_{[\lambda^+,\bar{\theta}^M)}]$ is closed under $\lambda$-sequences in $V[G_\theta]$. Now $H_{[\theta,\lambda)}\times H_\lambda$ is $V[G_\theta]$-generic for $\Q_{[\theta,\lambda)}\times\R_\lambda$ and since $\R_\lambda$ is isomorphic to $\widetilde{\R}_\lambda$ in $V[G_\theta]$ this implies that there is a $V[G_\theta][H_{[\theta,\lambda)}]$-generic $\widetilde{H}_\lambda$ in $V[G_\theta][H_{[\theta,\lambda)}\times H_\lambda]$ for $\widetilde{\R}_\lambda$. Since the forcing $\Q_{[\theta,\lambda)}\times\R_\lambda\cong \Q_{[\theta,\lambda)}\times\widetilde{\R}_\lambda$ is $\lambda^+$-c.c. it follows by Lemma \ref{lemmachain} that $M[G_\theta][\widetilde{H}_{[\lambda^+,\bar{\theta}^M)}][H_{[\theta,\lambda)}\times\widetilde{H}_\lambda]$ is closed under $\lambda$-sequences in $V[G_\lambda]$. Furthermore, by the product forcing lemma $\widetilde{H}_{[\theta,\bar{\theta}^M)}:=H_{[\theta,\lambda)}\times\widetilde{H}_\lambda\times\widetilde{H}_{[\lambda^+,\bar{\theta}^M)}$ is $M[G_\theta]$-generic for $\Q_{[\theta,\lambda)}\times\widetilde{\R}_\lambda\times\widetilde{\Q}_{[\lambda^+,\bar{\theta}^M)}$ and $M[G_\theta][\widetilde{H}_{[\theta,\bar{\theta}^M)}]$ is closed under $\lambda$-sequences in $V[G_\lambda]$. Finally since $\widetilde{\P}_{[\bar{\theta}^M,j(\kappa))}$ is $\leq\lambda$-closed in $M[G_\theta][\widetilde{H}_{[\theta,\bar{\theta}^M)}]$ one can build an $M[G_\theta][\widetilde{H}_{[\theta,\bar{\theta}^M)}]$-generic $\widetilde{G}_{[\bar{\theta}^M,j(\kappa))}$ for it in $V[G_\lambda][H_\lambda]$. Since conditions in $G_\kappa$ have bounded support it follows that $j"G_\kappa\subseteq j(G_\kappa):=G_\theta*\widetilde{H}_{[\theta,\bar{\theta}^M)}*\widetilde{G}_{[\bar{\theta}^M,j(\kappa))}$ and thus the embedding lifts to $j:V[G_\kappa]\to M[j(G_\kappa)]$ in $V[G_\lambda][H_\lambda]$ where $M[j(G_\kappa)]$ is closed under $\lambda$-seqeunces in $V[G_\lambda][H_\lambda]$.

It remains to demonstrate that $j$ lifts through $\P_{[\kappa,\theta)}*(\Q_{[\theta,\lambda)}\times \R_\lambda)$. 
%By Lemma 21.7 in \cite{Jech:Book} 
Since $\P_{[\kappa,\theta)}$ is ${<}\kappa$-directed closed in $V[G_\kappa]$, the forcing $j(\P_{[\kappa,\theta)})$ is ${\leq}{\lambda}$-directed closed. Since $H_{[\kappa,\theta)}$ has size at most $\lambda$ (a slightly different calculation depending on if $\theta$ is singular or inaccessible in $M$) it follows that there is a master condition $p\in j(\P_{[\kappa,\theta)})$ below each element of $j" H_{[\kappa,\theta)}$. Let me show that one may build an $M[j(G_\kappa)]$-generic filter for $j(\P_{[\kappa,\theta)})$ below $p$. Since $\P_{[\kappa,\theta)}$ has size at most $\lambda$ it has at most $2^\lambda$ dense subsets. By elementarity $j(\P_{[\kappa,\theta)})$ has at most $j(2^\lambda)$ dense subsets in $M[j(G_\kappa)]$. Since each element of $j(2^\lambda)$ is witnessed by a function $P_\kappa\lambda\to 2^\lambda$ in $V[G_\kappa]$ it follows that $j(\P_{[\kappa,\theta)})$ has at most $(2^\lambda)^{\lambda^{<\kappa}}=(2^\lambda)^\lambda = 2^\lambda = \lambda^+$ dense subsets in $M[j(G_\kappa)]$. Furthermore, since $j(\P_{[\kappa,\theta)})$ is $\leq\lambda$-closed in $M[j(G_\kappa)]$ and since $M[j(G_\kappa)]$ is closed under $\lambda$-sequences in $V[G_\lambda][H_\lambda]$, it follows that one can build an $M[j(G_\kappa)]$-generic filter $j(H_{[\kappa,\theta)})$ for $j(\P_{[\kappa,\theta)})$ in $V[G_\lambda][H_\lambda]$ and lift the embedding to 
$$j:V[G_\kappa][H_{[\kappa,\theta)}]\to M[j(G_\kappa)][j(H_{[\kappa,\theta)})]$$
in $V[G_\lambda][H_\lambda]$. Additionally, $M[j(G_\kappa)][j(H_{[\kappa,\theta)})]$ is closed under $\lambda$-sequences in $V[G_\lambda][H_\lambda]$ since $j(H_{[\kappa,\theta)})\in V[G_\lambda][H_\lambda]$.

A standard argument shows that since $\lambda$ is regular, $\Q_{[\theta,\lambda)}$ is ${<}\theta$-directed closed in $V[G_\theta]$. By elementarity, $j(\Q_{[\theta,\lambda)})$ is $\leq\lambda$-directed closed in the model $M[j(G_\kappa)][j(H_{[\kappa,\theta)})]$. Since $\Q_{[\theta,\lambda)}$ has at most $2^\lambda=\lambda^+$ dense subsets it follows that $j(\Q_{[\theta,\lambda)})$ has at most $\lambda^+$ dense subsets in $M[j(G_\kappa)][j(H_{[\kappa,\theta)})]$. Thus one can use a master condition argument to lift $j$ to
$$j:V[G_\lambda]\to M[j(G_\lambda)]$$
in $V[G_\lambda][H_\lambda]$ where $j(G_\lambda):=j(G_\kappa)*j(G_{[\kappa,\theta)})*j(H_{[\theta,\lambda)})$ and $M[j(G_\lambda)]$ is closed under $\lambda$-seqeunces in $V[G_\lambda][H_\lambda]$ since $j(H_{[\theta,\lambda)})\in V[G_\lambda][H_\lambda]$.

Now I will lift the embedding through the forcing $\R_\lambda=\Add(\lambda,\lambda^+)^{V[G_\lambda]}$. It follows that $j(\R_\lambda)=\Add(j(\lambda),j(\lambda^+))$. Suppose $A$ is a maximal antichain of $j(\R_\lambda)$ in $M[j(G_\lambda)]$ and $r\in j(\R_\lambda)$ is compatible with every element of $j"H_\lambda$. I will show that there is a condition $r'\leq r$ that decides $A$ and remains compatible with $j"H_\lambda$. Since $j(\R_\lambda)$ is $j(\lambda^+)$-c.c. in $M[j(G_\lambda)]$ one has $|A|^{M[j(G_\lambda)]}\leq j(\lambda)$. Since $j(\lambda^+)$ is regular in $M[j(G_\lambda)]$ and since $\sup j"\lambda^+=j(\lambda^+)$ there is an $\alpha<\lambda^+$ such that $A\subseteq \Add(j(\lambda),j(\alpha))$ and also $r\in\Add(j(\lambda),j(\alpha))$. Now $j"(H_\lambda\cap\Add(\lambda,\alpha))$ is a subset of $\Add(j(\lambda),j(\alpha))$ of size at most $\lambda$ which is in $V[G_\lambda][H_\lambda]$ and therefore also in $M[j(G_\lambda)]$. Now in $M[j(G_\lambda)]$, each element of $j"(H_\lambda\cap\Add(\lambda,\alpha))$ has size less than $j(\lambda)$. Hence $p:=\bigcup j"(H_\lambda\cap\Add(\lambda,\alpha))$ has size less than $j(\lambda)$ in $M[j(G_\lambda)]$ and is thus a condition in $\Add(j(\lambda),j(\alpha))^{M[j(G_\lambda)]}$. Since $r$ is compatible with every element of $j"H_\lambda$ it follows that $r$ and $p$ are compatible in $\Add(j(\lambda),j(\alpha))$ and thus one may let $r'\in \Add(j(\lambda),j(\alpha))$ decide the maximal antichain $A$ of $\Add(j(\lambda),j(\alpha))$ with $r'\leq r,p$. Choose a condition $q\in j"H_\lambda\subseteq\Add(j(\lambda),j(\lambda^+))$, then $q$ is of the form $q=q_0\cup q_1$ where $\dom(q_0)\subseteq j(\alpha)\times j(\lambda)$ and $\dom(q_1)\subseteq [j(\alpha),j(\lambda^+))\times j(\lambda)$. Since $r'\leq p$ we conclude that $r'\leq q_0$ and since $\dom(r')\subseteq j(\alpha)\times j(\lambda)$ it follows that $r'$ is compatible with $q_1$. Thus $r'$ is compatible with $q$.

Finally, since $j(\R_\lambda)$ has at most $\lambda^+$ maximal antichains in $M[j(G_\lambda)]$ the above procedure can be iterated to build in $V[G_\lambda][H_\lambda]$ a descending sequence of conditions meeting every maximal antichain of $j(\R_\lambda)$ in $M[j(G_\lambda)]$ such that each member of the sequence is compatible with $j"H_\lambda$. Let $j(H_\lambda)$ be the filter generated by the descending sequence. It follows that $j(H_\lambda)$ is $M[j(G_\lambda)]$-generic for $j(\R_\lambda)$ and $j"H_\lambda\subseteq j(H_\lambda)$. This implies that the embedding lifts to 
$$j:V[G_\lambda][H_\lambda]\to M[j(G_\lambda)][j(H_\lambda)]$$
in $V[G_\lambda][H_\lambda]$.
\end{proof}

Notice that in Theorem \ref{theorempartiallysupercompact} if $\lambda$ happened to be a measurable cardinal, then one could also preserve this through the forcing, using the argument given in Theorem \ref{theoremlevinski}.

\subsection{$n$-huge Cardinals and the Levinski Property}\label{sectionnhuge}

A cardinal $\kappa>\omega$ is \emph{$n$-huge with target $\lambda$} if there is a nontrivial elementary embedding $j:V\to M$ with critical point $\kappa$ such that $j^n(\kappa)=\lambda$ and $M^{\lambda}\subseteq M$ in $V$. As with many other large cardinal notions, $n$-hugeness can be characterized by the existence of a certain type of ultrafilter. I will give a brief outline showing how this is done. Suppose $j:V\to M$ witnesses the $n$-hugeness of $\kappa$ and let $\lambda_i:=j^i(\kappa)$ for $0< i\leq n$ and $\lambda_0:=\kappa$. Define 
$$U:=\{X\subseteq [\lambda]^{\lambda_{n-1}}\mid j"\lambda\in j(X)\}.$$ Then $U$ is a $\kappa$-complete ultrafilter and has the following properties. 
\begin{enumerate}
\item For each $\alpha<\lambda$, $\{x\in[\lambda]^{\lambda_{n-1}}\mid \alpha\in x\}\in U$.
\item For every function $f$ that is regressive on a set in $U$, meaning $\{x\in [\lambda]^{\lambda_{n-1}}\mid f(x)\in x\}\in U$, $f$ is constant on a set in $U$. 
\item Furthermore, $U$ has the property that 
\textrm{for each $i<n$, $\{x\in [\lambda]^{\lambda_{n-1}}\mid \ot(x\cap\lambda_{i+1})=\lambda_i\}\in U$.} \label{hugefilterproperty}
\end{enumerate}
Let $j_U:V\to M_U=\Ult(V,U)$. Then (\ref{hugefilterproperty}) implies that for $i<n$, $j_U(\lambda_i)=\lambda_{i+1}$, since:
\begin{align*}
\lambda_{i+1}&=\ot(j_U"\lambda\cap j_U(\lambda_{i+1}))\\
	&=[\ot(\langle x \cap \lambda_{i+1})\mid x \in [\lambda]^{\lambda_{n-1}}\rangle]_U\\
	&=[\langle \lambda_i\mid x\in[\lambda]^{\lambda_{n-1}}\rangle]_U\hspace{0.8in} \textrm{(by (\ref{hugefilterproperty}))}\\
	&=j(\lambda_i)
\end{align*}
Furthermore, $j_U"\lambda=[\id]_U\in M_U$ so that $M_U^\lambda\subseteq M_U$. Since $j_U$ is the ultrapower by $U$ it follows that 
$M_U=\{j_U(h)(j_U"\lambda)\mid h:[\lambda]^{\lambda_{n-1}}\to V, h\in V\}$. This establishes the following.

\begin{lemma}\label{lemmanhugeultrapower}
A cardinal $\kappa>\omega$ is $n$-huge if an only if there is an ultrapower embedding $j:V\to M$ by a normal fine measure on $[\lambda]^{\lambda_{n-1}}$ witnessing the $n$-hugeness of $\kappa$ where $\lambda_i:=j^i(\kappa)$, $\lambda_0=\kappa$, and $\lambda=\lambda_n$. Furthermore, $M=\{j(h)(j"\lambda)\mid h:[\lambda]^{\lambda_{n-1}}\to V, h\in V\}$.
\end{lemma}

%Let $X:=\{j(h)(j"\lambda)\mid h:[\lambda]^{\lambda_{n-1}}\to V, h\in V\}$. It follows that $X\elemsub M$ and we let $\pi:X\to M_0$ be the Mostowski collapse of $X$. Define $k:=\pi^{-1}$ and let $j_0:=k\circ j:V\to M_0$. 

%\begin{lemma}\label{lemmanhugeultrapower}
%A cardinal $\kappa>\omega$ is $n$-huge with target $\lambda$ if and only if there are cardinals $\kappa=\lambda_0<\lambda_1<\ldots<\lambda_n=\lambda$ and there is a $\kappa$-complete normal ultrafilter $U$ on $(\lambda)^{\lambda_{n-1}}$ such that for each $i<n$, $\{x\in (\lambda)^{\lambda_{n-1}}\mid \ot(x\cap\lambda_{i+1})=\lambda_i\}\in U$.
%\end{lemma}
%\noindent For a proof of Lemma \ref{lemmanhugeultrapower} see \cite[Lemma 24.8]{Kanamori:Book}. Of course, in Lemma \ref{lemmanhugeultrapower} for $i\in[1,n]$ we have $\lambda_i=j^i(\kappa)$ where $j:V\to M$ witnesses the $n$-hugeness of $\kappa$. 
Note that, using the notation above, for each $i\in[1,n]$, by elementarity $\lambda_i$ is measurable in $M$. Furthermore, since $M$ and $V$ have the same powerset of $\lambda_i$ it follows that $\lambda_i$ is measurable in $V$. This will simplify various calculations performed in $V$ below such as $\lambda_n^{\lambda_i}=\lambda_n$.

Now I will show that the $n$-hugeness of a cardinal is preserved by various forcing notions.

\begin{lemma}\label{lemmanhugedistributive}
If $\kappa$ is $n$-huge with target $\lambda$ then this is preserved by ${\leq}\lambda$-distributive forcing.
\end{lemma}

\begin{proof}
Suppose $\P$ is $\leq \lambda$-distributive and $G$ is $V$-generic for $\P$. Let $j:V\to M=\Ult(V,U)$ where $U$ is a normal fine measure on $[\lambda]^{\lambda_{n-1}}$ as in Lemma \ref{lemmanhugeultrapower}. Since each open dense subset of $j(\P)$ is of the form $j(h)(j"\lambda)$ where $h$ is a function from $[\lambda]^{\lambda_{n-1}}$ to open dense subsets of $\P$, one may show that $H:=\{ p\in j(\P)\mid \exists q\in G\ j(q)\leq p\}$ is $M$-generic for $j(\P)$ by intersecting $|[\lambda]^{\lambda_{n-1}}|=\lambda^{\lambda_{n-1}}=\lambda$ open dense sets. Then since $j"G\subseteq H$, Lemma \ref{lemmaliftingcriterion} implies that $j$ lifts to $j:V[G]\to M[j(G)]$ in $V[G]$ where $j(G):=H$. Since $\P$ is $\leq\lambda$-distributive it follows that $M[j(G)]$ is closed under $\lambda$-sequences of ordinals in $V[G]$ and hence $M[j(G)]^\lambda\subseteq M[j(G)]$ in $V[G]$. Thus $\kappa$ is huge in $V[G]$. 
\end{proof}

The next theorem (Theorem \ref{theoremnhugegch}) is part of the folklore; I include a proof as the proof of Theorem \ref{theoremnhugelevinski} builds naturally upon the proof of Theorem \ref{theoremnhugegch}.

\begin{theorem}\label{theoremnhugegch}
If $\kappa$ is $n$-huge with target $\lambda$ then there is a forcing extension in which $\GCH$ holds and $\kappa$ remains $n$-huge with target $\lambda$.
\end{theorem}

\begin{proof}
Let $\P$ be the canonical forcing of the $\GCH$ and suppose $j:V\to M$ witnesses that $\kappa$ is $n$-huge with target $\lambda$. Furthermore, assume that $M=\Ult(V,U)$ where $U$ is as in Lemma \ref{lemmanhugeultrapower}. It follows that
$$\P\cong \P_\kappa * \dot\P_{[\lambda_0,\lambda_1)} * \dot\P_{[\lambda_1,\lambda_2)} * \cdots *\dot\P_{[\lambda_{n-1},\lambda_n)} * \dot\P_{tail}$$
where 
\begin{enumerate}
\item[(1)] $\lambda_0=\kappa$, $\lambda_n=\lambda$, and $\lambda_i=j^i(\kappa)$,
\item[(2)] $\P_\kappa$ denotes the iteration up to $\kappa$, 
\item[(3)] for each $i\in[0,n)$ the symbol $\dot\P_{[\lambda_i,\lambda_{i+1})}$ is a $\P_{\lambda_i}$-term for the iteration over the interval $[\lambda_i,\lambda_{i+1})$, and 
\item[(4)] $\dot\P_{tail}$ is a $\P_\lambda$-term for $\leq\lambda$-distributive forcing in $V^{\P_\lambda}$. 
\end{enumerate}

Let $G_\kappa*G_{[\lambda_0,\lambda_1)}$ be $V$-generic for $\P_\kappa*\dot\P_{[\lambda_0,\lambda_1)}$. Then $j(\P_\kappa)\cong \P_\kappa * \dot\P_{[\lambda_0,\lambda_1)}$ and since conditions in $\P_\kappa$ have support of size less than $\kappa$ it follows that $j"G_\kappa\subseteq G_\kappa*G_{[\lambda_0,\lambda_1)}$. Thus the embedding lifts to $j:V[G_\kappa]\to M[G_{\lambda_1}]$ in $V[G_{\lambda_1}]$ where $j(G_\kappa)=G_{\lambda_1}:=G_\kappa*G_{[\lambda_0,\lambda_1)}$. Since $\P_{\lambda_1}$ is $\lambda^+$-c.c. it follows that $M[G_{\lambda_1}]^\lambda\subseteq M[G_{\lambda_1}]$ in $V[G_{\lambda_1}]$. 

By elementarity, $j(\P_{[\lambda_0,\lambda_1)})=\P_{[\lambda_1,\lambda_2)}$. Clearly $j(\P_{[\lambda_0,\lambda_1)})$ is $\leq \lambda_1$-directed closed. Since $j"G_{[\lambda_0,\lambda_1)}$ is a directed subset of $j(\P_{[\lambda_0,\lambda_1)})$ of size $\lambda_1$ it follows that there is a master condition $p_1\in j(\P_{[\lambda_0,\lambda_1)})=\P_{[\lambda_1,\lambda_2)}$ which extends every condition in $j"G_{[\lambda_0,\lambda_1)}$. Force with $\P_{[\lambda_1,\lambda_2)}$ below this master condition to obtain $G_{[\lambda_1,\lambda_2)}$ which is $V[G_{\lambda_1}]$-generic for $\P_{[\lambda_1,\lambda_2)}$ and has $p_1\in G_{[\lambda_1,\lambda_2)}$. It follows that $j"G_{[\lambda_0,\lambda_1)}\subseteq G_{[\lambda_1,\lambda_2)}$ and thus the embedding lifts to 
$$j:V[G_\kappa][G_{[\lambda_0,\lambda_1)}]\to M[G_\kappa*G_{[\lambda_0,\lambda_1)}][G_{[\lambda_1,\lambda_2)}]$$
in $V[G_\kappa][G_{[\lambda_0,\lambda_1)}][G_{[\lambda_1,\lambda_2)}]$ where $j(G_\kappa)=G_\kappa*G_{[\lambda_0,\lambda_1)}$ and $j(G_{[\lambda_0,\lambda_1)})=G_{[\lambda_1,\lambda_2)}$. As before it follows that $M[G_{\lambda_2}]^\lambda\subseteq M[G_{\lambda_2}]$ in $V[G_{\lambda_2}]$.

Continuing in this way one may lift the embedding to $j:V[G_{\lambda_{n-1}}]\to M[G_{\lambda_n}]$ in $V[G_{\lambda_n}]$ where $M[G_{\lambda_n}]^\lambda\subseteq M[G_{\lambda_n}]$ in $V[G_{\lambda_n}]$. Now I must lift the embedding further through $\P_{[\lambda_{n-1},\lambda_n)}$. Since $j(\P_{[\lambda_{n-1},\lambda_n)})$ is $\leq\lambda$-directed closed in $M[G_{\lambda_n}]$ it follows that there is a master condition $p_n\in j(\P_{[\lambda_{n-1},\lambda_n)})$ which is below every condition in $j"G_{[\lambda_{n-1},\lambda_n)}$. Let $H$ be $V[G_{\lambda_n}]$-generic for $j(\P_{[\lambda_{n-1},\lambda_n)})$ with $p_n\in H$.  Since the master condition $p_n$ is in $H$ it follows that $j"G_{[\lambda_{n-1},\lambda_n)}\subseteq H$ and thus one may lift the embedding to $j:V[G_{\lambda_n}]\to M[j(G_{\lambda_n})]$ in $V[G_{\lambda_n}][H]$ where $j(G_{\lambda_n})=G_{\lambda_n}*H$. Since $j(\P_{[\lambda_{n-1},\lambda_n)})$ is $\leq\lambda$-closed over $M[G_{\lambda_n}]$ and hence also over $V[G_{\lambda_n}]$ it follows that $\GCH$ holds below $\lambda_n=\lambda$ in $V[G_{\lambda_n}][H]$ and that $M[G_{\lambda_n}*H]^\lambda\subseteq M[G_{\lambda_n}*H]$ in $V[G_{\lambda_n}][H]$. Furthermore, $\GCH$ holds at $\lambda$ in $V[G_{\lambda_n}][H]$ since $\Add(\lambda^+,1)^{M[G_{\lambda_n}]}=\Add(\lambda^+,1)^{V[G_{\lambda_n}]}$. 
Let $j(H)$ be the filter generated by $j"H$ and note that by intersecting open dense sets one may see that $j(H)$ is $M[j(G_{\lambda_n})]$-generic for $j(j(\P_{[\lambda_{n-1},\lambda_n)}))$. Thus one may lift the embedding to $j:V[G_{\lambda_n}][H]\to M[j(G_{\lambda_n})][j(H)]$ in $V[G_{\lambda_n}][H]$ where $M[j(\lambda_n)][j(H)]$ is closed under $\lambda$-sequences in $V[G_{\lambda_n}][H]$. This shows that $\kappa$ remains $n$-huge with target $\lambda$ in $V[G_{\lambda_n}][H]$ where $\GCH$ holds up to and including at $\lambda$. One may now use $\leq\lambda$-distributive forcing over $V[G_{\lambda_n}][H]$ to obtain a model in which $\GCH$ holds everywhere. By Lemma \ref{lemmanhugedistributive} this produces a forcing extension in which $\GCH$ holds and $\kappa$ is $n$-huge with target $\lambda$.
\end{proof}

Next I will show that the target of an $n$-hugeness embedding for $\kappa$ can be forced to have the Levinski property while preserving the $n$-hugeness of $\kappa$.

\begin{theorem}\label{theoremnhugelevinski}
Suppose $\kappa$ is $n$-huge witnessed by $j:V\to M$ and $\GCH$ holds. Then there is a cofinality-preserving forcing extension in which $j^n(\kappa)$ has the strict Levinski property and to which the embedding $j$ lifts, witnessing the $n$-hugeness of $\kappa$ in the extension.

%\marginpar{\tiny This implies that it is consistent that $\kappa$ is $n$-huge and $j^n(\kappa)$ is not $\Sigma_2$-reflecting.}
\end{theorem}

\begin{proof}

Without loss of generality, one may assume that $\GCH$ holds by using Theorem \ref{theoremnhugegch}. Let $\P_\lambda$ be the Levinski iteration up to $\lambda$ and let $j:V\to M=\Ult(V,U)$ witness that $\kappa$ is $n$-huge with target $\lambda=j^n(\kappa)$. Let $\kappa=\lambda_0<\lambda_1<\cdots<\lambda_n=\lambda$ be as in Lemma \ref{lemmanhugeultrapower}, that is let $\lambda_i=j^i(\kappa)$ for $0\leq i\leq n$. The Levinski iteration factors as
$$\P_\lambda\cong\P_\kappa*\dot\P_{[\lambda_0,\lambda_1)}*\cdots*\dot\P_{[\lambda_{n-1},\lambda_n)}.$$ 
Let $G_\kappa*G_{[\lambda_0,\lambda_1)}$ be $V$-generic for $\P_\kappa*\dot\P_{[\lambda_0,\lambda_1)}$. Then $j(\P_\kappa)\cong \P_{\lambda_1}\cong \P_\kappa*\dot\P_{[\lambda_0,\lambda_1)}$. Since conditions in $\P_\kappa$ support bounded below the critical point of $j$, it follows that $j"G_\kappa\subseteq G_\kappa * G_{[\lambda_0,\lambda_1)}$ and hence the embedding lifts to $j:V[G_\kappa]\to M[j(G_\kappa)]$ in $V[G_{\lambda_1}]$ where $j(G_\kappa)=G_\kappa*G_{[\lambda_0,\lambda_1)}$ and $M[j(G_\kappa)]^\lambda\subseteq M[j(G_\kappa)]$ in $V[G_{\lambda_1}]$. 

Now $j"G_{[\lambda_0,\lambda_1)}\subseteq j(\P_{[\lambda_0,\lambda_1)})=\P_{[\lambda_1,\lambda_2)}$. Since $j"G_{[\lambda_0,\lambda_1)}$ has size $\lambda_1$ in $M[j(G_\kappa)]$ and $\P_{[\lambda_1,\lambda_2)}$ is only $<\lambda_1$-directed closed in $M[j(G_\kappa)]$ one may factor $\P_{[\lambda_0,\lambda_1)}\cong\Q_{[\lambda_0,\bar{\lambda}_0)}*\dot\P_{[\bar{\lambda}_0,\lambda_1)}$ where $\bar{\lambda}_0$ is the least inaccessible cardinal greater than $\lambda_0$ and $\Q_{[\lambda_0,\bar{\lambda}_0)}=\prod_{\gamma\in[\lambda_0,\bar{\lambda}_0)\cap\REG}\Add(\gamma,\gamma^{++})$ where the product has Easton support. Let me write $G_{[\lambda_0,\lambda_1)}=H_{[\lambda_0,\bar{\lambda}_0)}*G_{[\bar{\lambda}_0,\lambda_1)}$ where $H_{[\lambda_0,\bar{\lambda}_0)}$ is $V[G_\kappa]$-generic for $\Q_{[\lambda_0,\bar{\lambda}_0)}$ and $G_{[\bar{\lambda}_0,\lambda_1)}$ is $V[G_\kappa][H_{[\lambda_0,\bar{\lambda}_0)}]$-generic for $\P_{[\bar{\lambda}_0,\lambda_1)}$. Now $j"H_{[\lambda_0,\bar{\lambda}_0)}\subseteq j(\Q_{[\lambda_0,\bar{\lambda}_0)})=\Q_{[\lambda_1,\bar{\lambda}_1)}$ where $j"H_{[\lambda_0,\bar{\lambda_0})}$ has size $\bar{\lambda}_0$ and $\Q_{[\lambda_1,\bar{\lambda}_1)}$ is $<\lambda_1$-directed closed. Since $\lambda_1$ is measurable it follows that $\bar{\lambda}_0<\lambda_1$ and hence there is a master condition $p_1$ in $\Q_{[\lambda_1,\bar{\lambda}_1)}$ below every element of $j"H_{[\lambda_0,\bar{\lambda}_0)}$. Force below the master condition $p_1$ to obtain $H_{[\lambda_1,\bar{\lambda}_1)}$, a $V[G_{\lambda_1}]$-generic for $\Q_{[\lambda_1,\bar{\lambda}_1)}$. Now one has $j"G_{[\bar{\lambda}_0,\lambda_1)}\subseteq j(\P_{[\bar{\lambda}_0,\lambda_1)})=\P_{[\bar{\lambda}_1,\lambda_2)}$ and since $j"G_{[\bar{\lambda}_0,\lambda_1)}$ has size $\lambda_1$ and $\P_{[\bar{\lambda}_1,\lambda_2)}$ is $<\bar{\lambda}_1$-directed closed there is a master condition $p_1'$ in $\P_{[\bar{\lambda}_1,\lambda_2)}$ below every element of $j"G_{[\bar{\lambda}_0,\lambda_1)}$. Now force over $V[G_{\lambda_1}][H_{[\lambda_1,\bar{\lambda}_1)}]$ below the master condition $p_1'$ to obtain a $V[G_{\lambda_1}][H_{[\lambda_1,\bar{\lambda}_1)}]$-generic $G_{[\bar{\lambda}_1,\lambda_2)}$ for $\P_{[\bar{\lambda}_1,\lambda_2)}$ with $p_1'\in G_{[\bar{\lambda}_1,\lambda_2)}$. Since $p_1*p_1'\in H_{[\lambda_1,\bar{\lambda}_1)}*G_{[\bar{\lambda}_1,\lambda_2)}$ it follows that $j"(H_{[\lambda_0,\bar{\lambda}_0)}*G_{[\bar{\lambda}_0,\lambda_1)})\subseteq H_{[\lambda_1,\bar{\lambda}_1)}*G_{[\bar{\lambda}_1,\lambda_2)}$ and hence the embedding lifts to 
$$j:V[G_\kappa][G_{[\lambda_0,\lambda_1)}]\to M[j(G_\kappa)][j(G_{[\lambda_0,\lambda_1)})]$$
in $V[G_\kappa][G_{[\lambda_0,\lambda_1)}][G_{[\lambda_1,\lambda_2)}]$ where $j(G_{[\lambda_0,\lambda_1)})=G_{[\lambda_1,\lambda_2)}=H_{[\lambda_1,\bar{\lambda}_1)}*G_{[\bar{\lambda}_1,\lambda_2)}$. Furthermore, $M[j(G_\kappa)][j(G_{[\lambda_0,\lambda_1)})]$ is closed under $\lambda$-sequences in the model $V[G_\kappa][G_{[\lambda_0,\lambda_1)}][G_{[\lambda_1,\lambda_2)}]$ since the forcing $G_{[\lambda_1,\lambda_2)}$ is $\lambda^+$-c.c.

Continuing in this way, the embedding lifts to
$$j:V[G_{\lambda_{n-1}}]\to M[j(G_{\lambda_{n-1}})]$$
in $V[G_{\lambda_{n-1}}][G_{[\lambda_{n-1},\lambda_n)}]$ where $j(G_{\lambda_{n-1}})=G_\kappa*G_{[\lambda_0,\lambda_1)}*\cdots*G_{[\lambda_{n-1},\lambda_n)}$ and where $M[j(G_{\lambda_{n-1}})]$ is closed under $\lambda$-sequences in $V[G_{\lambda_n}]$.

Now $j(\P_{[\lambda_{n-1},\lambda_n)})\cong\widetilde{\P}_{[\lambda,j(\lambda))}$ where $\widetilde{\P}_{[\lambda,j(\lambda))}$ is the Levinski iteration, as defined in $M[j(G_{\lambda_{n-1}})]=M[G_\lambda]$, over the interval $[\lambda,j(\lambda))$. Hence 
$$j(\P_{[\lambda_{n-1},\lambda_n)})\cong (\widetilde{\Q}_\lambda\times\widetilde{\Q}_{[\lambda^+,\bar{\lambda}^M)})*\dot{\widetilde{\P}}_{[\bar{\lambda}^M,j(\lambda))}$$
where $\widetilde{\Q}_\lambda=\Add(\lambda,\lambda^{++})^{M[G_\lambda]}$, $\widetilde{\Q}_{[\lambda^+,\bar{\lambda}^M)}$, is an Easton support product that is $\leq\lambda$-directed closed in $M[G_\lambda]$, and $\widetilde{\P}_{[\bar{\lambda}^M,j(\lambda))}$ is the tail of the iteration in $M[G_\lambda]$. Since $j"H_{[\lambda_{n-1}^+,\bar{\lambda}_{n-1})}$ is a directed subset of $\widetilde{\Q}_{[\lambda^+,\bar{\lambda}^M)}$ in $M[G_\lambda]$ and has size $\lambda$ it follows that there is a master condition $p\in \widetilde{\Q}_{[\lambda^+,\bar{\lambda}^M)}$ below every element of $j"H_{[\lambda_{n-1}^+,\bar{\lambda}_{n-1})}$. Since $\widetilde{\Q}_{[\lambda^+,\bar{\lambda})}$ has at most $\lambda^+$-dense subsets in $M[G_\lambda]$ as counted in $V[G_\lambda]$ it follows that one may diagonalize to build an $M[G_\lambda]$-generic $\widetilde{H}_{[\lambda^+,\bar{\lambda}^M)}$ for $\widetilde{\Q}_{[\lambda^+,\bar{\lambda}^M)}$ in $V[G_\lambda]$ below the master condition $p$. It follows that $j"H_{\lambda_{n-1}}$ is a directed subset of $\widetilde{\Q}_\lambda$ of size at most $\lambda_{n-1}^{++}$ in $M[G_\lambda]$. Thus, since $\widetilde{\Q}_\lambda$ is $<\lambda$-directed closed, there is a master condition $p'\in \widetilde{\Q}_\lambda$ below every condition in $j"H_{\lambda_{n-1}}$. Since the forcing $\widetilde{\Q}_\lambda$ is merely ${<}\lambda$-closed in $M[G_\lambda]$ one cannot build a generic for it. So force with $\widetilde{\Q}_\lambda$ below $p'$ to obtain a $V[G_\lambda]$-generic $\widetilde{H}_\lambda$ containing $p'$. Since $\widetilde{H}_{[\lambda^+,\bar{\lambda}^M)}$ was built in $V[G_\lambda]$ it follows that $\widetilde{H}_\lambda$ is $V[G_\lambda][\widetilde{H}_{[\lambda^+,\bar{\lambda}^M)}]$-generic for $\widetilde{\Q}_\lambda$ and is hence $M[G_\lambda][\widetilde{H}_{[\lambda^+,\bar{\lambda}^M)}]$-generic as well. By the product forcing lemma $\widetilde{H}_\lambda\times\widetilde{H}_{[\lambda^+,\bar{\lambda}^M)}$ is $M[G_\lambda]$-generic for $\widetilde{\Q}_\lambda\times\widetilde{\Q}_{[\lambda^+,\bar{\lambda}^M)}$. Since $(p,p')\in \widetilde{H}_\lambda\times\widetilde{H}_{[\lambda^+,\bar{\lambda}^M)}$ it follows that $j"H_{[\lambda_{n-1},\bar{\lambda}_{n-1})}\subseteq \widetilde{H}_\lambda\times\widetilde{H}_{[\lambda^+,\bar{\lambda}^M)}$ and thus the embedding lifts to 
$$j:V[G_{\lambda_{n-1}}][H_{[\lambda_{n-1},\bar{\lambda}_{n-1})}]\to M[j(G_{\lambda_{n-1}})][j(H_{[\lambda_{n-1},\bar{\lambda}_{n-1})})]$$
in $V[G_\lambda][\widetilde{H}_\lambda]$ where $j(H_{[\lambda_{n-1},\bar{\lambda}_{n-1})})=\widetilde{H}_\lambda\times\widetilde{H}_{[\lambda^+,\bar{\lambda}^M)}$. Furthermore since $\widetilde{\Q}_\lambda\cong\Add(\lambda,\lambda^+)^{V[G_\lambda]}$ it follows that $\widetilde{\Q}_\lambda$ is $\lambda^+$-c.c. in $V[G_\lambda]$ and hence $M[G_\lambda][\widetilde{H}_\lambda\times\widetilde{H}_{[\lambda^+,\bar{\lambda}^M)}]$ is closed under $\lambda$-sequences in $V[G_\lambda][\widetilde{H}_\lambda]$. Since $j"G_{[\bar{\lambda}_{n-1}, \lambda_n)}$ is a directed subset of $\widetilde{\P}_{[\bar{\lambda}^M, j(\lambda))}$ in $M[G_\lambda][\widetilde{H}_\lambda\times\widetilde{H}_{[\lambda^+,\bar{\lambda}^M)}]$ of size $\lambda$ and since $\widetilde{\P}_{[\bar{\lambda}^M, j(\lambda))}$ is $\leq\lambda$-directed closed in $M[G_\lambda][\widetilde{H}_\lambda\times\widetilde{H}_{[\lambda^+,\bar{\lambda}^M)}]$ it follows that there is a master condition $p''\in \widetilde{\P}_{[\bar{\lambda}^M, j(\lambda))}$ below every condition in $j"G_{[\bar{\lambda}_{n-1}, \lambda_n)}$. Now since $M[G_\lambda][\widetilde{H}_\lambda\times\widetilde{H}_{[\lambda^+,\bar{\lambda}^M)}]$ is closed under $\lambda$-sequences in $V[G_\lambda][\widetilde{H}_\lambda]$, one may easily diagonalize to build $\widetilde{G}_{[\bar{\lambda}^M,j(\lambda))}$ an $M[G_\lambda][\widetilde{H}_\lambda\times\widetilde{H}_{[\lambda^+,\bar{\lambda}^M)}]$-generic filter for $\widetilde{\P}_{[\bar{\lambda}^M, j(\lambda))}$ in $V[G_\lambda][\widetilde{H}_\lambda]$ below the master condition $p''$. Thus $j"G_{[\bar{\lambda}_{n-1},\lambda_n)}\subseteq\widetilde{G}_{[\bar{\lambda}^M,j(\lambda))}$ and the embedding lifts to
$$j:V[G_{\lambda_{n-1}}][H_{[\lambda_{n-1},\lambda_n)}]=V[G_\lambda]\to M[j(G_{\lambda_{n-1}})][j(H_{[\lambda_{n-1},\lambda_n)})]$$
in $V[G_\lambda][\widetilde{H}_\lambda]$ where $j(H_{[\lambda_{n-1},\lambda_n)})=(\widetilde{H}_\lambda\times\widetilde{H}_{[\lambda^+,\bar{\lambda}^M)})*\widetilde{G}_{[\bar{\lambda}^M,j(\lambda))}$. Note that $M[j(G_{\lambda_{n-1}})][j(H_{[\lambda_{n-1},\lambda_n)})]$ is closed under $\lambda$ sequences in $V[G_\lambda][\widetilde{H}_\lambda]$.

It remains to lift the embedding through $\widetilde{H}_\lambda$ which is $V[G_\lambda]$-generic for $\widetilde{\Q}_\lambda=\Add(\lambda,\lambda^{++})^{M[G_\lambda]}\cong\Add(\lambda,\lambda^+)^{V[G_\lambda]}$. Let $H_\lambda$ be $V[G_\lambda]$-generic for $\Q_\lambda:=\Add(\lambda,\lambda^+)^{V[G_\lambda]}$. It will suffice to lift the embedding through $\Q_\lambda$ since $V[G_\lambda][\widetilde{H}_\lambda]=V[G_\lambda][H_\lambda]$. This can be done by building a descending sequence of increasingly masterful conditions as in the proofs of Theorems \ref{theoremlevinski} and \ref{theorempartiallysupercompact}. Thus the embedding lifts to $j:V[G_\lambda][H_\lambda]\to M[j(G_\lambda)][j(H_\lambda)]$ and $M[j(G_\lambda)][j(H_\lambda)]$ is closed under $\lambda$-sequences in $V[G_\lambda][H_\lambda]$ since $\Q_\lambda$ is $\lambda^+$-c.c.
\end{proof}

As mentioned above, if $\kappa$ is $n$-huge then $j^n(\kappa)$ is measurable. It follows, as a corollary to Theorem \ref{theoremnhugelevinski}, that if $\kappa$ is $n$-huge then $j^n(\kappa)$ need not exhibit any nontrivial degree of strength or supercompactness.

\subsection{$I_1(\kappa)$ and the Levinski Property}\label{sectioni1}

$I_1(\kappa)$ is the assertion that for some $\lambda$ there is a nontrivial elementary embedding $j:V_{\lambda+1}\to V_{\lambda+1}$ with critical point $\kappa$.  Define $I_1(\kappa,\lambda)$ to be the axiom asserting the existence of such an embedding for a specified $\lambda$. Hamkins shows in \cite[Theorem 5.3]{Hamkins:FragileMeasurability}, that if $I_1(\kappa,\lambda)$ holds witnessed by $j$, and $\P_\lambda$ is an  Easton support iteration such that for each $\gamma<\lambda$,
\begin{enumerate}
\item[(1)] $j(\P_\gamma)=\P_{j(\gamma)}$ and
\item[(2)] $\P_\gamma$ forces $\dot{\Q}_\gamma$ is ${<}\gamma$-directed closed forcing of size at most $2^\gamma$,
\end{enumerate}
then forcing with $\P_\lambda$ preserves $I_1(\kappa,\lambda)$. Thus, as noted in \cite{Hamkins:FragileMeasurability}, the canonical forcing of the $\GCH$ preserves $I_1(\kappa,\lambda)$.

The usual reflection argument shows that if $I_1(\kappa,\lambda)$ holds then $\kappa$ cannot have the Levinski property. In analogy to the previous sections, one could hope to show that if $I_1(\kappa,\lambda)$ holds then there is a forcing extension preserving this in which $\lambda$ has the Levinski property. The Levinski iteration up to $\lambda$ does not satisfy Hamkins's (2) above because at inaccessible stages $\gamma$, the forcing is of the form $\Q_{[\gamma,\bar{\gamma})}$, which has size greater than $2^\gamma$ in $V^{\P_\gamma}$. Nonetheless, I will now show that, the methods of \cite{Hamkins:FragileMeasurability} establish the following.

\begin{theorem}\label{theoremi1}
Suppose $I_1(\kappa,\lambda)$ holds witnessed by $j:V_{\lambda+1}\to V_{\lambda+1}$. Then there is a forcing extension in which $I_1(\kappa,\lambda)$ holds witnessed by a lift of $j$, and in which $\lambda$ has the strict Levinski property. 
\end{theorem}

\begin{proof}
Let $\P_\lambda$ be the Levinski iteration up to $\lambda$. Then  
$$\P_\lambda\cong\P_{\kappa_0}*\P_{[\kappa_0,\kappa_1)}*\cdots*\P_{[\kappa_n,\kappa_{n+1})}*\cdots$$ where $\lambda=\sup_{n<\omega}\kappa_n$ is singular and thus conditions in $\P_\lambda$ may have support unbounded in $\lambda$. Let $\1$ denote the trivial condition in $\P_\lambda$. 

%\marginpar{\tiny Doesn't this mean that $\P_\lambda$ may not even be $\lambda^+$-c.c. It has size $2^\lambda=\lambda^+$ and is $\lambda^{++}$-c.c. So it might collapse cardinals above $\lambda$?}

First I will show that $\P_\lambda$ admits a master condition for $j$; in other words, there is a $q\in \P_\lambda$ such that $q\forces p\in\dot{G}\rightarrow j(p)\in\dot{G}$. Define $q$ inductively as follows. Let $q\restrict \kappa_1$ be the trivial condition in $\P_{\kappa_1}$. Assuming that $q\restrict \kappa_n$ has beed defined, we define $q\restrict\kappa_{n+1}$ as follows. Notice that for each $n<\omega$, 
\begin{enumerate}
\item[(1)] $\forces_{\kappa_n}$ $\dot{\Q}_{\kappa_n}=\dot{\R}_{[\kappa_n,\bar{\kappa}_n)}$ is $<\kappa_n$-directed closed,
\item[(2)] $\forces_{\kappa_n}$ $|(j"\dot{G})(\kappa_n)|=|\dot{G}(\kappa_{n-1})|=2^{\bar{\kappa}_{n-1}}<\kappa_n$, and
\item[(3)] $\forces_{\kappa_n}$ $j"\dot{G}(\kappa_n)\subseteq\dot{\Q}_{\kappa_n}$ is directed.
\end{enumerate}
Conditions (1) - (3) imply that there is a $\P_{\kappa_n}$-name $a$ such that 
$$\forces_{\kappa_n}\textrm{$a\in\dot{\Q}_{\kappa_n}$ and $\forall p\in\dot{G}(a\leq j(p)(\kappa_n))$.}$$ 
Now define $q(\kappa_n):=a$ and let $q(\alpha)$ be the trivial condition for $\alpha\in(\kappa_n,\bar{\kappa}_n)$. At stage $\bar{\kappa}_n$ one has
\begin{enumerate}
\item[(4)] $\forces_{\bar{\kappa}_n}$ $\dot{\P}_{[\bar{\kappa}_n,\kappa_{n+1})}$ is $\leq\kappa_n$-directed closed,
\item[(5)] $\forces_{\bar{\kappa}_n}$ $|(j"\dot{G})\restrict(\bar{\kappa}_n,\kappa_{n+1})|=|\dot{G}\restrict(\bar{\kappa}_{n-1},\kappa_n)|\leq\kappa_n$, and
\item[(6)] $\forces_{\bar{\kappa}_n}$ $(j"\dot{G})\restrict(\bar{\kappa}_n,\kappa_{n+1})\subseteq\P_{[\bar{\kappa}_n,\kappa_{n+1})}$ is directed.
\end{enumerate}
Hence by conditions (4) - (6) there is a $\P_{\bar{\kappa}_n}$-name, call it $b$, such that 
$$\forces_{\bar{\kappa}_n}\textrm{$b\in\P_{[\bar{\kappa}_n,\kappa_{n+1})}$ and $\forall p\in\dot{G}\ (b\leq j(p)\restrict [\bar{\kappa}_n,\kappa_{n+1}))$}$$
Now define $q\restrict[\bar{\kappa}_n,\kappa_{n+1}):=b$ and $q\restrict\kappa_{n+1}:= q\restrict\kappa_n\concat q\restrict[\kappa_n,\kappa_{n+1})$. Now define $q=q\restrict\lambda=\bigcup_{n<\omega}q\restrict\kappa_n$. Since an inverse limit is taken at $\lambda$ (because $\lambda$ is singular) the support of $q$ matches that of $\P_\lambda$. 

I will now show that $q$ is the desired master condition; i.e., $q\forces$ $p\in G\rightarrow j(p)\in G$. Suppose $G$ is $V$-generic for $\P_\lambda$ with $q\in G$. If $p\in G$ then by construction, for each $n\geq 1$ it follows that $q\restrict[\kappa_n,\kappa_{n+1}) \leq j(p)\restrict [\kappa_n,\kappa_{n+1})$ in $V^{\P_{\kappa_n}}$. Hence $q\restrict[\kappa_1,\lambda)\leq j(p)\restrict [\kappa_1,\lambda)$ in $V^{\P_{\kappa_1}}$. Since $q\restrict [\kappa,\kappa_1)$ and $j(p)\restrict [\kappa,\kappa_1)$ are both trivial, it follows that $q\restrict [\kappa,\lambda)\leq j(p)\restrict[\kappa,\lambda)$ in $V^{\P_\kappa}$. It follows that $q\leq (\1\restrict\kappa) \concat q\restrict[\kappa,\lambda)\leq (\1\restrict\kappa)\concat j(p)\restrict[\kappa,\lambda)$ and thus $(\1\restrict\kappa)\concat j(p)\restrict[\kappa,\lambda)\in G$. Since $p\in G$ it follows that $(p\restrict\kappa)\concat \1\restrict[\kappa,\lambda)\in G$. Hence $(p\restrict\kappa)\concat \1\restrict[\kappa,\lambda)$ and $(\1\restrict\kappa)\concat j(p)\restrict[\kappa,\lambda)$ have a common extension, call it $r$, with $r\in G$. Since $r\leq j(p) = (p\restrict\kappa)\concat j(p)\restrict[\kappa,\lambda)$ it follows that $j(p)\in G$.

%One also has that $1\restrict\kappa \concat j(p)\restrict[\kappa,\lambda)$

%Since $q\restrict\kappa$ is trivial, $j(p)\restrict\kappa= p\restrict\kappa \leq q\restrict\kappa$. Thus... \marginpar{\tiny Finish this!}

The following lemma will be required to show that $q\forces$ $j$ lifts to $V[G]_{\lambda+1}$.
\begin{lemma}\label{lemmai1}
If $G\subseteq \P_\lambda$ is $V$-generic, then $V_{\lambda+1}[\bigcup_{n<\omega}G_n]=V[G]_{\lambda+1}$ where $G_n:=G_{\kappa_n}$.
\end{lemma}

\begin{proof}
First I prove $V_\lambda[\bigcup_{n<\omega}G_n]=V[G]_{\lambda}$. $\subseteq$ is immediate. Suppose $x=\tau_G\in V[G]_\lambda$. It follows that $\tau_G\in V[G]_{\kappa_n}$ for some $n$. The closure properties of the iteration $\P_\lambda$ imply that $\tau_G$ could not have been added by the tail $\P_{[\kappa_{n+1},\lambda)}$. Let $\sigma$ be a $\P_{\kappa_{n+1}}$-name with $x=\sigma_{G_{n+1}}=\tau_{G}$. We can assume that the rank of $\sigma$ is less than $\kappa_{n+2}$. Thus $x=\sigma_{G_{n+1}}=\tau_{G}\in V_\lambda[\bigcup_{n<\omega}G_n]$.

Now I will show that $V_{\lambda+1}[\bigcup_{n<\omega}G_n]=V[G]_{\lambda+1}$. $\subseteq$ is immediate. Suppose $x=\tau_G\in V[G]_{\lambda+1}$. Then $x\subseteq V[G]_\lambda$ and $x=\tau_G=\bigcup_{n<\omega}\tau_G\cap V[G]_{\kappa_n}$. From the previous paragraph, each piece of the union has a name in $V_\lambda$, call it $\sigma_n$, which can be evaluated using only an initial segment of $G$. It follows that the sequence $\langle \sigma_n\mid n<\omega\rangle$ is in $V_{\lambda+1}$. Hence $\bigcup_{n<\omega}\sigma_n\in V_{\lambda+1}$. It follows that $x=\tau_G\in V_{\lambda+1}[\bigcup_{n<\omega}G_n]$.
\end{proof}

Now I argue that $q\forces$ $j$ lifts to $V[G]_{\lambda+1}$. Suppose $q\in G$ where $G$ is $V$-generic for $\P_\lambda$. Define $j:V[G]_{\lambda+1}\to V[G]_{\lambda+1}$ by $j(\tau_G)=j(\tau)_G$. By Lemma \ref{lemmai1} one can assume that $\tau\in V_{\lambda+1}$ and can be evaluated using $\bigcup_{n<\omega} G_n$, thus $\tau\in\dom(j)$. Let me show that $j$ is well defined. If $\sigma_G=\tau_G$ then $p\forces\sigma=\tau$ for some $p\in G$. Since $V_{\lambda+1}$ can verify the statement $p\forces\sigma=\tau$ it follows by applying $j$ that $j(p)\forces j(\sigma)=j(\tau)$. Now $j(p)\in G$ since $q\in G$ serves as a master condition. Thus $j(\sigma)_G=j(\tau)_G$. Similarly, $j$ is elementary. The lift of $j$ witnesses $I_1(\kappa,\lambda)$ holds in $V[G]$.
\end{proof}

I end with a question.

\begin{question}
Is there some large cardinal notion, say $\varphi(\kappa)$, among those commonly considered, such that the consistency strength of ``\,$\ZFC + \exists\kappa(\varphi(\kappa)\wedge\textrm{$\kappa$ has the Levinski property})$'' is strictly greater than the consistency strength of ``\,$\ZFC+\exists\kappa\varphi(\kappa)$''?
\end{question}

%===========================================================
%===========================================================
%===========================================================
%===========================================================
%===========================================================
%===========================================================
%===========================================================
%
%
%  BACKMATTER
%
%===========================================================
%===========================================================%===========================================================
%===========================================================

% If you want an introduction in the table of contents, but
% without its own chapter number, do this:
%\chapter*{Introduction}
%\addcontentsline{toc}{chapter}{\numberline{}Introduction}
%\markboth{Introduction}{INTRODUCTION}

% Also, to ensure that the bibliography is in the table of
% contents, do this when you reach the end of the main text:

\backmatter

% Use this to generate a bib file.
%\bibliography{masterbib}
%\bibliographystyle{alpha}

\addcontentsline{toc}{chapter}{\numberline{}Bibliography }

\end{document}